%% file: qprime.tex
 \documentclass[12pt]{amsart}
\input{header}

\def\sideremark#1{\ifvmode\leavevmode\fi\vadjust{\vbox to0pt{\vss
 \hbox to 0pt{\hskip\hsize\hskip1em
 \vbox{\hsize3cm\tiny\raggedright\pretolerance10000
  \noindent #1\hfill}\hss}\vbox to8pt{\vfil}\vss}}}%

\newcommand{\D}{\mathbb{D}}

\begin{document}

\title[The $P^\prime$-operator and $Q^\prime$-curvature via CR tractors]{The $P^\prime$-operator, the $Q^\prime$-curvature, and the CR tractor calculus}
\author{Jeffrey S.\ Case}
\address{109 McAllister Building \\ Penn State University \\ University Park, PA 16802}
\email{jscase@psu.edu}
\author{A.\ Rod Gover}
\thanks{ARG gratefully acknowledges support from the Royal Society of New Zealand via Marsden Grants 13-UOA-018 and 16-UOA-051 }
\address{Department of Mathematics \\ University of Auckland \\ Private Bag 92019 \\ Auckland 1042, New Zealand}
\email{r.gover@auckland.ac.nz}
\keywords{CR pluriharmonic functions; pseudo-Einstein manifold; CR invariant operator; $P$-prime operator; $Q$-prime curvature}
\subjclass[2010]{Primary 32V05; Secondary 32T15}
\begin{abstract}
 We establish an algorithm which computes formulae for the CR GJMS operators, the $P^\prime$-operator, and the $Q^\prime$-curvature in terms of CR tractors.  When applied to torsion-free pseudo-Einstein contact forms, this algorithm both gives an explicit factorisation of the CR GJMS operators and the $P^\prime$-operator, and shows that the $Q^\prime$-curvature is constant, with the constant explicitly given in terms of the Webster scalar curvature.  We also use our algorithm to derive local formulae for the $P^\prime$-operator and $Q^\prime$-curvature of a five-dimensional pseudo-Einstein manifold.  Comparison with Marugame's formulation of the Burns--Epstein invariant as the integral of a pseudohermitian invariant yields new insights into the class of local pseudohermitian invariants for which the total integral is independent of the choice of pseudo-Einstein contact form.
\end{abstract}
\maketitle

\input{intro}

\section*{Acknowledgments} The authors would like to thank the Centre de Recerca Matem\`atica and Princeton University for their hospitality while portions of this work were being completed.  They would also like to thank Paul Yang for many fruitful discussions about the $P^\prime$-operator and the $Q^\prime$-curvature, and Taiji Marugame for pointing out Proposition~\ref{prop:stein_vanish} and supplying its proof.

\input{bg}
\input{tractor}
\input{algorithm}
\input{product}
\input{dim5}

\bibliographystyle{abbrv}
\bibliography{qprime_bib}
\end{document}

%% file: header.tex
\usepackage{amsmath}
\usepackage{amssymb}
\usepackage{amsthm}

\DeclareMathOperator{\compose}{\circ}

\DeclareMathOperator{\dvol}{dvol}

\DeclareMathOperator{\Ric}{Ric}

\DeclareMathOperator{\hash}{\sharp}

\DeclareMathOperator{\Imaginary}{Im}
\DeclareMathOperator{\Real}{Re}

\newcommand{\omL}{\overline{\mathcal{L}}}

\newcommand{\cg}{\widetilde{g}}

\newcommand{\cf}{\widetilde{f}}

\newcommand{\cT}{\widetilde{T}}
\newcommand{\cV}{\widetilde{V}}

\newcommand{\cDelta}{\widetilde{\Delta}}
\newcommand{\cmC}{\widetilde{\mathcal{C}}}
\newcommand{\cmE}{\widetilde{\mathcal{E}}}
\newcommand{\hmI}{\widehat{\mathcal{I}}}
\newcommand{\hmO}{\widehat{\mathcal{O}}}

\newcommand{\htheta}{\widehat{\theta}}
\newcommand{\hA}{\widehat{A}}
\newcommand{\hI}{\widehat{I}}
\newcommand{\hP}{\widehat{P}}
\newcommand{\hQ}{\widehat{Q}}
\newcommand{\hX}{\widehat{X}}

\newcommand{\lv}{\lvert}
\newcommand{\rv}{\rvert}

\newcommand{\bCP}{\mathbb{C}P}

\newcommand{\dbar}{\overline{\partial}}
\renewcommand{\d}{\partial}

\newcommand{\mA}{\mathcal{A}}

\newcommand{\mC}{\mathcal{C}}
\newcommand{\mD}{\mathcal{D}}
\newcommand{\mE}{\mathcal{E}}
\newcommand{\mF}{\mathcal{F}}

\newcommand{\mI}{\mathcal{I}}

\newcommand{\mK}{\mathcal{K}}
\newcommand{\mL}{\mathcal{L}}

\newcommand{\mN}{\mathcal{N}}
\newcommand{\mO}{\mathcal{O}}
\newcommand{\mP}{\mathcal{P}}
\newcommand{\mQ}{\mathcal{F}}  

\newcommand{\mS}{\mathcal{S}}

\newcommand{\bC}{\mathbb{C}}
\newcommand{\bD}{\mathbb{D}}

\newcommand{\bN}{\mathbb{N}}

\newcommand{\bR}{\mathbb{R}}

\newcommand{\bV}{\mathbb{V}}

\newcommand{\bZ}{\mathbb{Z}}

\newcommand{\MA}{M_{\mathcal{A}}}
\newcommand{\PF}{P_{\mathcal{F}}}
\newcommand{\rr}{r}

\newcommand{\bw}{\mathbf{w}}

\newcommand{\btheta}{\boldsymbol{\theta}}

\makeatletter
    \@ifdefinable{\invariant}{\def\invariant/{secondary invariant}}
\makeatother

\newcommand{\comment}[1]{}

\newtheorem{thm}{Theorem}[section]
\newtheorem{prop}[thm]{Proposition}
\newtheorem{lem}[thm]{Lemma}
\newtheorem{cor}[thm]{Corollary}

\theoremstyle{definition}

\theoremstyle{remark}
\newtheorem{remark}[thm]{Remark}

\numberwithin{equation}{section}

%% file: intro.tex
\section{Introduction}
\label{sec:intro}

An important class of differential operators in CR geometry are the CR
GJMS (or Gover--Graham) operators~\cite{GoverGraham2005}.  The CR GJMS
operator $P_{2k}$ is a formally self-adjoint differential operator
with principal part the $k$-th power $(-\Delta_b)^k$ of the negative
of the sublaplacian; our convention is that $-\Delta_b$ is a positive
operator.  This operator is defined on any pseudohermitian manifold
$(M^{2n+1},H,\theta)$ with $k\leq n+1$ and is CR invariant,
$P_{2k}\colon\mE(-\frac{n+1-k}{2})\to\mE(-\frac{n+1+k}{2})$; i.e.\ if
$\htheta=e^\Upsilon\theta$, then
\[ e^{\frac{n+1+k}{2}\Upsilon}\hP_{2k}\left(f\right) = P_{2k}\left(e^{\frac{n+1-k}{2}\Upsilon}f\right) \]
for all $f\in C^\infty(M)$.  Special cases are the CR Laplacian $P_2$ studied by Jerison and Lee~\cite{JerisonLee1987} and the CR Paneitz operator $P_4$ which are, for example, important in the study of the embedding problem in three dimensions~\cite{ChanilloChiuYang2010}.

The critical CR GJMS operators $P_{2n+2}$ are of particular interest.
The kernel of $P_{2n+2}$ is nontrivial, containing the space $\mP$ of
CR pluriharmonic functions~\cite{CaseYang2012,Hirachi2013}; indeed,
this characterizes the kernel on the standard CR
spheres~\cite{BransonFontanaMorpurgo2007}.  As such, Branson's
argument of analytic continuation in the dimension~\cite{Branson1995}
gives rise to the $P^\prime$-operator .  This operator was first
identified on the sphere by Branson, Fontana and
Morpurgo~\cite{BransonFontanaMorpurgo2007}, then on general
three-dimensional CR manifolds by Yang and the first-named
author~\cite{CaseYang2012}, and then in general dimensions by Hirachi~\cite{Hirachi2013}.  As an operator $P^\prime\colon\mP\to
C^\infty(M)$, the $P^\prime$-operator is not invariant; rather, if
$\htheta=e^\Upsilon\theta$, then
\begin{equation}
 \label{eqn:pprime_trans_intro}
 e^{(n+1)\Upsilon}\hP^\prime(f) = P^\prime(f) + P_{2n+2}(f\Upsilon)
\end{equation}
for all $f\in\mP$.  In this way, one could think of $P^\prime$ as a $Q$-curvature operator (cf.\ \cite{BransonGover2005}).  For geometric applications, it is often preferable to regard the $P^\prime$-operator as a map $P^\prime\colon\mP\to C^\infty(M)/\mP^\perp$.  Both $\mP\subset\mE(0)$ and $\mP^\perp\subset\mE(-n-1)$ are CR invariant spaces, while the self-adjointness of $P_{2n+2}$ and the fact $\mP\subset\ker P_{2n+2}$ combine with~\eqref{eqn:pprime_trans_intro} to imply that $P^\prime\colon\mP\to C^\infty(M)/\mP^\perp$ is CR invariant.  In particular, $P^\prime$ determines a CR invariant pairing $\mP\times\mP\ni(u,v)\mapsto\int u\,P^\prime v$.

The natural extension of Branson's $Q$-curvature to the CR setting is
$Q:=P^\prime(1)$ (cf.\ \cite{FeffermanHirachi2003}).  While the total $Q$-curvature is a CR invariant, it is often trivial: The
total $Q$-curvature of a compact three-dimensional CR manifold is
always zero~\cite{Hirachi1990} and the $Q$-curvature vanishes
identically for any pseudo-Einstein
manifold~\cite{FeffermanHirachi2003}.  The latter fact implies that we
may again use analytic continuation in the dimension to define the
$Q^\prime$-curvature as a pseudohermitian invariant of pseudo-Einstein
manifolds; see~\cite{CaseYang2012} in dimension three
and~\cite{Hirachi2013} in general dimension.  Suppose that $\theta$ is
pseudo-Einstein.  Then $\htheta=e^\Upsilon\theta$ is pseudo-Einstein if
and only if $\Upsilon\in\mP$; see~\cite{Lee1988}.
If $\Upsilon\in\mP$, we find that
\begin{equation}
 \label{eqn:qprime_trans_intro}
 e^{(n+1)\Upsilon}\hQ^\prime = Q^\prime + P^\prime(\Upsilon) + \frac{1}{2}P_{2n+2}(\Upsilon^2) ,
\end{equation}
where we regard $P^\prime$ and $Q^\prime$ as $C^\infty(M)$-valued.  Regarding instead $P^\prime$ and $Q^\prime$ as $C^\infty(M)/\mP^\perp$-valued, we have the transformation rule
\[ e^{(n+1)\Upsilon}\hQ^\prime = Q^\prime + P^\prime(\Upsilon) . \]
It is in this context that the $Q^\prime$-curvature prescription problem seems solvable; see~\cite{CaseHsiaoYang2014} for progress in the three-dimensional setting.

A key property of the $Q^\prime$-curvature is that its total integral over a compact pseudo-Einstein manifold is a \emph{\invariant/}.  Following Hirachi~\cite{Hirachi2013}, by a \invariant/ we mean a pseudohermitian invariant which is not CR invariant, but which is invariant within the distinguished class of pseudo-Einstein contact forms.  That the total $Q^\prime$-curvature is independent of the choice of pseudo-Einstein contact form follows from~\eqref{eqn:qprime_trans_intro} and the self-adjointness of the critical CR GJMS operator and of the $P^\prime$-operator~\cite{CaseYang2012,Hirachi2013}; that it is not independent of the choice of contact form follows from~\cite[Proposition~6.1]{CaseYang2012}.

As a global \invariant/, the total $Q^\prime$-curvature is a biholomorphic invariant of domains in $\bC^{n+1}$ (cf.\ \cite{FeffermanHirachi2003}).  It is interesting to compare it with the Burns--Epstein invariant~\cite{BurnsEpstein1988,BurnsEpstein1990c,Marugame2013}.  In dimension three, the total $Q^\prime$-curvature agrees with the Burns--Epstein invariant up to a universal constant~\cite{CaseYang2012,Hirachi2013}, whereupon one obtains the Gauss--Bonnet formula
\[ \chi(X) = \int_X \left( c_2 - \frac{1}{3}c_1^2\right) + \frac{1}{16\pi^2}\int_M Q^\prime \]
for $X\subset\bC^2$ a bounded strictly pseudoconvex domain with boundary $M^3=\partial X$ and the Chern forms are computed with
respect to a complete K\"ahler--Einstein metric in $X$.  In dimension
five, the total $Q^\prime$-curvature and the Burns--Epstein invariant
do not in general agree; see Theorem~\ref{thm:marugame} and Proposition~\ref{prop:iprime} below for more precise statements.

At present, there are only two ways to study the CR GJMS operators,
the $P^\prime$-operator, and the $Q^\prime$-curvature.  The first is
to restrict to dimension three, where local formulae are known and can be used to address questions involving the
signs of these objects~\cite{CaseYang2012}.  The second is to pass
to the ambient manifold, where the definitions are relatively simple and may be readily
used to prove many formal properties of these objects~\cite{Hirachi2013}.  However, it
is not straightforward to produce local formulae for these operators
from the ambient definition, nor is it known how to use the ambient
definition to address issues such as the sign of the CR GJMS operators
and the $P^\prime$-operator or the value of the total
$Q^\prime$-curvature.

The goal of this article is to rectify some of these issues by giving
a new interpretation of the CR GJMS operators, the
$P^\prime$-operator, and the $Q^\prime$-curvature.  Specifically, we
give an interpretation of these objects in terms of the CR tractor
calculus, building on the work of Graham and the second-named author
on the CR GJMS operators~\cite{GoverGraham2005}.  Our main result is
an algorithm, encoded in Theorem~\ref{thm:general_formula}, which
produces a tractor formula for these operators in terms of tractor $D$-operators and the tractor curvature (cf.\ \cite{GoverPeterson2003}).  As an immediate application, we compute the $Q^\prime$-curvature and obtain factorisations of the CR GJMS operators
and the $P^\prime$-operator on any Einstein pseudohermitian manifold (cf.\ \cite{Gover2006});
i.e.\ on any pseudo-Einstein manifold with vanishing torsion.

\begin{thm}
 \label{thm:factorisation}
 Let $(M^{2n+1},H,\theta)$ be an embeddable Einstein pseudohermitian manifold.  For any integer $1\leq k\leq n+1$, the CR GJMS operator $P_{2k}$ is equal to
 \begin{equation}
  \label{eqn:factorisation_P}
  P_{2k} = \begin{cases}
             \displaystyle\prod_{\ell=1}^{\frac{k}{2}} \left(-\Delta_b+ic_\ell\nabla_0+d_\ell P\right)\left(-\Delta_b-ic_\ell\nabla_0+d_\ell P\right), & \text{if $k$ is even}, \\
             Y\displaystyle\prod_{\ell=1}^{\frac{k-1}{2}} \left(-\Delta_b+ic_\ell\nabla_0+d_\ell P\right)\left(-\Delta_b-ic_\ell\nabla_0+d_\ell P\right), & \text{if $k$ is odd} ,
            \end{cases}
 \end{equation}
 where $c_\ell=k-2\ell+1$ and $d_\ell=\frac{n^2-(k-2\ell+1)^2}{n}$ and $Y=-\Delta_b+nP$ is the CR Yamabe operator.  Moreover, the $P^\prime$-operator is
 \begin{equation}
  \label{eqn:factorisation_Pprime}
  P_{2n+2}^\prime = n!\left(\frac{2}{n}\right)^{n+1}\prod_{\ell=0}^{n} \left(-\Delta_b + 2\ell P\right)
 \end{equation}
 and the $Q^\prime$-curvature is
 \begin{equation}
  \label{eqn:factorisation_Qprime}
  Q_{2n+2}^\prime = (n!)^2\left(\frac{4P}{n}\right)^{n+1} .
 \end{equation}
\end{thm}

An alternative proof of Theorem~\ref{thm:factorisation} has been given by Takeuchi~\cite{Takeuchi2017} using purely ambient techniques.

Since the standard CR sphere and the Heisenberg group with its standard contact form are both Einstein pseudohermitian manifolds, \eqref{eqn:factorisation_P} recovers the known formulae for the CR GJMS operators on these manifolds~\cite{BransonFontanaMorpurgo2007,Graham1984}.  Moreover, \eqref{eqn:factorisation_Pprime} recovers the formula for the $P^\prime$-operator obtained by Branson, Fontana and Morpurgo on the sphere and~\eqref{eqn:factorisation_Qprime} gives a geometric meaning to the constant in their sharp Onofri-type inequality on CR pluriharmonic functions~\cite{BransonFontanaMorpurgo2007}.

Einstein pseudohermitian manifolds are equivalent to $\eta$-Sasaki--Einstein manifolds (cf.\ \cite{Leitner2007,Sparks2011}).  This observation leads to a wealth of examples to which Theorem~\ref{thm:factorisation} applies (cf.\ \cite{BoyerGalicki2008,BoyerGalickiMatzeu2006,Sparks2011}).

Theorem~\ref{thm:factorisation} gives factorisations of the CR GJMS operators in terms of Folland--Stein operators and of the $P^\prime$-operator in terms of the sublaplacian.  In particular, the spectrum of the $P^\prime$-operator is completely understood in terms of the scalar curvature and the spectrum of the sublaplacian of an Einstein pseudohermitian manifold.  Likewise, it determines the total $Q^\prime$-curvature of an Einstein pseudohermitian manifold in terms of its scalar curvature and volume.  As a special case of these observations, we have the following corollary.

\begin{cor}
 \label{cor:positivity}
 Let $(M^{2n+1},H,\theta)$ be a compact embeddable Einstein pseudohermitian manifold with nonnegative CR Yamabe constant.  Then $P^\prime\geq0$, $\ker P^\prime=\bR$, and
 \begin{equation}
  \label{eqn:positivity}
  \int_M Q^\prime \leq \int_{S^{2n+1}} Q_0^\prime,
 \end{equation}
 where the right-hand side denotes the total $Q^\prime$-curvature of the standard CR $(2n+1)$-sphere.  Moreover, equality holds in~\eqref{eqn:positivity} if and only if $(M^{2n+1},H,\theta)$ is CR equivalent to the standard CR sphere.
\end{cor}

In three dimensions, the conclusions of Corollary~\ref{cor:positivity}
are true under weaker hypotheses involving only the CR Paneitz
operator and the CR Yamabe constant~\cite{CaseYang2012}.  It is
natural to ask if similar positivity results extend to higher
dimensions; Corollary~\ref{cor:positivity} suggests that there is
scope for such a result.  Reasons to be interested in such a result
are its characterization of the standard CR sphere and its role in finding metrics of
constant $Q^\prime$-curvature by variational methods
(cf.\ \cite{CaseHsiaoYang2014}).

In five dimensions, it is straightforward to produce from Theorem~\ref{thm:general_formula} an explicit tractor formula for the CR GJMS operators, the $P^\prime$-operator, and the $Q^\prime$-curvature.  In particular, we obtain an explicit local formula for the $Q^\prime$-curvature in this dimension:
\begin{multline}
 \label{eqn:dim5_qprime}
 Q^\prime = 4\Delta_b^2 P + 4\Delta_b\lv A_{\alpha\beta}\rv^2 - 16\Imaginary\nabla^\gamma\left(A_{\beta\gamma}\nabla^\beta P\right) + 16i\nabla^{\gamma}Y_{\gamma} \\ - 16\Delta_b P^2 - 32P\lv A_{\alpha\beta}\rv^2 + 32P^3 - 16A^{\alpha\gamma}Q_{\alpha\gamma} ;
\end{multline}
see Section~\ref{sec:bg} for a description of our notation.  We thus
obtain a formula for the total $Q^\prime$-curvature of a compact
pseudo-Einstein five-manifold (cf.\ \eqref{eqn:dim5_total_qprime} and
\cite{HirachiMarugameMatsumoto2015}).  On the other hand,
Marugame~\cite{Marugame2013} has computed the Burns--Epstein invariant
in this setting.  By comparing these formulae, we obtain the following
Gauss--Bonnet formula for bounded strictly pseudoconvex domains in
$\bC^3$.


\begin{thm}
 \label{thm:marugame}
 Let $X\subset\bC^3$ be a bounded strictly pseudoconvex domain with boundary $M^5=\partial X$.  Let $\rho$ be a defining function for $M$ such that $g=-i\log\dbar\partial\log\rho$ is a complete K\"ahler--Einstein metric in $X$.  Then
 \begin{equation}
  \label{eqn:marugame}
  \chi(X) = \int_X \left(c_3 - \frac{1}{2}c_1c_2 + \frac{1}{8}c_1^3\right) + \frac{1}{\pi^3}\int_M \left( Q^\prime + 16\mI^\prime \right) ,
 \end{equation}
 where $\mI^\prime$ is the pseudohermitian invariant
 \begin{equation}
  \label{eqn:fg6_invariant}
  \mI^\prime = -\frac{1}{8}\Delta_b\left|S_{\alpha\bar\beta\gamma\bar\sigma}\right|^2 + \lv V_{\alpha\bar\beta\gamma}\rv^2 + \frac{1}{2}P\lv S_{\alpha\bar\beta\gamma\bar\delta}\rv^2 .
 \end{equation}
\end{thm}

In~\eqref{eqn:fg6_invariant}, $S_{\alpha\bar\beta\gamma\bar\delta}$
denotes the Chern tensor --- the completely tracefree part of the
pseudohermitian curvature --- and $V_{\alpha\bar\beta\gamma}$ is the
CR analogue of the Cotton tensor; see Section~\ref{sec:bg} for
details.  The pseudohermitian invariant $\mI^\prime$ should be
regarded as the analogue in the critical dimension of the nontrivial
conformal invariant of weight $-6$ discovered by Fefferman and Graham
(cf.\ \cite[(9.3)]{FeffermanGraham2012}).  More precisely, there is a
CR invariant $\mI$ of weight $-3$ and of the form $\lv\nabla_\rho
S_{\alpha\bar\beta\gamma\bar\delta}\rv^2$ plus terms involving
$V_{\alpha\bar\beta\delta}$ (see~\eqref{eqn:mI}) in general dimensions
which is a pure divergence in dimension five.  Arguing by analytic
continuation in the dimension yields, modulo divergences, the
pseudohermitian invariant $\mI^\prime$ on five-dimensional
pseudo-Einstein manifolds; in particular, one expects the total
$\mI^\prime$-curvature to be a global \invariant/.  In
Proposition~\ref{prop:iprime}, we give an intrinsic proof of this fact
provided $c_2(H^{1,0})$ vanishes in $H^4(M;\bR)$.  If $M$ is the boundary of a Stein manifold, then $c_2(H^{1,0})=0$; see Section~\ref{sec:dim5} for details.  Note that Marugame
has already given an extrinsic proof~\cite{Marugame2013} of this fact,
without assuming the vanishing of the second real Chern class.  Our
study of $\mI^\prime$ suggests that the CR analogue of the
Deser--Schwimmer conjecture is more subtle than its conformal
analogue; see Remark~\ref{rk:deser-schwimmer} for further discussion.

We conclude this introduction by outlining the algorithm for producing
tractor formulae for the CR GJMS operators, the $P^\prime$-operator,
and the $Q^\prime$-curvature contained in
Theorem~\ref{thm:general_formula} and how it is applied to obtain
Theorem~\ref{thm:factorisation}.  To that end, we first recall the
definitions of these objects via the ambient
manifold~\cite{Hirachi2013}.

Suppose that $(M^{2n+1},H)$ is a strictly pseudoconvex CR manifold
which is embedded in a complex manifold $X^{n+1}$; note that compact strictly pseudoconvex CR manifolds of dimension at least five are automatically embeddable~\cite{Boutet1975,HarveyLawson1977,Lempert1995}.  Let $\rho\in
C^\infty(X)$ be a defining function for $M$ which is positive on the
pseudoconvex side, and let $\theta=\Imaginary\dbar\rho\rv_{TM}$ be the
induced contact form.  Suppose that, near $M$ in $X$, there is a
$(n+2)$-nd root $\mL_X$ of the canonical bundle $\mK_X$ of $X$.  The
\emph{ambient space of $M$}, which we denote $\MA$,  is the total space of
$\mL_X\setminus\{0\}\to X$, and the restriction
$\MA\rv_M$ is denoted by $\mF$.  Note that $\mF$ is a CR
manifold of (real) dimension $2n+3$ with Levi form which is positive
definite except in the fibre direction, and the pullback of $\rho$ to
$\MA$, also denoted by $\rho$, is a defining function for
$\mF$.


Given $\lambda\in\bC^\ast$, define the \emph{dilation} $\delta_\lambda\colon\MA\to\MA$ by fibre-wise scalar multiplication, $\delta_\lambda(\xi)=\lambda\xi$.  Given $w\in\bR$, denote
\[ \cmE(w) = \left\{ f\in C^\infty(\MA;\bC) \colon \delta_\lambda^\ast f = \lv\lambda\rv^{2w}f \text{ for all $\lambda\in\bC^\ast$} \right\} . \]
A natural choice of defining function $\rr\in\cmE(1)$ for $\mF$ is obtained by Fefferman's construction~\cite{Fefferman1976}: it is the unique defining function modulo $O(\rho^{n+3})$ such that
\begin{equation}
 \label{eqn:ambient_metric_ricci}
 \Ric[\rr] = i\eta \rr^n\d \rr\wedge\dbar \rr + O(\rho^{n+1}) ,
\end{equation}
where $\Ric[\rr]$ is the Ricci curvature of the \emph{ambient metric} $\cg[\rr]=-i\d\dbar \rr$ defined in a neighborhood of $\mF$ in $\MA$ and $\eta\rv_{\mF}$ is a CR invariant, the \emph{obstruction function}.

Let $\mK=\Lambda^{n+1}(H^{0,1})^\perp$ denote the canonical bundle of $M$.  Note that $\mK=\mK_X\rv_M$.  Let $\mL_M=\mL_X\rv_M$, so that $\mL_M$ is a $(n+2)$-nd root of $\mK$.  Given $w,w^\prime\in\bC$ such that $w-w^\prime\in\bZ$, we denote
\[ \mE(w,w^\prime) = \mL_M^{-w}\otimes\omL_M^{-w^\prime} . \]
A \emph{CR density of weight $w\in\bR$} is a smooth section of the bundle $\mE(w)=\mE(w,w)$.  When clear by context, we also use $\mE(w)$ to denote the space of CR densities of weight $w$.  Given a homogeneous function $\cf\in\cmE(w)$ on the ambient space, its restriction to $\mF$ defines a CR density $f=\cf\rv_{\mF}\in\mE(w)$.  We call $\cf$ an \emph{ambient extension of $f$}.  Such functions are unique up to adding terms of the form $\phi \rr$ with $\phi\in\cmE(w-1)$.

Let $k\in\{0,1,\dotsc,n+1\}$ and set $w=-\frac{n+1-k}{2}$.  Given $f\in\mE(w)$, define
\[ P_{2k}f := (-2\cDelta)^k\cf\rv_{\mF} . \]
This definition is independent of the choice of ambient extension $\cf$. In particular, $P_{2k}\colon\mE(w)\to\mE(w-k)$ is a conformally covariant operator, the $k$-th order CR GJMS operator~\cite{GoverGraham2005}.  Our normalization is such that $P_{2k}$ has leading order term $(-\Delta_b)^k$.

Set $h_\theta=\rr/\rho$.  Define the $P$-prime operator $P_{2n+2}^\prime$ on $\mP$ by
\begin{equation}
 \label{eqn:ambient_pprime}
 P_{2n+2}^\prime f = -(-2\cDelta)^{n+1}\left(f\log h_\theta\right)\rv_{\mF} \in \mE(-n-1) .
\end{equation}
This operator depends only on $f$ and the choice of contact form $\theta$.  Moreover, computing with respect to the contact form $\htheta=e^\Upsilon\theta$ yields the transformation formula~\eqref{eqn:pprime_trans_intro}.

Suppose now that $\theta$ is a pseudo-Einstein contact form; equivalently, suppose that $\log h_\theta\rv_{\mF}\in\mP$.  Define the $Q$-prime curvature $Q_{2n+2}^\prime$ by
\begin{equation}
 \label{eqn:ambient_qprime}
 Q_{2n+2}^\prime = \frac{1}{2}(-2\cDelta)^{n+1}\left(\log h_\theta\right)^2\rv_{\mF} \in \mE(-n-1) .
\end{equation}
This scalar depends only on the choice of contact form $\theta$.  Moreover, computing with respect to the pseudo-Einstein contact form $\htheta=e^\Upsilon\theta$ yields the transformation formula~\eqref{eqn:qprime_trans_intro}.

An alternative approach to these definitions can be made through the
CR tractor calculus~\cite{GoverGraham2005}.  Specifically, \v{C}ap and
the second-named author~\cite{CapGover2002,CapGover2008} have provided
a dictionary which effectively equates definitions of CR invariant
objects made via the ambient metric with definitions made via the CR
tractor calculus.  Using this dictionary, we develop in
Section~\ref{sec:algorithm} an algorithm for generating tractor
formulae for the CR GJMS operators, the $P^\prime$-operators, and the
$Q^\prime$-curvatures in general dimension.  This has two benefits.
First, it is easy to execute this algorithm in low dimensions, and
this allows us to derive~\eqref{eqn:dim5_qprime}; see
Section~\ref{sec:dim5} for further discussion.  Second, the algorithm
almost immediately yields Theorem~\ref{thm:factorisation} using the
local correspondence between Einstein contact forms and parallel CR
standard tractors.  More precisely, the algorithm leads to a formula
(cf.\ Theorem~\ref{thm:general_formula}) for the CR GJMS operators in
terms of tractor $D$-operators and the CR Weyl tractor, a tractor
version of the curvature tensor of the ambient metric.  Since
contractions of a parallel CR standard tractor $I_A$ into the
curvature necessarily vanish, we obtain a formula for the CR GJMS
operators in terms of compositions of $I^AI^{\bar B}\bD_A\bD_{\bar B}$
and $I^{\bar B}\bD_{\bar B}$ which, after some reorganization,
recovers~\eqref{eqn:factorisation_P} (cf.\ \cite{Gover2006}).  The
factorisations for $P^\prime$ and $Q^\prime$ then follow from the
``Branson trick,'' made rigorous using log densities in a manner
analogous to the ambient definitions~\eqref{eqn:ambient_pprime}
and~\eqref{eqn:ambient_qprime}; see Section~\ref{sec:pluri} and
Section~\ref{sec:algorithm} for further discussion.

%% file: bg.tex
\section{Background}
\label{sec:bg}

\subsection{CR geometry}

\newcommand{\Ga}{\Gamma}
\def\Cal{\mathcal}
\newcommand{\x}{\times}
\renewcommand{\th}{\theta}
\newcommand{\Le}{{\Cal L^{\theta}}}
\newcommand{\ol}[1]{\overline{#1}}
\newcommand{\al}{\alpha}
\newcommand{\be}{\beta}
\newcommand{\ce}{{\Cal E}}
\newcommand{\beb}{{\overline{\beta}}}
\newcommand{\alb}{{\overline{\alpha}}}
\newcommand{\bh}{{\boldsymbol h}}
\newcommand{\nd}{\nabla}
\newcommand{\cQ}{\mathcal{Q}}

Recall that an \textit{almost CR structure}, of hypersurface type, on
a smooth manifold $M$ of real dimension $2n+1$ is a rank $n$ complex
subbundle $H$ of the tangent bundle $TM$. For simplicity, throughout
the following we assume $M$ is orientable.  We denote by $J:H\to H$
the almost complex structure on the subbundle. We write $q:TM\to \mC$
for the canonical bundle surjection onto the real (quotient) line
bundle $\mC:=TM/H$.  For two sections $\xi,\eta\in\Ga(H)$ the
expression $q([\xi,\eta])$ is bilinear over smooth functions, and so
there is a skew symmetric bundle map $\Cal L:H\x H\to \mC$ given by
$\Cal L(\xi(x),\eta(x))= q([\xi,\eta](x))$. If this skew form is
non-degenerate then the almost CR structure is said to be
\textit{non-degenerate}; such non-deneracy exactly means that $H$ is a
contact distribution on $M$.

We shall write $\mathcal{B}_{\bC}$ for the complexification of a
real vector bundle $\mathcal{B}$.  Considering now $T_{\bC}M$ and  $H_{\bC}\subset T_{\bC}M$, the complex structure on $H$ is equivalent to a splitting of the
subbundle $H_{\bC}$ into the direct sum of the holomorphic part
$H^{1,0}$ and the antiholomorphic part
$H^{0,1}=\overline{H^{1,0}}$. The almost CR structure is called
\textit{integrable}, or a \textit{CR structure}, if the subbundle
$H^{1,0}\subset T_{\bC}M$ is involutive; i.e.~the space of its
sections is closed under the Lie bracket.
Then, in particular,  $\Cal L$ is of  type
$(1,1)$, meaning $\Cal L(J\xi,J\eta)=\Cal L(\xi,\eta)$ for all
$\xi,\eta\in H$. We assume integrability.

Let $q_{\bC}$ denote the complex linear extension of $q$. The
\textit{CR Levi form} $\Cal L_{\bC}$ of an almost CR structure is the
$\mC_{\bC}$-valued Hermitian form on $H^{1,0}$ induced by
$(\xi,\eta)\mapsto 2iq_{\bC}([\xi,\overline{\eta}])$. Note that
$\Cal L$ can be naturally identified with the imaginary part of $\Cal
L_{\bC}$, and so non-degeneracy of the
CR structure can be characterised by non-degeneracy of
the Levi form.

Choosing a local trivialisation of $\mC$ and using the induced
trivialisation of $\mC_{\bC}$, $\Cal L_{\bC}$ gives rise to a
Hermitian form. If $(p,q)$ is the signature of this form, then one
also says that $M$ is non-degenerate of signature $(p,q)$. If $p\neq
q$, then such local trivialisations of $\mC$ necessarily fit
together to give a global trivialisation. In the case of symmetric
signature $(p,p)$ we assume that a global trivialisation of $\mC$
exists. A global trivialisation of $\mC$ is equivalent to a ray subbundle
of the line bundle of contact forms for $H\subset TM$, so it gives a
notion of positivity for contact forms.

An important class of CR structures are those which arise from generic
real hypersurfaces in complex manifolds, as follows. Let $\Cal M$ be a
complex manifold of complex dimension $n+1$ and let $M\subset\Cal M$
be a smooth real hypersurface. For each point $x\in M$, the tangent
space $T_xM$ is a subspace of the complex vector space $T_x\Cal M$ of
real codimension one. This implies that the maximal complex subspace
$H_x$ of $T_xM$ must be of complex dimension $n$. These subspaces fit
together to define a smooth subbundle $H\subset TM$, equipped with a
complex structure. Since the bundle $H^{1,0}\subset T_{\bC}M$ can
be viewed as the intersection of the involutive subbundles $T_{\bC}M$ and $T^{1,0}\Cal M$ of $T_{\bC}\Cal M|_M$ we see that we
always obtain a CR structure in this way. Generically this structure
is non-degenerate, and in this case is referred to as an
\textit{embedded CR manifold}.

\subsection{CR density bundles} \label{dens}

In CR geometry an important role is played by a natural family of
line bundles that arise as follows. In the complexified cotangent bundle
 the annihilator of
$H^{0,1}$ has complex dimension
$n+1$, and so its $(n+1)$st complex exterior power is a complex line
bundle $\Cal K$; this is the {\em canonical bundle}.

It is convenient to assume the existence of certain roots of $\Cal
K$. Specifically we assume that there exists, and we have chosen,
a complex line bundle $\Cal E(1,0)\to M$ with the property that there is
a duality between $\Cal E(1,0)^{\otimes^{(n+2)}}$ and the canonical
bundle $\Cal K$. Such a bundle may not exist  globally, but such a choice is always possible locally. For CR manifolds embedded in $\bC^{n+1}$ the canonical bundle is trivial, so such a bundle $\Cal
E(1,0)$ exists globally in this setting. For $w,w'\in\bR$ such that
$w'-w\in\bZ$, the map $\lambda\mapsto
|\lambda|^{2w}\overline{\lambda}^{(w'-w)}$ is a well-defined
one-dimensional representation of $\bC^*$. Hence we can define a
complex line bundle $\Cal E(w,w')$ over $M$ by forming the associated
bundle to the frame bundle of $\Cal E(1,0)$ with respect to this
representation. By construction we get $\Cal E(w',w)=\overline{\Cal
  E(w,w')}$, $\Cal E(-w,-w')=\Cal E(w,w')^*$ and $\Cal E(k,0)=\Cal
E(1,0)^{\otimes^k}$ for $k\in\bN$. Finally, by definition $\Cal
K\cong\Cal E(0,-n-2)$.

\subsection{Pseudohermitian structures}\label{pseu}

For the purposes of explicit calculations on a CR manifold $(M,H)$ it
is convenient to use pseudohermitian structures, and we review some
basic facts about these. This also serves to fix conventions,
which follow \cite{GoverGraham2005}.  Since $M$ is
orientable the annihilator $H^{\perp}$ of $H$
in $T^*M$ admits a nonvanishing global section. A
\textit{pseudohermitian structure} is a choice $\th$ of such a section
and, from the non-degeneracy of the CR structure, is a contact form on
$M$. We fix an orientation on $H^{\perp}$ and restrict consideration
to choices of $\theta$ which are positive with respect to this orientation.
The \textit{Levi form} of $\theta$ is the Hermitian form $h^\theta$
(or simply $h$) on $H^{1,0}\subset T_{\bC}M$ defined by
$$
h (Z,{\ol{W}})=-2id\theta (Z,{\ol{W}}) .
$$
With the trivialisation of $\mC_{\bC}$ given by $\th$, this corresponds to $\Cal L^{\bC}$ introduced above.

Given a pseudohermitian structure $\theta$, we define the \textit{Reeb
  field} $T$ to be the unique vector field on $M$ satisfying
\begin{equation}\label{Tnorm}
\theta (T)=1 ~~~{\rm and}~~~i_Td\theta=0.
\end{equation}
An \textit{admissible coframe} is a set of complex valued forms $\{
\theta^{\al}\}$, $\al =1,\dotsc ,n$, which satisfy $\theta^{\al}(T)=0$
and whose restrictions to $H^{1,0}$ are complex linear and form a
basis for $(H^{1,0})^*$.  We use lower case Greek indices to
refer to frames for $T^{1,0}$ or its dual.  We shall also interpret
these indices abstractly, and use $\ce^{\al}$ as an abstract index
notation for the bundle $H^{1,0}$ (or its space of smooth sections)
and write $\ce_{\al}$ for its dual.  This notation is extended in an
obvious way to the conjugate bundles, and to tensor products of
various of these.

There is a natural inclusion of the real line bundle $\mC=TM/H$
into the density bundle $\ce(1,1)$ which is defined as follows. For a local
nonzero section $\al$ of $\ce(1,0)$ recall that one can, by
definition, view $\al^{-(n+2)}$ as a section of the canonical bundle
$\Cal K$. Then, by \cite[Lemma 3.2]{Lee1986}, there is a unique
positive contact form $\th$ with respect to which $\al^{-(n+2)}$ is
length normalised. From the formula in \cite[Lemma 3.2]{Lee1986}
one sees that, in the other direction, $\theta$ determines $\al$ up a phase
factor, and scaling $\theta$ causes the inverse scaling of $\alpha\ol{\al}$.
Thus the mapping $TM\ni\xi\mapsto\th(\xi)\al\overline{\al}$
descends to an inclusion of $\mC$ into $\mE(1,1)$ which, by construction, is CR
invariant.   A \emph{scale} $\alpha\bar\alpha\in\mE(1,1)$ is a section of the image of $\mC$ in $\mE(1,1)$ under this inclusion.

 By integrability and \eqref{Tnorm}, we have
$$
d\theta = ih_{\al\beb} \theta^{\al}\wedge \theta^{\beb}
$$ for a smoothly varying Hermitian matrix $h_{\al\beb}$, which we may
interpret as the matrix of the Levi form $h$ determined by $\theta$,
in the frame $\theta^{\al}$, or as the Levi form $h$ itself in
abstract index notation. Using the inclusion
$\mC\hookrightarrow\ce(1,1)$ from above, the CR Levi form $\Cal
L^{\bC}$ can be viewed as a canonical section of
$\ce_{\al\beb}(1,1)$ which we also denote by $\bh_{\al\beb}$; this
agrees with $h_{\al\beb}$ if $\ce(1,1)$ is trivialised using $\theta$.  By
$\bh^{\al\beb}\in\ce^{\al\beb}(-1,-1)$ we denote the inverse of
$\bh_{\al\beb}$ and this will be used to raise and lower indices
without further mention.

By $\nabla$ we denote the \textit{Tanaka--Webster connections} (on
various bundles) associated to $\th$. In particular, these satisfy
$\nabla\th=0$, $\nabla h=0$, $\nd \bh=0$, $\nabla T=0$, and $\nd J=0$,
so the decomposition $T_{\bC}M=H^{1,0}M\oplus H^{0,1}M\oplus\bC
T$ is invariant under $\nabla$.  On tensors, the Tanaka--Webster connection is determined from the Webster connection forms $\omega_\alpha{}^\beta$ and the torsion forms $\tau_\gamma=A_{\alpha\gamma}\theta^\alpha$, defined in terms of an admissible coframe by
\begin{align*}
 d\theta^\alpha & = \theta^\beta\wedge\omega_\beta{}^\alpha + \theta\wedge\tau^\alpha, \\
 dh_{\alpha\bar\beta} & = \omega_{\alpha\bar\beta} + \omega_{\bar\beta\alpha} , \\
 A_{\alpha\gamma} & = A_{\gamma\alpha} .
\end{align*}
We call $A_{\alpha\gamma}$ the \emph{torsion} of $\theta$.  The \emph{pseudohermitian curvature} $R_{\alpha\bar\beta\gamma\bar\sigma}$ of $\theta$ is obtained from the curvature forms $\Pi_\alpha{}^\beta=d\omega_\alpha{}^\beta-\omega_\alpha{}^\gamma\wedge\omega_\gamma{}^\beta$ via the structure equations
\begin{multline}
 \label{eqn:structure_equation}
 \Pi_\alpha{}^\beta = R_\alpha{}^\beta{}_{\mu\bar\nu}\theta^\mu\wedge\theta^{\bar\nu} + \nabla^\beta A_{\alpha\mu}\theta^\mu\wedge\theta - \nabla_\alpha A_{\bar\nu}{}^\beta\theta^{\bar\nu}\wedge\theta \\ + ih_{\alpha\bar\nu}A_{\bar\sigma}{}^\beta\theta^{\bar\nu}\wedge\theta^{\bar\sigma} - iA_{\alpha\mu}\theta^\mu\wedge\theta^\beta
\end{multline}
The \emph{pseudohermitian Ricci tensor} is $R_{\alpha\bar\beta}=R_{\alpha\bar\beta\gamma}{}^\gamma$ and the \emph{pseudohermitian scalar curvature} is $R=R_\gamma{}^\gamma$.  The \emph{sublaplacian} is $\Delta_b=\nabla^\gamma\nabla_\gamma + \nabla_\gamma\nabla^\gamma$.

A contact form $\theta$ on $(M^{2n+1},H)$ is \emph{pseudo-Einstein} if
\[ \begin{cases}
    R_{\alpha\bar\beta} = \frac{1}{n}Rh_{\alpha\bar\beta}, & \text{if $n>1$} \\
    \nabla_\alpha R = i\nabla^\gamma A_{\alpha\gamma}, & \text{if $n=1$} .
   \end{cases} \]
The set of pseudo-Einstein contact forms, when non-empty, forms a distinguished class of contact forms parameterised by $\mP$: If $\theta$ is pseudo-Einstein, then $\htheta=e^\Upsilon\theta$ is pseudo-Einstein if and only if $\Upsilon$ is a CR pluriharmonic function~\cite{Hirachi1990,Lee1988}.

We may decompose a tensor field relative to the splittings of $T_{\bC}M$ and its dual.  In this way, we may calculate the covariant derivative componentwise.  Each of the components may be regarded as a section of a tensor product of $\mE^\alpha$ or its dual or conjugates thereof.  We therefore often restrict consideration to the action of the connection on $\mE^\alpha$ or $\mE_\alpha$.  We use indices $\alpha,\overline{\alpha},0$ for components with respect to the frame $\{\theta^\alpha,\theta^{\bar\alpha},\theta\}$ and its dual, so that the $0$-components incorporate weights.  If $f$ is a (possibly density-valued) tensor field, we denote components of the (tensorial) iterated covariant derivatives of $f$ in such a frame by preceding $\nabla$'s; e.g.\ $\nabla_\alpha\nabla_0\dotsm\nabla_{\bar\beta}$.  As usual, such indices may alternately be interpreted abstractly.  For example, if $f_\beta\in\mE_\beta(w,w^\prime)$, we consider $\nabla f$ as the triple $\nabla_\alpha f_\beta\in\mE_{\alpha\beta}(w,w^\prime)$, $\nabla_{\bar\alpha}f_\beta\in\mE_{\bar\alpha\beta}(w,w^\prime)$, $\nabla_0f_\beta\in\mE_\beta(w-1,w^\prime-1)$.

From the standpoint of CR geometry, it is convenient to consider certain modifications of the curvature $R_{\alpha\bar\beta\gamma\bar\sigma}$ and its traces.  The \emph{pseudohermitian Schouten tensor} is defined by
\[ P_{\alpha\bar\beta} = \frac{1}{n+2}\left( R_{\alpha\bar\beta} - Ph_{\alpha\bar\beta} \right), \]
where $P=R/2(n+1)$ is its trace.  To describe the tractor connection, it is convenient to introduce the tensors
\begin{align*}
 T_\alpha & = \frac{1}{n+2}\left(\nabla_\alpha P - i\nabla^\gamma A_{\alpha\gamma}\right), \\
 S & = -\frac{1}{n}\left(\nabla^\alpha T_\alpha + \nabla_\alpha T^\alpha + P_{\alpha\bar\beta}P^{\alpha\bar\beta} - A_{\alpha\gamma}A^{\alpha\gamma}\right)
\end{align*}
(cf.\ \cite{GoverGraham2005,Lee1986}).  The \emph{Chern tensor} is defined by
\[ S_{\alpha\bar\beta\gamma\bar\sigma} = R_{\alpha\bar\beta\gamma\bar\sigma} - P_{\alpha\bar\beta}h_{\gamma\bar\sigma} - P_{\alpha\bar\sigma}h_{\gamma\bar\beta} - P_{\gamma\bar\beta}h_{\alpha\bar\sigma} - P_{\gamma\bar\sigma}h_{\alpha\bar\beta} . \]
This tensor is the analogue of the Weyl tensor, in that it is CR invariant, has Weyl-type symmetries, and, when $n\geq5$, is the obstruction to $(M^{2n+1},H)$ being locally equivalent to the standard CR $(2n+1)$-sphere~\cite{ChernMoser1974}.  Some other important curvature tensors, which together constitute the curvature of the CR tractor connection~\cite{GoverGraham2005}, are
\begin{align*}
 V_{\alpha\bar\beta\gamma} & = \nabla_{\bar\beta}A_{\alpha\gamma} + i\nabla_\gamma P_{\alpha\bar\beta} - iT_\gamma h_{\alpha\bar\beta} - 2iT_\alpha h_{\gamma\bar\beta} \\
 Q_{\alpha\gamma} & = i\nabla_0A_{\alpha\gamma} - 2i\nabla_\gamma T_\alpha + 2P_{\alpha}{}^\rho A_{\rho\gamma} \\
 U_{\alpha\bar\beta} & = \nabla_{\bar\beta}T_\alpha + \nabla_\alpha T_{\bar\beta} + P_\alpha{}^\rho P_{\rho\bar\beta} - A_{\alpha\rho} A^\rho{}_{\bar\beta} + Sh_{\alpha\bar\beta} \\
 Y_\alpha & = \nabla_0T_\alpha - i\nabla_\alpha S + 2iP_\alpha{}^\rho T_\rho - 3A_{\alpha\rho}T^\rho .
\end{align*}
Note that $V_{\alpha\bar\beta\gamma}$ and $U_{\alpha\bar\beta}$ are tracefree; this follows from the definitions of $T_\alpha$ and $S$, respectively.  Note also that $V_{\alpha\bar\beta\gamma}$ and $Q_{\alpha\gamma}$ are symmetric.  Indeed, \cite[(2.10)]{Lee1988} implies that
\begin{equation}
 \label{eqn:V_symmetry}
 \nabla_\gamma P_{\alpha\bar\beta} - \nabla_\alpha P_{\gamma\bar\beta} = T_\alpha h_{\gamma\bar\beta} - T_\gamma h_{\alpha\bar\beta}
\end{equation}
and~\cite[(2.9)]{Lee1988} implies that
\begin{equation}
 \label{eqn:Q_symmetry}
 \nabla_\alpha T_\gamma - \nabla_\gamma T_\alpha = iP_\alpha{}^\rho A_{\rho\gamma} - iP_\gamma{}^\rho A_{\rho\alpha} ;
\end{equation}
these equations imply $V_{\alpha\bar\beta\gamma}=V_{\gamma\bar\beta\alpha}$ and $Q_{\alpha\gamma}=Q_{\gamma\alpha}$, respectively.  When $n=1$, the \emph{Cartan tensor} $Q_{\alpha\beta}$ is CR invariant and is the obstruction to $(M^3,H)$ being locally CR equivalent to the standard CR three-sphere~\cite{Cartan1932b,Cartan1932a}.  These curvature tensors are all related via simple divergence formulae:

\begin{lem}
 \label{lem:div_identities}
 Let $(M^{2n+1},H,\theta)$ be a pseudohermitian manifold.  Then
 \begin{align}
  \label{eqn:divS} \nabla^{\bar\sigma} S_{\alpha\bar\beta\gamma\bar\sigma} & = -inV_{\alpha\bar\beta\gamma} , \\
  \label{eqn:divV} \nabla^{\bar\beta} V_{\alpha\bar\beta\gamma} & = -(n-1)Q_{\alpha\gamma} + S_{\alpha\bar\beta\gamma\bar\sigma} A^{\bar\beta\bar\sigma} , \\
  \label{eqn:divbarV} \nabla^\gamma V_{\alpha\bar\beta\gamma} & = niU_{\alpha\bar\beta} - iS_{\alpha\bar\beta\gamma\bar\sigma}P^{\gamma\bar\sigma} , \\
  \label{eqn:divQ} \nabla^\gamma Q_{\alpha\gamma} & = -nY_\alpha + 2V_{\alpha\bar\beta\gamma}P^{\gamma\bar\beta} , \\
  \label{eqn:divU} \nabla^{\bar\beta} U_{\alpha\bar\beta} & = -(n-1)iY_\alpha + iV_{\alpha\bar\beta\gamma}P^{\gamma\bar\beta} + V_{\bar\beta\alpha\bar\sigma}A^{\bar\beta\bar\sigma}, \\
  \label{eqn:RedivY} \Real\nabla^\gamma Y_\gamma & = \Imaginary A^{\alpha\gamma}Q_{\alpha\gamma} .
 \end{align}
\end{lem}

\begin{proof}
 In terms of the pseudohermitian Schouten tensor, the Bianchi identity~\cite[(2.11)]{Lee1988} states that
 \begin{equation}
  \label{eqn:divP}
   \nabla^{\bar\beta}P_{\alpha\bar\beta} = \nabla_\alpha P + (n-1)T_\alpha .
 \end{equation}
 Using this and the Bianchi identity~\cite[(2.7)]{Lee1988}, it follows that
 \[ \nabla^{\bar\sigma}S_{\alpha\bar\beta\gamma\bar\sigma} = (n+1)\nabla_\alpha P_{\gamma\bar\beta} - \nabla_\gamma P_{\alpha\bar\beta} - in\nabla_{\bar\beta}A_{\alpha\gamma} - (2n+1)T_\gamma h_{\alpha\bar\beta} - (n-1)T_\alpha h_{\gamma\bar\beta} . \]
 Writing this in terms of $V_{\alpha\bar\beta\gamma}$ yields~\eqref{eqn:divS}.

 Using commutator formulae from~\cite[Lemma~2.3]{Lee1988}, we observe that
 \[ \nabla^{\bar\beta}\nabla_\gamma P_{\alpha\bar\beta} = \nabla_\gamma\nabla^{\bar\beta} P_{\alpha\bar\beta} + (n-1)iTP_{\alpha\bar\beta} A^{\bar\beta}{}_\gamma + iP_{\gamma\bar\beta} A^{\bar\beta}{}_\alpha - iPA_{\alpha\gamma} . \]
 It follows from this, \eqref{eqn:divP} and the definitions of $V_{\alpha\bar\beta\gamma}$ and $T_\alpha$ that
 \begin{multline}
  \label{eqn:pre_divV}
  \nabla^{\bar\beta}V_{\alpha\bar\beta\gamma} = \nabla^{\bar\beta}\nabla_{\bar\beta} A_{\alpha\gamma} - \nabla_\gamma\nabla_{\bar\beta} A^{\bar\beta}{}_\alpha + (2n-1)i\nabla_\gamma T_\alpha - i\nabla_\alpha T_\gamma \\ - (n-1)P_{\alpha\bar\beta} A^{\bar\beta}{}_\gamma - P_{\gamma\bar\beta} A^{\bar\beta}{}_\alpha + PA_{\alpha\gamma} .
 \end{multline}
 In terms of the Chern tensor and the pseudohermitian Schouten tensor, the Bianchi identity~\cite[(2.9)]{Lee1988} states that
 \begin{multline*}
  \nabla^{\bar\beta}\nabla_{\bar\beta} A_{\alpha\gamma} - \nabla_\gamma\nabla_{\bar\beta} A^{\bar\beta}{}_\alpha = -(n-1)i\nabla_0A_{\alpha\gamma} + S_{\alpha\bar\beta\gamma\bar\sigma} A^{\bar\beta\bar\sigma} \\ - nP_{\alpha\bar\beta} A^{\bar\beta}{}_\gamma + 2P_{\gamma\bar\beta} A^{\bar\beta}{}_\alpha - PA_{\alpha\gamma}
 \end{multline*}
 Inserting this into~\eqref{eqn:pre_divV} and using~\eqref{eqn:Q_symmetry} yields~\eqref{eqn:divV}.

 From the symmetry $V_{\alpha\bar\beta\gamma}=V_{\gamma\bar\beta\alpha}$, we may write
 \[ \nabla^\gamma V_{\alpha\bar\beta\gamma} = \nabla^\gamma\nabla_{\bar\beta}A_{\alpha\gamma} + i\nabla^\gamma\nabla_\alpha P_{\gamma\bar\beta} - i\nabla_{\bar\beta}T_\alpha - 2i\nabla^\gamma T_\gamma h_{\alpha\bar\beta} . \]
 Commutator formulae from~\cite[Lemma~2.3]{Lee1988} yield
 \begin{align*}
  \nabla^\gamma\nabla_{\bar\beta}A_{\alpha\gamma} & = \nabla_{\bar\beta}\nabla^\gamma A_{\alpha\gamma} - niA_{\alpha\gamma}A^\gamma{}_{\bar\beta} + i\left| A_{\gamma\rho}\right|^2 h_{\alpha\bar\beta} , \\
  \nabla^\gamma\nabla_\alpha P_{\gamma\bar\beta} & = \nabla_\alpha\nabla^\gamma P_{\gamma\bar\beta} + i\nabla_0P_{\alpha\bar\beta} - S_{\alpha\bar\beta\gamma\bar\sigma} P^{\gamma\bar\sigma} + nP_\alpha{}^\gamma P_{\gamma\bar\beta} - \left|P_{\gamma\bar\sigma}\right|^2h_{\alpha\bar\beta} .
 \end{align*}
 In terms of the pseudohermitian Schouten tensor, the Bianchi identity~\cite[(2.12)]{Lee1988} states that
 \begin{equation}
  \label{eqn:nabla0P}
  \nabla_0P_{\alpha\bar\beta} = i\nabla_{\bar\beta}T_\alpha - i\nabla_\alpha T_{\bar\beta} .
 \end{equation}
 Combining these three displays yields~\eqref{eqn:divbarV}.

 Combining~\cite[(2.6)]{Lee1988} and a commutator formula from~\cite[Lemma~2.3]{Lee1988} yields
 \[ \nabla^\gamma\nabla_0 A_{\alpha\gamma} = \nabla_0\nabla^\gamma A_{\alpha\gamma} + \nabla_\alpha\left|A_{\gamma\rho}\right|^2 + A_{\alpha\gamma}\nabla_\rho A^{\gamma\rho} . \]
 Writing this in terms of $T_\alpha$ and using another commutator formula from~\cite[Lemma~2.3]{Lee1988} yields
 \begin{multline*}
  \nabla^\gamma\nabla_0 A_{\alpha\gamma} = (n+2)i\nabla_0T_\alpha + \nabla_\alpha\nabla^\gamma T_\gamma - \nabla_\alpha\nabla_\gamma T^\gamma \\ + \nabla_\alpha\left|A_{\gamma\rho}\right|^2 - (n+2)iA_{\alpha\gamma}T^\gamma + 2iA_{\alpha\gamma}\nabla^\gamma P .
 \end{multline*}
 In particular, when combined with the definition of $Q_{\alpha\gamma}$, this yields
 \begin{multline}
  \label{eqn:pre_divQ}
  \nabla^\gamma Q_{\alpha\gamma} = -(n+2)\nabla_0T_\alpha + i\nabla_\alpha\nabla^\gamma T_\gamma - i\nabla_\alpha\nabla_\gamma T^\gamma - 2i\nabla^\gamma\nabla_\gamma T_\alpha \\ + 2\nabla^\rho\left(P_\alpha{}^\gamma A_{\gamma\rho}\right) + i\nabla_\alpha\left|A_{\gamma\rho}\right|^2 + (n+2)A_{\alpha\gamma}T^\gamma - 2A_{\alpha\gamma}\nabla^\gamma P .
 \end{multline}
 Next, \eqref{eqn:Q_symmetry} and a commutator formula from~\cite[Lemma~2.3]{Lee1988} yields
\begin{align*}
  \nabla^\gamma\nabla_\gamma T_\alpha - \nabla_\alpha\nabla^\gamma T_\gamma =& i\nabla_0T_\alpha + (n+2)P_\alpha{}^\gamma T_\gamma + PT_\alpha - i\nabla^\rho\left(P_\alpha{}^\gamma A_{\gamma\rho}\right) \\
  & + i\nabla^\rho\left(P_\rho{}^\gamma A_{\gamma\alpha}\right) .
  \end{align*}
 Inserting this into~\eqref{eqn:pre_divQ} and using~\eqref{eqn:divP} and the definitions of $S$, $V_{\alpha\bar\beta\gamma}$ and $Y_\alpha$ yields~\eqref{eqn:divQ}.

 Using the identities
 \begin{align*}
  \nabla_\beta\nabla_\alpha T^\beta - \nabla_\alpha\nabla_\beta T^\beta & = (n-1)iA_{\alpha\beta}T^\beta, \\
  \nabla_\beta\nabla^\beta T_\alpha - \nabla_\alpha \nabla^\beta T_\beta & = -(n-1)i\nabla_0 T_\alpha + iP_\beta{}^\gamma\nabla^\beta A_{\gamma\alpha} - iA_{\gamma\beta}\nabla^\beta P_\alpha{}^\gamma \\
   & \quad - P_\alpha{}^\gamma\nabla_\gamma P + (n+2)P_\alpha{}^\gamma T_\gamma + iA_{\alpha\gamma}\nabla^\gamma P \\ & \quad + (n-1)iA_{\alpha\gamma}T^\gamma ,
 \end{align*}
 one readily derives~\eqref{eqn:divU}.

 From the definitions of $T_\alpha$, $S$, and $Y_\alpha$ we compute that
 \[ \Real \nabla^\gamma Y_\gamma = -P_\alpha{}^\gamma\nabla_0P_\gamma{}^\alpha + \frac{1}{2}\nabla_0\left|A_{\alpha\gamma}\right|^2 + iP_\alpha{}^\gamma\left(\nabla^\alpha T_\gamma - \nabla_\gamma T^\alpha\right) - \Real A_{\alpha\gamma}\nabla^\gamma T^\alpha . \]
 Combining this with~\eqref{eqn:nabla0P} and the definition of $Q_{\alpha\gamma}$ yields~\eqref{eqn:RedivY}.
\end{proof}

%% file: tractor.tex
\section{Some tractor calculus}
\label{sec:tractor}

\subsection{The CR tractor connection}\label{CRtract}

\newcommand{\E}{\mathcal{E}}
\newcommand{\up}{\Upsilon}
\newcommand{\Up}{\Upsilon}
\newcommand{\s}{\sigma}
\newcommand{\si}{\sigma}
\newcommand{\g}{\gamma}
\renewcommand{\cT}{\mathcal{T}}
\newcommand{\cG}{\mathcal{G}}
\newcommand{\cF}{\mathcal{F}}
\newcommand{\cC}{\mathcal{C}}
\newcommand{\MF}{M_\cF}

\newcommand{\sfrac}[2]{{\textstyle \frac{#1}{#2}}}

\newcommand{\nn}[1]{(\ref{#1})}

\newcommand{\lpl}
{\mbox{$
\begin{picture}(12.7,8)(-.5,-1)
\put(2,0.2){$+$}
\put(6.2,2.8){\oval(8,8)[l]}
\end{picture}$}}


On a hypersurface type CR structure there is no invariant connection
on the tangent bundle, or its contact subbundle. However there is a natural invariant connection on a higher rank natural vector bundle
known as the CR cotractor bundle~\cite{GoverGraham2005}.  A defining
feature of this bundle $\ce_A$ is that for each choice of
pseudo-Hermitian contact form $\theta$ this bundle decomposes into a
direct sum
\begin{equation}\label{tsplit}
\ce_A\stackrel{\theta}{=}  \ce(1,0)\oplus\ce_\al(1,0)\oplus\ce(0,-1).
\end{equation}
So for a section $T_A\in \Gamma(\ce_{A})$ we may write
  $[v_A]_\theta= (\si, \tau_\alpha, \rho)$, where $\s\in \ce(1,0)$, $\tau_\alpha \in\ce_\alpha(1,0)$, and
$\rho\in\ce(0,-1)$.  When the choice of
  $\theta$ is understood it will be omitted from the notation.  A
  change of contact form to $\hat{\theta}=e^\Up \theta$, where $\Up\in
  C^\infty (M)$, induces a different identification to the same direct
  sum bundle, with the components in the $\hat{\theta}$ direct sum
  related to those in the $\theta$ direct sum by the transformation formula
\begin{equation}\label{identif}
     [v_A]_{\widehat{\theta}}=\left( \begin{array}{c} \hat{\sigma}\\
                            \hat{\tau_{\al}}\\
          \hat{\rho} \end{array}\right)
=
\left( \begin{array}{c} \sigma\\
                            \tau_{\al}+\up_\al\sigma\\
\rho -\up^\be\tau_\be-\frac{1}{2}(\up^\be\up_\be+i\up_0)\sigma
\end{array}\right) ,
\end{equation}
where we have used an obvious notation; e.g.\ $\Up_\alpha:=\nabla_\alpha \Up$.  It follows from \nn{identif} that
$\ce_A$ has a composition series
$$
\ce_A=\ce(1,0)\lpl \ce_\al(1,0)\lpl \ce(0,-1),
$$ meaning that in a CR invariant way, $\ce(0,-1)$ is a subbundle of
$\ce_A$, $\ce(1,0)$ is a quotient bundle of $\ce_A$, and the kernel of
the surjection $\ce_A\to \ce(1,0)$ is the reducible subbundle
$\ce_\al(1,0)\lpl \ce(0,-1)$. We write $Z_A$ to denote the canonical bundle injection
$$
Z_A: \ce(0,-1)\to \ce_A ,
$$
and also view this as a section $Z_A\in \Gamma(\ce_A(0,1))$.
Note that conjugation extends to
tractors in the obvious way and, for example, the conjugate tractor
bundle has the composition series
$$
\ce_{\bar{A}}=\ce(0,1)\lpl \ce_{\bar\al}(0,1)\lpl \ce(-1,0),
$$
with the inclusion of $\ce(-1,0)$ denoted
\begin{equation}\label{basic-incl}
Z_{\bar{A}}:\ce(-1,0)\to\ce_{\bar{A}}.
\end{equation}

For two sections $v_A$ and $v'_B$ the quantity
$\si\overline{\rho'}+\rho\overline{\si'}+\bh^{\al\beb}\tau_\al\overline{\tau'_\be}$
is independent of the choice of $\theta$; this is the formula, in a
given scale $\theta$, for the CR invariant Hermitian metric
$h^{A\bar{B}}$ on $\ce_A$. The {\em tractor metric} $h_{A\bar{B}}$ is
the inverse of this, and it gives the Hermitian metric on the {\em (standard) tractor
bundle} $\ce^A$, which by definition is the bundle dual to $\ce_A$.
The tractor metric is used to raise and lower tractor indices in the
usual way; e.g.\ $Z^A:=h^{A\bar{B}}Z_{\bar{B}}$.  Note that from the formula for the metric this gives the surjection
$$
Z^A:\ce_A\to \ce(1,0).
$$
We refer to $Z^A$ as the {\em (CR) canonical tractor}.


In terms of the tractor splitting \nn{tsplit}, the CR tractor
connection is given by the formulae
$$
\nd_\be v_A=\left( \begin{array}{c}
                            \nd_\be\sigma-\tau_\be\\
                            \nd_\be\tau_\al+iA_{\al\be}\sigma\\
         \nd_\be \rho-P_\be{}^\al\tau_\al+T_\be\sigma
                                   \end{array}\right) ,
$$
\begin{equation}\label{tracconn}
\nd_{\beb} v_A= \left( \begin{array}{c}
                               \nd_{\beb} \sigma\\
\nd_{\beb}\tau_\al
+\bh_{\al\beb}\rho+P_{\al\beb} \sigma\\
\nd_{\beb}\rho+iA_{\beb}{}^\al \tau_\al-T_{\beb}\sigma
 \end{array} \right),
\end{equation}
$$
\nd_0 v_A= \left( \begin{array}{c}
                               \nd_0 \sigma+\frac{i}{n+2}P\sigma-i\rho\\
\nd_0\tau_\al-iP_\al{}^\be\tau_\be+\frac{i}{n+2}P\tau_\al+
2iT_\al\sigma\\
\nd_0\rho+\frac{i}{n+2}P\rho+2iT^\al\tau_\al+iS\sigma
 \end{array} \right),
$$
where the $\nd$'s on the right hand side refer to the
pseudohermitian connection on the appropriate weighted bundles.  This
connection is canonically determined by the CR structure (and so, in
particular, is independent of $\theta$); indeed, it is equivalent to the
normal Cartan connection on CR manifold \cite{CapGover2008}. The tractor
connection preserves the tractor metric, $\nabla h_{A\bar{B}}=0$,
and so covariant differentiation commutes with the raising and
lowering of tractor indices.

\subsection{The tractor $D$-operator}\label{Dop} For the construction of
differential operators an important tool is the second order \emph{tractor $D$-operator}
$$
\bD_A:\ce^\star (w,w')\to \ce_A\otimes \ce^\star (w-1,w') ,
$$
where $\ce^\star(w,w')$ indicates any weighted
 tractor bundle, meaning it is the tensor product of $\ce(w,w')$ with any bundle constructed by taking a tensor part of any tensor product of the tractor bundle, its dual, and the conjugates of these.
This is defined by
$$
\bD_A f=\left( \begin{array}{c} w(n+w+w') f\\
                            (n+w+w')\nd_\al f\\
  -(\nd^\be\nd_\be f+iw\nd_0 f +w(1+\frac{w'-w}{n+2})Pf)
                            \end{array}\right)
$$ in the splitting \nn{tsplit} determined by a choice of $\theta$,
but is independent of the choice of $\theta$.  Here $\nd_\al f$ refers
to the tractor connection defined above coupled with the
pseudohermitian connection.
 Conjugation produces the CR invariant operator
\[ \bD_{\bar{A}}:\ce^\star (w,w')\to \ce_{\bar{A}}\otimes \ce^\star (w,w'-1) . \]
By construction both $\bD_A$ and
$\bD_{\bar{A}}$ commute with raising and lowering tractor
indices.

Two other important operators on weighted tractor bundles are the weight operator $\bw$ and its conjugate $\bw^\prime$.  The weight operators are defined to be the unique derivations such that
\[ \bw f = wf \quad\text{and}\quad \bw^\prime f = w^\prime f \]
for all $f\in\mE^\star(w,w^\prime)$.  In particular, the tractor $D$-operator can be written as an operator on arbitrary weighted tractor bundles by
\[ \bD_Af = \begin{pmatrix} \bw(n+\bw+\bw^\prime)f \\ \nabla_\alpha(n+\bw+\bw^\prime)f \\ -\left(\nabla^\beta\nabla_\beta f + i\nabla_0\bw f + P\bw\left(1+\frac{\bw^\prime-\bw}{n+2}\right)f\right) \end{pmatrix} . \]
One can similarly write compositions of tractor $D$-operators; of particular importance in this article is the formula for the composition $\bD_A\bD_{\bar B}$.

\begin{prop}
 \label{prop:ddbar-weight}
 Let $(M^{2n+1},H)$ be a CR manifold.  Given any scale $\sigma\bar\sigma\in\mE(1,1)$, the operator $\bD_A\bD^B=\bh^{B\bar B}\bD_A\bD_{\bar B}$ acts on elements of $\mE(w,w)$, $w\in\bR$, by
 \[ \bD_A\bD^B = \begin{pmatrix} C_3 & (C_2)^\beta & C_1 \\ (C_5)_\alpha & (C_4)_\alpha{}^\beta & (C_2)_\alpha \\ C_6 & (C_5)^\beta & \overline{C_3} \end{pmatrix} , \]
 where
 \begin{align*}
  C_1 & = \bw^2(n+2\bw)(n+2\bw-1), \\
  (C_2)^\beta & = \nabla^\beta\bw(n+2\bw)(n+2\bw-1) , \\
  C_3 & = -\frac{1}{2}(\Delta_b - i\nabla_0(n+2\bw) + 2P\bw)\bw(n+2\bw-1) , \\
  (C_4)_\alpha{}^\beta & = \left(\nabla_\alpha\nabla^\beta - \frac{1}{n}h_\alpha{}^\beta\nabla_\gamma\nabla^\gamma\right)(n+2\bw)(n+2\bw-1) \\
   & \quad + P_\alpha{}^\beta\bw(n+2\bw)(n+2\bw-1) + \frac{1}{n}h_\alpha{}^\beta\left(\Delta_b - nP\right)\bw(n+2\bw-1) , \\
  (C_5)_\alpha & = -\frac{1}{n}\left(\nabla_\alpha\nabla_\gamma\nabla^\gamma + inA_{\alpha\gamma}\nabla^\gamma\right)(n+2\bw)(n+2\bw-1) \\
   & \quad + \left(\frac{1}{n}\nabla_\alpha\Delta_b - \nabla_\alpha P - T_\alpha(n+2\bw)\right)\bw(n+2\bw-1) , \\
  C_6 & = \frac{1}{n^2}\nabla^\gamma\left(\nabla_\gamma\nabla_\alpha\nabla^\alpha + inA_{\alpha\gamma}\nabla^\alpha\right)(n+2\bw)^2 + P^{(\alpha\bar\beta)_0}\nabla_\alpha\nabla_{\bar\beta}(n+2\bw) \\
   & \quad - \frac{1}{n^2}\Delta_b^2\bw(n+\bw) + \frac{2}{n}\Imaginary \nabla^\gamma A_{\beta\gamma}\nabla^\beta\bw(n+2\bw) - \frac{1}{n}P\Delta_b\bw(n-1+2\bw) \\
   & \quad + \frac{4}{n}\Real\nabla^\gamma P\nabla_\gamma \bw(n+\bw) - \frac{2}{n}\Real (\nabla^\gamma P + nT^\gamma)\nabla_\gamma\bw(n+2\bw) \\
   & \quad + \frac{1}{n}\left(\nabla^\gamma(\nabla_\gamma P + nT_\gamma)\right)\bw(n+2\bw) + \bigl| P_{(\alpha\bar\beta)_0}\bigr|^2\bw(n+2\bw) \\
   & \quad + \frac{1}{n}\left((n+2)P^2 - (\Delta_b P) - nS(n+2\bw)\right)\bw^2 ,
 \end{align*}
 $P_{(\alpha\bar\beta)_0}=P_{\alpha\bar\beta}-\frac{P}{n}h_{\alpha\bar\beta}$
 is the tracefree part of the CR Schouten tensor, the term $T_\alpha$
 in the definition of $(C_5)_\alpha$ acts as a multiplication
 operator, the terms $\left(\nabla^\gamma(\nabla_\gamma P +
 nT_\gamma)\right)$ and $(\Delta_b P)$ in the definition of $C_6$ act
 as multiplication operators, and all other operators in the
 definitions of $C_1,\dotsc,C_6$ act to the right.
\end{prop}

\begin{proof}
 Using the definition of the tractor $D$-operator and~\cite[Proposition~2.2]{GoverGraham2005}, we observe that if $f\in\mE(w,w)$, then
 \begin{equation}
  \label{eqn:DB}
  \bD^Bf = \left( -\frac{1}{2}\Delta_bf + \frac{n+2w}{2}i\nabla_0f - wPf , (n+2w)\nabla^\beta f , w(n+2w)f \right) .
 \end{equation}
 Using this and the formula
 \[ \bD_A\bD^Bf = \begin{pmatrix} w(n+2w-1)D^Bf \\ (n+2w-1)\nabla_\alpha D^Bf \\ -\left(\nabla^\gamma\nabla_\gamma D^Bf + iw\nabla_0D^Bf + \frac{n+1}{n+2}wPD^Bf\right) \end{pmatrix} \]
 yields the expressions for the operators $C_1,(C_2)^\beta$, and $C_3$.

 Next observe that
 \[ \frac{1}{2}\left(\Delta_bf - (n+2w)i\nabla_0f + 2wPf\right) = \frac{1}{n}\left((n+2w)\nabla_\beta\nabla^\beta f - w\Delta_b f + wnPf\right) . \]
 Combining this with~\eqref{eqn:DB} and the expression for the tractor connection yields the expressions for the operators $(C_4)_\alpha{}^\beta$ and $(C_5)_\beta$.

 Finally, note that
 \[ \nabla^\gamma\nabla_\gamma I_{\bar B} = \begin{pmatrix} \nabla^\gamma\nabla_\gamma\sigma - \nabla_\gamma\tau^\gamma - P\sigma - n\rho \\ \nabla^\gamma\nabla_\gamma\tau_{\bar\beta} + P_{\gamma\bar\beta}\nabla^\gamma\sigma + \sigma\left(\nabla_{\bar\beta} P + (n-1)T_{\bar\beta}\right) + \nabla_{\bar\beta}\rho - iA_{\bar\beta}{}^\gamma\nabla_\gamma\sigma \\ \mD \end{pmatrix} , \]
 where
 \[ \mD = \nabla^\gamma\nabla_\gamma\rho - i\nabla^\gamma(A_{\alpha\gamma}\tau^\alpha) - \nabla^\gamma(\sigma T_\gamma) - P_\alpha{}^\gamma\left(\nabla_\gamma\tau^\alpha + \sigma P_\gamma{}^\alpha + \rho h_\gamma{}^\alpha\right) + T^\gamma\nabla_\gamma\sigma . \]
 Combining this with~\eqref{eqn:DB} and the definition of the tractor $D$-operator yields the expression for the operator $C_6$.
\end{proof}

\input{pluri}

\section{Tractors and the Fefferman ambient metric}\label{ambsection}

On a CR manifold the tractor calculus provides the basic invariant
calculus. It is the CR analogue of the calculus surrounding the
Levi-Civita connection in Riemannian geometry. We need to link this to
the Fefferman ambient metric for two reasons: First the CR GJMS
operators and the related basic objects are {\em defined} in terms of
the ambient metric. Second the ambient metric provides a powerful tool
for simplifying tractor calculus computations; this works well
because the ambient metric is effectively a
non-linear extension of the tractor bundle and connection that
captures these in terms of a K\"ahler metric (of mixed signature) and
connection, see in particular Theorem \ref{trthm} below.

Most components of this link between tractors and the ambient metric
are available in the literature.  To adapt and extend these as
required for our current purposes it is useful to first
understand the principal bundle structure equivalent to the tractor
connection, namely the Cartan connection. This provides a conceptual
framework for the tractor connection and its use. In particular, it
enables us below to construct and understand the Fefferman space
and the ambient connection from this perspective. To understand the
groups involved we first recall the model for CR geometry.

\subsection{The Cartan connection and the model}\label{Cartan}
Fix a complex vector space $\bV$ of complex dimension $n+2$, equipped with
a Hermitian inner product $\langle\ ,\ \rangle$ of signature
$(p+1,q+1)$, where $p+q=n$.  Let $\mN\subset\bV$ be the cone of
nonzero null vectors in $\bV$.
Then the image $S$ of $\mN$ in
the complex projectivisation $\Cal P\bV\cong\bCP^{n+1}$ has a CR
structure, and this provides the usual flat model for hypersurface
type CR geometry.

Denote by $G\cong SU(p+1,q+1)$ the special unitary group of
$(\bV,\langle\ ,\ \rangle)$.  Note that $G$ acts transitively on
$S$. Thus $S$ may be naturally identified with the homogeneous space
$G/P$, where $P\subset G$ is the isotropy subgroup of a nominated
point on $S$. Note that $P$ stabilises a complex 1-dimensional subpace
$\mathbb{V}^1$ in $\mathbb{V}$.

Restricting the standard representation of $G$ to the subgroup $P$, we
obtain the associated bundle $\cT= G \times_P\bV$. Since $\bV$
carries a representation of $G$, the map $G\times V\ni (g,v)\mapsto
(gP,g\cdot v)\in (G/P)\times \mathbb{V}$ descends to a canonical
trivialisation of $\cT= G \times_P\bV$. Thus $\cT$ has canonical
connection and this is the specialisation to $S$ of standard tractor
connection described (as a complex vector bundle) in \nn{tracconn}
above. In the flat homogeneous setting of $S=G/P$ this tractor
connection may be viewed as arising as an associated connection from the
Maurer--Cartan form $\omega_{\rm MC}$ on $G$.

On a general curved (hypersurface type) CR manifold $M$ one can
construct a principal bundle $\cG$ with fibre $P$, $P\to \cG\to M$, and
this is canonically equipped with a structure $\omega$, the {\em
  Cartan connection} \cite{CapSchichl2000,ChernMoser1974,Tanaka1976}. Indeed $\cG$ can be
recovered as an adapted frame bundle for the tractor bundle $\cT$ and
then the Cartan connection derived from the tractor connection of
\nn{tracconn}, see \cite{CapGover2002,CapGover2008}. The Cartan connection
should be viewed as a curved analgue of the Maurer--Cartan form, with
just weaker equivariance properties. Its characterising properties
are as follows. First, $\omega$ is a $\mathfrak{g}=\operatorname{Lie}(G)$ valued
1-form field on $\cG$ that provides a trivialisation of $T\cG$. Second, this
trivialisation is $P$-equivariant and reproduces the generators of
fundamental vector fields.  Finally, there is a notion of curvature for any
such Cartan connection and one requires that this satisfy a normalisation
condition defined in terms of Lie algebra cohomology. With these properties satisfied,  the pair
$(\cG,\omega)$ is uniquely determined up to isomorphism and is then
called the {\em normal} Cartan connection. The tractor connection
\nn{tracconn} is normal in this sense, in that the equivalent Cartan
connection is normal.

Given the Cartan bundle $\cG$ and any representation
of $P$, we may form associated vector bundles.
 For example, the tangent bundle is $\cG\times_P
\mathfrak{g}/\mathfrak{p}$, where $\mathfrak{p}=\operatorname{Lie}(P) $ and the
representation on $\mathfrak{g}/\mathfrak{p}$ is induced by the restriction to
$P$ of the adjoint representation. Although in general the Cartan
connection does not induce a linear connection on such associated
bundles, it does induce a connection on $\cG\times_P \mathbb{W}$ if
$\mathbb{W}$ is the restriction (to $P$) of a $G$-representation,
which we shall denote $\rho$.   A section $t\in
\Gamma\bigl( {\cG\times_P \mathbb{W}}\bigr)$ is represented by function
$\tilde{t}:\cG\to \mathbb{W}$ satisfying the equivariance property
$\tilde{t}(u\cdot q)=\rho(q^{-1})\tilde{t}(u)$, for all $u\in \cG$,
and $q\in P$.  The tractor connection is given by
\begin{equation}\label{trex}
\nabla_{Tp\cdot\xi}t(x)=\underline{u}\big(\xi\cdot\tilde t(u)+
\rho'(\omega(\xi))(\tilde t(u))\big),
\end{equation}
where $p:\cG\to M$ is the bundle map, $\xi\in T_u\Cal G$ is any
tangent vector, $\rho':\mathfrak{g}\to \operatorname{End}(\mathbb{W})$
denotes the representation of $\mathfrak{g}$ on $\mathbb{W}$, and
$\underline{u}\colon\bV\to \Cal T_x$, is the isomorphism from
$\mathbb{V}$ to the fiber of $\Cal T$ over $x\in M$ determined by $u$
(viewing $\cG$ as the adapted frame bundle for the standard tractor
bundle $\cT$).  Thus such bundles are called tractor bundles, and note that the
standard tractor bundle is induced from $\mathbb{V}$. So the
connection \nn{tracconn} induces a connection on any such tractor
bundle.


We conclude this subsection by discussing some other groups linked to the
geometry of the CR model.
Let $\mN \subset \mathbb{V}$ be the cone of non-zero null
vectors. The CR manifold of the model $S$ is the image of this under
complex projectivisation. We are interested also in the real
projectivisation.

Let $\MF$ be the space of all real rays in $\bV$ which are null for
the inner product $\langle\ ,\ \rangle_{\bR}$, the real part of
$\langle\ ,\ \rangle$. The space $\MF$ is a smooth hypersurface
in $\Cal P_{\bR_+}\bV \cong \mathbb{R}P^{2n+3}$, and we have an
obvious projection $\mN\to\MF$, which is a principal bundle
with fibre group $\bR_+$.

Any real null ray generates a complex null line containing it.  Thus there is a
smooth projection $\MF \to M$ which is a
fibre bundle over $M=S$,
with fibre the space $ S^1$ of real rays in $\bC$.

Let $\tilde G$ be  the
orthogonal group of $(\bV,\langle\ ,\ \rangle_{\bR})$, and let $
\PF \subset\tilde G$ be the stabiliser of a real null ray in
$\mathbb{V}$. In this case we observe that there is a transitive
action of $\tilde G$ on $\MF$ which leads to an identification
$\MF\cong\tilde G/\PF$. By construction, $\tilde G$ acts by
conformal isometries on $\MF$.  It is well known that this action
identifies $\tilde G/Z(\tilde G)$ with the group of conformal
isometries of $\MF$.

Now the subgroup $G\subset\tilde G$  acts
transitively on $\mN$, so it  also acts transitively on $\MF$.
Taking a real null ray and the complex null line generated by it as
the base points of $\MF$ and $M$, respectively, we see that
$G\cap\PF\subset P$, and $G\cap\PF$ is the stabiliser of a
real null line, so we also obtain the identification $\MF\cong
G/(G\cap\PF)$.  This is the model structure for the Fefferman
space that we describe below.  Again using that $G$ acts transitively
on $\mathcal{N}$ we see that $\mathcal{N}$ can be identified with
$G/(G\cap \tilde{Q} )$, where $\tilde{Q}$ is the subgroup of
$\PF$ fixing a nominated point in the ray defining $\PF$.

\subsection{The Fefferman space and ambient metric}\label{Fsec}

To a given CR manifold $M$ there are associated two equivalent
geometric structures, both due to Fefferman \cite{Fefferman1976};
these are the Fefferman space and the Fefferman ambient metric. The
latter associates to $M$ a K\"ahler manifold $(\MA,J^{\MA},h^{\MA})$
that is, in a suitable sense, approximately Ricci flat.
This also, by
construction, admits an action by a $\mathbb{C}^*$-parametrised family
of homotheties and is equipped with a distinguished real
hypersurface embedding $\iota:\cF\to \MA$ that is equivariant with
respect to this action. We identify $\cF$ with its image in
$\MA$ and note that (as will become clear) it is a generalising
analogue of the cone $\mN$ (of non-zero null vectors in
$\mathbb{V}$) described above in connection with the CR model.

Considering the $\mathbb{C}^*=\mathbb{C}\setminus \{0 \}$-action on
$\cF$, the orbit space $\cF/\mathbb{C}^*$ is naturally identified with
$M$. That is, $M=\cF/\mathbb{C}^*$ and we write $\pi_\cF:\cF\to M$ for
the natural quotient map.  To handle the link between tensorial
structures along $\cF$, in $\MA$, and the corresponding objects on
$M$, we use that the tractor connection on $M$ can be recovered from
the Levi-Civita connection on $\MA$ as follows.  For $x\in M$, we
denote by $\cF_x$ the fiber of $\cF$ over $x$, and we view $\cF_x$ as
a 2-dimensional submanifold of $\MA$ via $\cF_x\subset \cF \subset
\MA$.

Then $T(\MA)\,\big{|}_{\cF_x}$ denotes the tangent bundle to $\MA$
restricted to the submanifold $\cF_x$, a (real) rank $2n+4$ vector
bundle over $\cF_x$. Again using the construction of the ambient
manifold, it follows that the restriction of the ambient Levi-Civita
connection is flat without holonomy; so $T(\MA)\,\big{|}_{\cF_x}$ may
be globally trivialised by parallel sections.

The standard tractor bundle on $M$ may be realised as the complex rank
$n+2$  vector bundle $\cT\rightarrow M$ with fibre
\begin{equation}\label{tracdef}
\cT_x=\left\{U\in\Gamma\big(T^{1,0}(\MA)\,\big{|}_{\cF_x}\big):
\nabla^\mathcal{A}_v U=0 \,\, \mbox{for all} \,\, v\in \Gamma(T\cF) \,\, \mbox{vertical} \right\}.
\end{equation}
Here $v\in \Gamma(T\cF)$ vertical means that $v$ is a generator of the
$\mathbb{C}^*$ action.  Thus a section of $\cT$ on $M$ is a vector
field in $\MA$, defined along $\cF$, which is constant in the vertical
directions. It is easily verified that the Hermitian metric and
Levi-Civita connection on $\MA$ induce a Hermition metric and
connection on $\cT$.

\begin{thm}\label{trthm}
The metric and connection on $\cT$, induced by the  metric and
connection on $\MA$ agree, up to isomorphism, with the standard
(normal) CR tractor metric and connection introduced above.
\end{thm}

This theorem may be proved by analogy with the treatments of
corresponding conformal results, as in
\cite{CapGover2003,GoverPeterson2003}. However we can recover the
entire picture, and exploit results in the existing literature, by
using the Fefferman space as an intermediate step.

 Writing $\mathbb{R}_+$ to denote a chosen ray
in $\mathbb{C}^*$, we can view $\mathbb{C}^*$ as a
direct group product $\mathbb{C}^*=S^1\times \mathbb{R}_+$. Then the
map $\pi_\cF:\cF\to M$ factors into the composition of
$\pi_{\MF}:\MF\to M$ with $\pi:\cF\to \MF$, where
$\MF=\cF/\mathbb{R}_+$ is the $\mathbb{R}_+$ orbit space of
$\cF$. Then $\pi_{\MF}:\MF\to M$ has fibre $S^1$ and $\MF$ is the
Fefferman space; this has a canonical conformal structure induced by
the CR structure on $M$. In fact, $M_\cF$ is very easy to construct
directly via the Cartan and tractor bundle on $M$. This provides a
nice conceptual picture, but also a route to the recovery of the
ambient metric, the proof of Theorem \ref{trthm}, and other related
identifications that we need.

\subsection{The Fefferman space}\label{Fbasic}

\newcommand{\Om}{\Omega}
\newcommand{\cc}{\boldsymbol{c}}
\newcommand{\tct}{{\tilde{\Cal T}}}
\newcommand{\tM}{{\tilde M}}
\newcommand{\Rp}{\mathbb{R}_{+}}

Here we start again on the CR manifold $M$ and build a Fefferman space
and then later the ambient metric directly. Our treatment is brief
since for the first part of our construction further details may be
found in \cite{CapGover2008}, while for the second part mainly
similiar ideas are used.  Related earlier constructions of the
Fefferman space exist in \cite{BurnsDiederichShnider1977,Kuranishi1997,Lee1986}. In
our treatment we use the groups defined in Section \ref{Cartan}
above.

On a CR manifold $M$, recall that we write $\ce(-1,0)$ for the dual of
$\ce(1,0)$ (the chosen $(n+2)$nd root of the anticanonical bundle).
We define the \textit{Fefferman space} $\MF$ of $M$ to mean the
space of real rays in $\ce(-1,0)$ constructed as follows. Let $\Cal
F$ be (the total space of) the bundle obtained by removing the zero section in
$\ce(-1,0)$. There is a free right action of $\bC^*$ on $\Cal F$
which is transitive on each fibre. Restricting this action to the
subgroup $\Rp$, we define $\MF$ to be the quotient $\Cal
F/\Rp$. Hence $\pi_{\MF}:\MF\to M$ is a principal fibre bundle with
structure group $\bC^*/\Rp \cong S^1$.

Via the bundle inclusion $\ce(-1,0)\to \cT$ we see that we may
identify the total space $\ce(-1,0)$ with $\Cal G\x_P\bV^1$.  By construction, we can
therefore view $\MF$ as the associated fibre bundle $\Cal G\x_P\Cal
P_{\Rp}(\mathbb{V}^1)$   with fibre
the space of real rays in $\bV^1$.  Since $G$ acts transitively on
the cone of nonzero null vectors, $P$ acts transitively on the space
of real rays in $\bV^1$; the stabiliser of one of these rays is
$G\cap\PF$ and the stabiliser of a point in that ray $G\cap
\tilde Q$,  whence $\Cal P_{\Rp}\bV^1\cong
P/(G\cap\PF)$ and $\bV^1\cong P/(G\cap\tilde Q)$. Now $\MF=\Cal
G\x_P(P/(G\cap\PF))$ and $\cF=\Cal G\x_P(P/(G\cap\tilde Q))$ are naturally identified with the orbit spaces $\Cal G/(G\cap\PF)$ and $\Cal G/(G\cap\tilde Q)$, respectively. Hence we can view $\Cal G$ as a principal
bundle over $\MF$ with structure group $G\cap\PF$ and, alternatively, as
a   principal
bundle over $\cF$ with structure group $G\cap\tilde Q$.


\newcommand{\cW}{\mathcal{W}} Now for any closed Lie subgroup
$H\subset P$ we have the following observations. As for the cases just
described we have a manifold $\cG/H$. The bundle $\cG$ is a principal
bundle over this with fibre $P/H$, and the normal CR Cartan connection
$\omega \in \Omega^1(\cG,\mathfrak{g})$ also provides a Cartan connection
on $\cG\to \cG/H$. This is clear: The property that $\omega$ gives a
trivialisation of $\cG$ is not dependent on the base; the
$P$-equivariance of this trivialisation restricts to $H$-equivariance;
and the fundamental vector fields for $\cG\to \cG/H$ have generators
in $\operatorname{Lie}(H)\subset \mathfrak{p}$ so $\omega$
provides the map to these generators simply by restriction.
Similarly any representation $\mathbb{W}$ of $P$ then
determines an associated bundle $\cG\times_H \mathbb{W}$ over $\cG/H$
that corresponds to the bundle $\cG\times_P \mathbb{W}$ over
$M=\cG/P$.  In particular this applies to the case that $\mathbb{W}$
is a $G$-representation, so corresponding to each tractor bundle $\cW$
on $M$ there is a corresponding tractor bundle $\cW_H$ on $\cG/H$ and
the Cartan connection induces a tractor connection on
$\cW_H$. Furthermore since sections of the tractor bundle $\cW$ on $M$
correspond to functions $\cG\to \mathbb{W}$ that are $P$-equivariant,
it follows at once from the explicit formula \nn{trex} that
these are the same as sections on $\cW_H$ that are
parallel in the vertical directions of $\cG/H\to M$.

Thinking of the cases that $H$ is $G\cap \PF$ or $G\cap
\tilde{Q}\subset G\cap \PF $, we may apply these results in particular
to the standard representation $\mathbb{V}$ of $G$ to obtain the
associated bundle $\cT_H:=\Cal G\x_{H}\bV\to\cG/H$. The Hermitian
inner product on $\bV$ is $G$-invariant, so it gives rise to a
Hermitian bundle metric on $\cT_H$ of signature $(p+1,q+1)$.  Taking
the real part of this defines a real bundle metric $h_H$ of signature
$(2p+2,2q+2)$ on $\cT_H$.  The real ray $\bV_{\Rp}^1\subset\bV$,
stabilised by $G\cap\PF$ (or in the case of $H=G\cap \tilde{Q}$, the
point in $\bV_{\Rp}^1$) gives rise to an oriented real line subbundle
$\cT_H^1\subset\cT_H$ (respectively, this line subbundle with also a
nowhere zero distinguished section), and each of these lines
is null with respect to $h_H$.  Thus, defining $\cT_H^0$ to be the
real orthogonal complement of $\cT_H^1$, we obtain a filtration
$\cT_H=\cT_H^{-1}\supset\cT_H^0\supset\cT_H^1$ by smooth
subbundles. The real volume form on $\bV$ induces a trivialisation of
the highest real exterior power $\Lambda^{2n+4}\cT_H$.

Specialising to the case $H=G\cap \PF$ we come to the following. We write
$\cT_{\MF}$, $h_{\MF}$ rather than $\cT_H$, $h_H$, etc.

\begin{thm}[\cite{CapGover2008}] \label{Fspaces}
Let $(M,H,\ce(1,0))$ be a CR geometry of signature $(p,q)$. The
corresponding Fefferman space $\MF$ canonically carries a
conformal structure $\cc_{\MF}$ of signature $(2p+1,2q+1)$.

 The Cartan connection $\omega$ on $\Cal G$ induces a tractor connection
  $\nabla^{\cT_{\MF}}$ on the bundle $\cT_{\MF}\to \MF$, and $(\cT_{\MF},\cT_{\MF}^1,h_{\MF},\nabla^\cT_{\MF})$ is a standard tractor bundle for the natural
  conformal structure on $\MF$. The tractor connection $\nabla^\cT_{\MF}$ is
  normal.
\end{thm}
\begin{proof} It is staightforward to identify the quotient bundle
$\cT^0_{\MF}/\cT^1_{\MF}$ with a weighted twisting of $T\MF$. Using this,
  the conformal metric is then seen to be induced by the tractor
  metric, as in the usual conformal case. For more details (with a
  slightly different approach) see \cite[Theorem 2.1]{CapGover2008}.  The
  identification of $\cT_{\MF}$ and $\nabla_{\MF}$ with the usual
  conformal tractor bundle and connection follows from the
  characterisation of the latter (see \cite{CapGover2002}), or is proved in
  detail and directly in \cite[Theorem 2.3]{CapGover2008}.
\end{proof}

\begin{remark}
 In fact the conformal structure on $\MF$ carries a canonical spin
structure, but we do not need that here.
\end{remark}

Finally, we note that conformal tractor bundle
$(\cT_{\MF},\nabla^\cT_{\MF})$ inherits a complex structure
$\mathbb{J}_{\MF}$ corresponding to multiplication by $i$ on the defining
representation $\mathbb{V}$ of $G$. This is parallel for the tractor
connection because $G$ is complex linear and the Cartan connection is
$\mathfrak{g}$-valued. We may complexify $\cT$ and identify $\cT_{\MF}$ with the part $\cT^{1,0}_{\MF}$ in
$\mathbb{C}\otimes \cT_{\MF}$.

\newcommand{\R}{\mathbb{R}}
\newcommand{\gt}{g^{\mbox{\tiny{$M_\mathcal{A}$}}}}
\newcommand{\cGt}{\widetilde{\mathcal{G}}}
\newcommand{\cFt}{\widetilde{\mathcal{F}}}
\newcommand{\nt}{{\nabla^{\mA}}}

\subsection{The ambient metric} We
now specialise to the case $H=G\cap \tilde{Q}$. Then $\cG / H=\cF$ and
so over this we have the bundle $\cT_\cF=\cT_H$ and this is equipped
with the metric $h_\cF= h_H$ and a connection $\nabla^\cF$ preserving
$h_\cF$. From the above we see that the sections of $\cT_\cF\to \cF$ that
are parallel along the submanifolds generated by the $\mathbb{C}^*$
action may be identified with sections of $\cT\to M$ and
conversely. Similarly one sees that sections of $\cT_\cF\to \cF$
that are parallel in the vertical directions of $\cF\to \MF$ (which
coincide with the directions of the $\mathbb{R}_+$ action) are the
same as sections of $\cT_{\cF}\to \MF$.

Now it is easily verified that, at each point $p$ of $\cF$,
$\cT^0_{\cF}$ can be naturally identified with $T_p\cF\cong
T_p^{1,0}\cF \subset T_\mathbb{C}\cF$, and $\cT^1_{\cF}$ with the vertical
subspace (with respect to $\pi:\cF\to \MF$) therein.  We write ${\bf
  g}$ for the restriction of $h_\cF$ to $T\cF\cong \cT_\cF^0$; i.e.\ ${\bf g}$ on $\cF$ is defined for $X$, $Y\in T_p\cF$ by ${\bf
  g}(X,Y)=h^\cF_x(X,Y)$.


Using these observations we have that the quotient space
$\cT^0_{\cF}/\cT^1_{\cF}$, at $p$, is naturally identified with
$T_{\pi(p)}\MF$. So each point $p$ of $\cF$ determines a metric
$g_{\pi(p)}$ (from the conformal class of the Fefferman space) by
lifting tangent vectors in $T_{\pi(p)}\MF$ to vectors in $T_p\cF$ and
evaluating using the pairing ${\bf g}$.  The result is independent of
the choice of lift because the subbundle $\cT^1_{\cF}$ is orthogonal
to $\cT^0_\cF$ (as mentioned above). For $x\in \MF$, distinct points
of the fibre $\pi^{-1}(x)$ determine distinct metrics on $T_x\MF$
interwining the $\mathbb{R}_+$ action. Thus $\cF$ may be identified
with the natural ray bundle
$$
\{(x,g_x):x\in \MF, g\in \cc_{\MF} \}\subset S^2T^*\MF
$$ of conformally related metrics over $\MF$; so metrics in the
conformal class $\cc_{\MF}$ are sections of the metric bundle $\cF$.  Let
$\delta_s:\cF\rightarrow \cF$ denote the dilations defined by
$\delta_s(x,g_x)=(x,s^2g_x)$, $s>0$, and let
$T=\frac{d}{ds}\delta_s|_{s=1}$. So $T$ is the infinitesimal generator
of the dilations.

\newcommand{\cA}{\mathcal{A}}

Now the ambient space is constructed as follows.
Regard $\cF$ as a hypersurface in $M_\cA=\cF\times \R$ via
$\iota(z)=(z,0)$, $z\in\cF$.  The variable in the $\R$ factor is denoted
$\rho$. In the language of \cite{FeffermanGraham2012},
a straight pre-ambient metric
for $(\MF,\cc_{\MF})$ is a smooth metric $\gt$ of signature $(p+1,q+1)$ on a
dilation-invariant neighborhood $\cFt$ of $\cF$ satisfying
\begin{enumerate}
\item[(1)] $\delta_s^* \gt =s^2 \gt\quad$ for $s>0$;
\item[(2)] $\iota^* \gt={\bf g}$;
\item[(3)] $\nt T =Id$, where $Id$ denotes the identity endomorphism and
$\nt$ the Levi-Civita connection of $\gt$.
\end{enumerate}

Now if $S_{IJ}$ is a symmetric 2-tensor field on an open neighborhood
of $\cF$ in $\cF \times \R$ and $m \geq 0$, we write $S_{IJ} =
O^+_{IJ}( \rho^m)$ if $S_{IJ} = O(\rho^m)$ and for each point $p\in
\cF$, the symmetric 2-tensor $(\iota^*(\rho^{-m}S))(p)$ is of the form
$\pi^*t$ for some symmetric 2-tensor $t$ at $x=\pi(p)\in \MF$
satisfying $\operatorname{tr}_{g_x} t = 0$.
Since the dimension
$2n+2$ of $\MF$ is even, an {\em ambient metric} for $(\MF,\cc_{\MF})$
is a straight pre-ambient metric $\gt$ such that $\Ric(\gt)=O^+_{IJ}(
\rho^{n})$.  From \cite{FeffermanGraham2012}, there exists an ambient
metric for $(\MF,\cc_{\MF})$ and it is unique up to addition of a term
which is $O^+_{IJ}( \rho^{n+1})$ and up to pullback by a
diffeomorphism defined on a dilation-invariant neighborhood of $\cF$
which commutes with dilations and which restricts to the identity on
$\cF$.  Since $\MF$ is a Fefferman space, an invariant natural density
obstructs the existence of smooth solutions to
$\Ric(\gt)=O(\rho^{n+1})$.

Next we note that the restriction of the ambient Levi-Civita
connection $\nabla^\cA$ to $T\MA|_\cF$ agrees (up to isomorphism) with
the normal tractor connection on $\cT_\cF$. This follows by combining
the results in \cite{CapGover2003} (and see also
\cite{GoverPeterson2003}) with Theorem \ref{Fspaces}.  The former show
that the usual conformal tractor connection is induced by the
Fefferman--Graham ambient connection. By the uniqueness of the
normal conformal tractor connection, it then follows from Theorem
\ref{Fspaces} that the conformal tractor connection on $\MF$
determines the usual CR tractor connection on $M$. It is easily
verified that both steps are compatible with our claim above Theorem
\ref{trthm} for the way in which the tractor bundle arises from
$T\MA|_\cF$ and that the CR tractor connection is then induced in the
obvious way from the ambient connection.

Putting the above together we see that Theorem \ref{trthm} is proved.


Finally, we claim that $(M_{\cA}, g_{\cA})$ can be taken to be K\"ahler.  Indeed, since $M$ is embedded, Fefferman's original construction (cf.\ \cite{Hirachi2013,HirachiMarugameMatsumoto2015}) produces an ambient metric which is K\"ahler.  By the uniqueness of ambient metrics explained above, this latter metric can be taken to be $(M_{\cA},g_{\cA})$.

\subsection{Tractors via the ambient metric}
Theorem \ref{trthm} and its proof show that any section $V$ of the
standard tractor bundle on $M$ may be identified with the restriction
to $\cF$ of a vector field $\tilde{V}$ on $\MA$ that is parallel in
the vertical directions of $\MA\supset \cF\to M$.  These directions
are generated by the $\mathbb{C}^*$ action on $\cF$ mentioned above.
 Taking tensor
powers we may assume that any tractor field on $M$ arises from a
tensor field on $\MA$ that is parallel along the $\mathbb{C}^*$-orbits of $\mF$.

Now in view of our treatment of Theorem \ref{trthm} above, creating the
Rosetta stone relating tractor operators and similar objects to their
ambient equivalents can be broken into  a two step process. First we
have the bijection between such objects on $M$ and their equivalents on the
Fefferman space $\MF$, then second we have the bijection between these
tractor tools on the Fefferman space and their equivalents on the
ambient manifold.  But the first part of this is treated in
\cite{CapGover2008}, see especially Sections 3.2 to 3.7. The second is
treated in \cite{CapGover2003}, \cite{GoverPeterson2003}, and
\cite{GoverPeterson2005}. Only some minor additional input is required
to specialise the latter to the case that conformal tractor bundle has
a parallel complex structure. In fact conformal manifolds with a
parallel complex structure are treated in Sections 4.4 to 4.7 of
\cite{CapGover2008} (cf.\ \cite{Leitner2007}).

Putting these results together we can simply read off the
correspondence between the canonical CR tractor fields and operators
and their (Fefferman CR) ambient metric counterparts.  For example, the
canonical tractor $Z^A$ giving the inclusion \nn{basic-incl}
corresponds to the $(1,0)$-part $E$ of the generator $\boldsymbol{X}$
of the standard $\mathbb{R}_+$-action on the ambient manifold. This
vector field, which we also denote by $Z^A$, provides
an {\em Euler operator} $E$ on the ambient manifold, in that if a
function $f:\MA\to \mathbb{C}$ is homogeneous of degree $w$, with
respect to the $\mathbb{C}^*$ action on $\mF$, then $E\cdot f=
Z^A\nabla_A f=w f$ along $\mF$. CR densities in $\Gamma(\ce(w,w'))$
correspond to functions $f$ on $\cF$ that satisfy
$Z^A\nabla^{\MA}_A f = w f$ and $\bar{E}\cdot
f=Z^{\bar{A}}\nabla^{\MA}_{\bar{A}} f =w' f$. We shall always extend
such functions to $\MA$ with the same homogeniety. 
Putting this
together with our earlier convention, weighted tractor fields on $M$
are associated with homogeneous tensor fields on $\MA$ that
satisfy the same transport equations.

The calculus of the conformal ambient metric is by now well-known from
\cite{FeffermanGraham2012,GJMS1992}; see also \cite{CapGover2003,GoverPeterson2003}. From this we
can read off useful identities for the case considered here. Since
$\boldsymbol{X}$ is a homothetic gradient field such that
$\nabla^{\MA}\boldsymbol{X}$ is the identity endomorphism field, we
have
\begin{equation}\label{ids}
\nabla^{\MA}_A Z^B =\delta^B_A \quad \mbox{ and } \quad \nabla^{\MA}_{\bar{A}} Z^B =0.
\end{equation}
Thus
\begin{equation}\label{wts}
[E,Z^A]=Z^A,  \quad \mbox{ and } \quad [E,Z^{\bar{A}}]=0,
\end{equation}
and
$$
Z^AR_{A\bar{B}}{}^C{}_D=0= Z^{\bar{B}}R_{A\bar{B}}{}^C{}_D.
$$
It also follows that
\begin{equation}\label{eqn:DeltaE}
[E,\nabla^{\MA}_A]=-\nabla^{\MA}_A, \quad [E,\nabla^{\MA}_{\bar{B}}]=0, \quad \mbox{and so}\quad [E,\Delta^{\MA}]=-\Delta^{\MA}.
\end{equation}
It also follows from the definition of $Z^A$ that $r=Z^AZ_A$
is a defining function for $\cF\subset\MA$, in that $\cF$ is the zero locus of $r$, and from \nn{ids} that
$\nabla^{\MA}_A r=Z_A$; in particular, $\nabla_Ar$ is non-vanishing along
$\cF$.

Putting all this together, it follows that the CR
double-D operators $\bD_{AB}$, $\bD_{A\bar{B}}$ on weighted tractor fields
correspond to the ambient operators
$$
D_{AB}=Z_B\nabla^{\MA}_A-Z_A \nabla^{\MA}_B \quad \mbox{and} \quad D_{A\bar{B}}=Z_{\bar{B}}\nabla^{\MA}_A-Z_A \nabla^{\MA}_{\bar{B}},
$$
respectively. From this it follows that the tractor-D operator
$\bD_A$ corresponds to
$$
D_A = (n+E+ \bar{E}+1)\nabla^{\MA}_A-Z_A\Delta^{\MA}
$$
on $\MA$.
 Note that these ambient operators are defined on all of $\MA$, but
strictly it is only along the hypersurface $\cF$ (and restricted to
homogeneous tensor fields) that they correspond precisely to the given
tractor operators. Along $\cF$ they each act tangentially, meaning
that they do not depend on how the given tensor field is smoothly extended off
$\cF$. This follows because $r$ is a defining function for $\cF$ and
$$
[D_{AB},r]=[D_{A\bar{B}},r]=0,
$$
while $ [D_{A},r]=r\cdot {\rm Op}$ for some differential operator ${\rm Op}$.

We need the following technical result.
\begin{lem}\label{logdec}
Let $\tilde\tau$ be a function on $\MA$ which is homogenous of bidegree
$(w,w')$ on $\MA$ and let $\tau $ denote the corresponding density on
$M$. Then
$$
[D_AD_{\bar{B}},\log \tilde{\tau}]|_\cF
$$
is a homogeneous linear differential operator along $\cF$ and so,
applied to functions homogeneous of degree $(0,0)$, determines a linear
differential operator from sections of $\mathcal{E}(0,0)$ (over $M$) to sections of the
tractor bundle $\cT_{A\bar{B}}(-1,-1)$. This tractor operator  agrees
with
$$
[\D_A\D_{\bar{B}},\log \tau]:\ce(0,0)\to \cT_{A\bar{B}}(-1,-1)
$$
calculated in any scale.

Similarly,
$$
[[D_AD_{\bar{B}},\log \tilde{\tau}],\log \tilde{\tau}] (1) |_\cF
$$
(where $1$ denotes the unit-valued constant function) is a
homogeneous tensor along $\cF$, so descends to a tractor field in
$\cT_{A\bar{B}}(-1,-1)$, and this agrees with
$$
[[\D_A\D_{\bar{B}},\log \tau],\log \tau] (1)
$$
as calculated in any scale.
\end{lem}
\begin{proof}
The first statement is immediate from the definition of the ambient
operators $D_{A}$ and $D_{\bar{B}}$. Note also that, from the
tangentiality of these operators, $ [D_AD_{\bar{B}},\log
  \tilde{\tau}]|_\cF $ is independent of how $\tilde{\tau}$ extends
off $\cF$. Then recall that, by definition, $\log \tau$ means the
log-density that corresponds to $\log \tilde{\tau}$.

Next by direct calculation one verifies that
$$
[D_{AB},\log \tilde{\tau}]|_\cF \quad \mbox{and}\quad  [D_{A\bar{B}},\log \tilde{\tau}]|_\cF
$$
are also homogeneous and correspond to the tractor commutators
$$
[\D_{AB},\log \tau] \quad \mbox{and}\quad  [\D_{A\bar{B}},\log \tau]|_\cF .
$$ But it is easily verified that there are algebraic formulae for
$\D_A$ and $\D_{\bar{B}}$ in terms of compositions of $\D_{AB}$ and
$\D_{A\bar{B}}$ and these are equivalent to corresponding formulae for
$D_A$ and $D_{\bar{B}}$ in terms of compositions of $D_{AB}$ and
$D_{A\bar{B}}$ on $\MA$; cf.\ \cite{CapGover2003} for the analogous conformal case. (Using the Leibniz property of $\D_{AB}$ and
$\D_{A\bar{B}}$, one sees that the formulae formulae for $\D_A$ and
$\D_{\bar{B}}$ in terms of $\D_{AB}$ and $\D_{A\bar{B}}$ give the
usual matrix formulae for $\D_A$ and $\D_{\bar{B}}$ when applied to
log densities.)  Thus $[D_{\bar{B}},\log \tilde{\tau}]|_\cF$ and
$[D_A,\log \tilde{\tau}]|_\cF$ are equivalent to well-defined fields
and these are $[\D_{\bar{B}},\log \tau]$ and $[D_A,\log \tau]$, as
calculated in any scale.

The final statement follows similarly.
\end{proof}

In the subsequent sections, we normally omit the superscript $\MA$
in the notation for ambient objects, relying instead on context.  Also,
we often identify tractor operators with the corresponding
tangential ambient space operators without comment.





%% file: pluri.tex
\section{CR pluriharmonic functions and (pseudo-)Einstein contact forms}
\label{sec:pluri}

\subsection{Log densities}
\label{subsec:pluri/log}

In order to study CR pluriharmonic functions via tractors, it is useful to introduce log densities (cf.\ \cite{GoverWaldron2011}).  Let $\cmC\subset T^\ast M$ be the $\bR_+$-bundle of positive elements of $H^\perp$.  Given $w\in\bR$, let $\mE(w)$ be the bundle associated to $\cmC$ via the representation $\lambda\mapsto\lambda^{2w}$ of $\bR_+$.  In particular, $\mE(w)$ can be identified with a real subbundle of $\mE(w,w)$ and $\mE(1)$ can be identified with $\cmC$.  Hence $\mE(w)$ is trivial as a vector bundle.  We likewise let $\mF(w)$ be the real line bundle induced by the log representations of $\bR_+$.  In particular, a section $\lambda\in\mF(w)$ is equivalent to a function $\underline{\lambda}\colon\cmC\to\bR$ with the equivariance property
\begin{equation}
\label{eqn:log-density_equivariant}
\underline{\lambda}(t^2\theta,p) = \underline{\lambda}(\theta,p) + 2w\log t .
\end{equation}
Note that if $\tau$ is a positive section of $\mE(w)$ and $\underline{\tau}$ is the corresponding equivariant section of $\cmC$, then the composition $\log\compose\underline{\tau}$ has the property~\eqref{eqn:log-density_equivariant}, and hence is equivalent to a section of $\mF(w)$.  We shall denote this section by $\log\tau$.  It is clear that a section of $\mF(1)$ is $\log\tau$ for some positive section $\tau\in\mE(1)$.  We define the operator $\nabla\colon\mF(1)\to T^\ast M$ by setting
\begin{equation}
 \label{eqn:log-density_connection}
 \nabla\log\tau = \tau^{-1}\nabla\tau .
\end{equation}
We can extend these definitions by complex linearity and thereby consider $\mF(w)$ as a complex bundle; in particular, we obtain the operator $\nabla\colon\mF(w)\to T_{\bC}^\ast M$.

The requirement that the weight operators $\bw$ and $\bw^\prime$ satisfy the Leibniz property means that the natural definition of the weight operators on log densities is such that for any $\lambda\in\mF(w_0)$, it holds that
\begin{equation}
\label{eqn:log-density_commutator}
\bw\lambda = w_0 = \bw^\prime\lambda .
\end{equation}
Equivalently, we have that $[\bw,\lambda]=w_0=[\bw^\prime,\lambda]$.

\subsection{CR pluriharmonic functions and tractors}
\label{subsec:pluri/pluri}

One use of log densities is to provide a rigorous method for carrying
out Branson's method of analytic continuation in the
dimension~\cite{Branson1995,BransonGover2008b,GoverWaldron2011}.  For
example, the derivation of the formula for the $P^\prime$-operator in
dimension three~\cite{CaseYang2012} proceeds by observing that the CR
Paneitz operator $P_4\colon\mE(w,w)\to\mE(w-2,w-2)$ can be written in
general dimension as $P_4=C-wR$ for $w=-\frac{n-1}{2}$ and
$C=4\nabla^\gamma(\nabla_\gamma\nabla_\beta+inA_{\beta\gamma})\nabla^\beta$;
note that $C$ annihilates CR pluriharmonic
functions~\cite{GrahamLee1988}.  Thus $-\frac{1}{w}P_4f=Rf$ makes
sense for any $f\in\mP$.  This expression in the case $n=1$,
corresponding to $w=0$, yields the operator $P_4^\prime=R$.  As we
explain below, by working with log densities and the tractor formula
for the CR Paneitz operator, this ``division by zero'' can be realized
through the commutator property~\eqref{eqn:log-density_commutator}.

To begin, we point out, as an immediate corollary of
Proposition~\ref{prop:ddbar-weight}, the tensor formula for the
operator $\D_A\D_{\bar B}$ on $\mE(0,0)$.  From this formula
we see that CR pluriharmonic functions are annihilated by $\D_A\D_{\bar B}$.
 A key point is that the
tractor formula for the critical CR GJMS operators always factors
through this operator; see Theorem~\ref{thm:general_formula}.  In
particular, $\D_A\D_{\bar B}$ acting on $\mE(0,0)$ is the tractor
formula for the CR Paneitz operator in dimension three.

\begin{lem}
\label{lem:dbard}
Let $(M^{2n+1},H)$ be a CR manifold and let $f\in\mE(0,0)$.  Given any scale $\sigma\bar\sigma\in\mE(1,1)$, it holds that
\begin{equation}
\label{eqn:dbard_matrix}
\D_A\D^Bf = \begin{pmatrix}
                   0 & 0 & 0 \\
                   -(n-1)P_\alpha(f) & n(n-1)B_\alpha{}^\beta(f) & 0 \\
                   \nabla^\gamma P_\gamma(f) + nP^{\gamma\bar\sigma}B_{\gamma\bar\sigma}(f) & -(n-1)P^\beta(f) & 0
                   \end{pmatrix} ,
\end{equation}
where
\begin{align*}
B_{\alpha\bar\beta}(f) & = \nabla_\alpha\nabla_{\bar\beta} f - \frac{1}{n}\nabla_\gamma\nabla^\gamma f\,h_{\alpha\bar\beta} \\
P_\alpha(f) & = \nabla_\alpha\nabla_\beta\nabla^\beta f + inA_{\alpha\beta}\nabla^\beta f .
\end{align*}
In particular, if $f\in\mP$, then $\D_A\D_{\bar B}f=0$.
\end{lem}

\begin{remark}
 When $n>1$, we readily see that $f\in\mP$ if and only if $f\in\ker \D_A\D_{\bar B}$.
\end{remark}

The fact that CR pluriharmonic functions lie in the kernel of
$\D_A\D_{\bar B}$ when restricted to $\mE(0,0)$ means that the
scale-dependent operator $K_{A\bar B}\colon\mP\to\mE_{A\bar B}(-1,-1)$
defined by
\begin{equation}
 \label{eqn:I_ABbar} K_{A\bar B}(f) = -\D_A\D_{\bar B}\left(f\log\sigma\bar\sigma\right)
\end{equation}
for a given choice of scale $\sigma\in\mE(1,0)$ is well-defined, as
follows.  We may regard Proposition~\ref{prop:ddbar-weight} as giving
a formula for $\D_A\D_{\bar B}$ acting on log densities by using the
fact $\mP\subset\ker \D_A\D_{\bar B}$ to write
\begin{equation}
 \label{eqn:I_ABbar_commutator}
 \D_A\D_{\bar B}\left(f\log\sigma\bar\sigma\right) = \left[ \D_A\D_{\bar B},\log\sigma\bar\sigma \right] f .
\end{equation}
By Lemma~\ref{logdec} below, the commutator $ \left[ \D_A\D_{\bar B}, \cdot
   \right]$ is well defined on log densities, and takes values in
 linear differential operators $\mathcal{E}(0,0)\to
 \cT_{A\bar{B}}(-1,-1)$.  Let
 $\theta=(\sigma\bar\sigma)^{-1}\btheta$ and note that, by~\eqref{eqn:log-density_commutator}, we have that  $[\bw,\log\sigma\bar\sigma]=1$. A
 tractor expression for $K_{A\bar B}(f)$ is readily derived using  Proposition~\ref{prop:ddbar-weight}; we give
 here the formula in the scale $\theta$.

\begin{lem}
\label{lem:dbard_logpluri}
Let $(M^{2n+1},H)$ be a CR manifold and let $\sigma\bar\sigma\in\mE(1,1)$ be a scale.  Given any $f\in\mP$, the function $K_{A\bar B}(f)$ is given by
\begin{equation}
\label{eqn:dbard_logpluri_matrix}
K_A{}^B(f) = \begin{pmatrix}
                 (n-1)\nabla_\gamma\nabla^\gamma f & -n(n-1)\nabla^\beta f & 0 \\
                 (n-1)C_\alpha(f) & (n-1)C_\alpha{}^\beta(f) & -n(n-1)\nabla_\alpha f \\
                 D(f) & (n-1)C^\beta(f) & (n-1)\nabla^\gamma\nabla_\gamma f
                 \end{pmatrix} ,
\end{equation}
where
\begin{align*}
 C_{\alpha\bar\beta}(f) & = -\frac{1}{n}h_{\alpha\bar\beta}\Delta_bf - nfP_{(\alpha\bar\beta)_0}, \\
 C_\alpha(f) & = -\frac{1}{n}\nabla_\alpha\Delta_bf + P\nabla_\alpha f + f(\nabla_\alpha P + nT_\alpha), \\
 D(f) & = \frac{1}{n}\Delta_b^2 f - 2\Imaginary\nabla^\beta\left(A_{\alpha\beta}\nabla^\alpha f\right) - 4\Real\nabla^\alpha\left(P\nabla_\alpha f\right) + \frac{n-1}{n}P\Delta_b f \\
  & \quad + 2\Real\left((\nabla_\alpha P+nT_\alpha)\nabla^\alpha f\right) - f\left(n\lv P_{(\alpha\bar\beta)_0}\rv^2+\nabla^\gamma(\nabla_\gamma P + nT_\gamma)\right) .
\end{align*}
In particular, if $\sigma\bar\sigma$ is a pseudo-Einstein scale, then $1\in\ker K_{A\bar B}$.
\end{lem}

\begin{remark}
 When $n>1$, we readily see that $1\in\ker K_{A\bar B}$ if and only if $\sigma\bar\sigma$ is a pseudo-Einstein scale.
\end{remark}

\begin{proof}
By definition, we have that $[\nabla_\alpha,\log\sigma\bar\sigma]=0$ in the scale $\sigma\bar\sigma$.  It then follows from~\eqref{eqn:log-density_commutator} and~\eqref{eqn:I_ABbar_commutator} that $K_{A\bar B}(f)$ arises as the negative of the coefficient of $\bw$ in Proposition~\ref{prop:ddbar-weight}.  This yields~\eqref{eqn:dbard_logpluri_matrix}.  Finally, $\sigma\bar\sigma$ is a pseudo-Einstein scale if and only if $P_{(\alpha\bar\beta)_0}=0$ and $\nabla_\alpha P+nT_\alpha=0$, from which the last claim readily follows.
\end{proof}

In the case $n=1$, Lemma~\ref{lem:dbard_logpluri} yields
$K_{A\bar B}(f)=Z_AZ_{\bar B}P_4^\prime f$ for all $f\in\mP$ and also
the the transformation formula for the $P^\prime$-operator.  The
corresponding result in general dimensions is described in
Section~\ref{sec:algorithm}.

The fact that constants lie in the kernel of $K_{A\bar B}$ when $\sigma\bar\sigma$ determines a pseudo-Einstein scale means that for such a scale, the tractor $I_{A\bar B}\in\mE_{A\bar B}(-1,-1)$ defined by
\begin{equation}
\label{eqn:I_AB}
I_{A\bar B} = \frac{1}{2}\D_A\D_{\bar B}\left(\left(\log \sigma\bar\sigma\right)^2\right)
\end{equation}
is well-defined.  Indeed, since $1\in\ker \D_A\D_{\bar B}\cap \ker
K_{A\bar B}$, we compute that
\begin{equation}
 \label{eqn:I_AB_commutator}
 \D_A\D_{\bar B}\left(\left(\log\sigma\bar\sigma\right)^2\right) = \bigl[[\D_A\D_{\bar B},\log\sigma\bar\sigma],\log\sigma\bar\sigma\bigr] .
\end{equation}
It follows at once from Lemma~\ref{logdec} below that $I_{A\bar B}$ is a well-defined tractor field for any
pseudo-Einstein scale; on a fixed CR manifold this tractor field is determined entirely by
$\theta$.
A tractor expression for $I_{A\bar B}$ is
readily derived; we give here the formula in the scale $\theta$.

\begin{lem}
\label{lem:dbard_log2}
Let $(M^{2n+1},H)$ be a CR manifold and let $\sigma\bar\sigma\in\mE(1,1)$ be a pseudo-Einstein scale.  Then $I_{A\bar B}$ is given by
\begin{equation}
\label{eqn:dbard_log2_matrix}
I_A{}^B = \begin{pmatrix}
              -(n-1)P & 0 & n(n-1) \\
              \frac{2(n-1)}{n}\nabla_\alpha P & \frac{2(n-1)}{n}Ph_\alpha{}^\beta & 0 \\
              -\frac{2}{n}\Delta_bP - \lv A_{\alpha\beta}\rv^2 + \frac{n+3}{n}P^2 & \frac{2(n-1)}{n}\nabla_{\bar\beta}P & -(n-1)P
              \end{pmatrix} .
\end{equation}
\end{lem}

\begin{proof}

 Note that~\eqref{eqn:log-density_commutator} implies $[\bw^2,\log^2\sigma\bar\sigma]=2$.  Since $[\nabla_\alpha,\log\sigma\bar\sigma]=0$ and $\sigma\bar\sigma$ is a pseudo-Einstein scale, we thus need only consider the coefficient of $\bw^2$ in Proposition~\ref{prop:ddbar-weight}.
\end{proof}

In the case $n=1$, Lemma~\ref{lem:dbard_log2} yields $I_{A\bar B}=Z_AZ_{\bar B}Q^\prime$ and also the transformation formula for the $Q^\prime$-curvature.  The corresponding result in general dimensions is described in Section~\ref{sec:algorithm}.

Finally, let us comment on our normalisations.  Suppose that $\sigma,s\in\mE(1,0)$ are two scales and $s\bar s=e^{-\Upsilon}\sigma\bar\sigma$; thus the contact forms $\theta=(\sigma\bar\sigma)^{-1}\btheta$ and $\htheta=(s\bar s)^{-1}\btheta$ are related by $\htheta=e^\Upsilon\theta$.  Let $K_{A\bar B},\widehat{K}_{A\bar B}\colon\mP\to\mE_{A\bar B}(-1,-1)$ be the operators defined in terms of $\sigma$ and $s$, respectively, by~\eqref{eqn:I_ABbar}.  It follows immediately from~\eqref{eqn:I_ABbar} that
\begin{equation}
 \label{eqn:I_ABbar_transformation}
\widehat{K}_{A\bar B}(f) = K_{A\bar B}(f) + \bD_A\bD_{\bar B}(\Upsilon f) .
\end{equation}
In particular, this normalisation recovers the familiar transformation formula~\eqref{eqn:pprime_trans_intro} for the $P^\prime$-operator in dimension three.

Suppose additionally that $\theta$ and $\htheta$ are both
pseudo-Einstein.  Then $\Upsilon\in\mP$ and the tractors $I_{A\bar B}$
and $I_{A\bar B}$ defined in terms of $\sigma$ and $s$, respectively,
by~\eqref{eqn:I_AB} are well-defined.  Moreover, \eqref{eqn:I_AB}
gives
\begin{equation}
 \label{eqn:I_AB_transformation}
 \hI_{A\bar B} = I_{A\bar B} + K_{A\bar B}(\Upsilon) + \frac{1}{2}\D_A\D_{\bar B}\left(\Upsilon^2\right) .
\end{equation}
In particular, this normalisation recovers the familiar transformation formula~\eqref{eqn:qprime_trans_intro} for the $Q^\prime$-curvature in dimension three.

\subsection{(Pseudo-)Einstein manifolds and (partially) parallel tractors}
\label{subsec:pluri/einstein}

From the tractor perspective, a natural reason to study pseudo-Einstein and Einstein structures on CR manifolds is that they correspond to the existence of holomorphic and parallel standard tractors, respectively (for the latter cf. \cite{CapGover2008,Leitner2007}).  Parallel tractors are especially useful; we use them in Section~\ref{sec:product} to derive simple local formulae for $P$, $P^\prime$ and $Q^\prime$ from tractor formulae in Einstein scales.

\begin{prop}
 \label{prop:einstein_holomorphic}
 Let $(M^{2n+1},H)$ be a CR manifold.  Suppose that $\theta$ is a pseudo-Einstein contact form.  Then locally there exists a $\sigma\in\mE(1,0)$, unique up to multiplication by a constant $\lambda\in\bC$ with $\lv\lambda\rv^2=1$, such that $\theta=(\sigma\bar\sigma)^{-1}\btheta$ and $\D_A\sigma$ is holomorphic; i.e.\ $\nabla_{\bar\beta}\D_A\sigma=0$.  Conversely, if $I_A\in\mE_A(0,0)$ is holomorphic, then $\theta=(\sigma\bar\sigma)^{-1}\btheta$ is pseudo-Einstein wherever $\sigma=Z^AI_A$ is nonzero.
\end{prop}

\begin{proof}
 Suppose that $\theta$ is a pseudo-Einstein contact form.  By~\cite[Lemma~7.2]{Hirachi1990} and~\cite[Theorem~4.2]{Lee1988}, locally there exists a closed form $\zeta\in\mK$ with respect to which $\theta$ is volume-normalized.  Let $\sigma\in\mE(1,0)$ be a $-(n+2)$-nd root of $\zeta$.  Then $\dbar_b\sigma=0$.  By~\cite[Proposition~2.4]{GoverGraham2005}, it holds that
 \[ -\frac{1}{n+1}\left(\nabla^\gamma\nabla_\gamma\sigma + i\nabla_0\sigma + \frac{n+1}{n+2}P\sigma\right) = -i\nabla_0\sigma + \frac{P}{n+2}\sigma . \]
 It is now straightforward to check that $\D_A\sigma$ is holomorphic.

 Conversely, suppose that $I_A$ is holomorphic and suppose that $\sigma=Z^AI_A$ is holomorphic.  Set $\zeta=\sigma^{-(n+2)}\in\mK$ and let $\theta$ be the unique contact form which is volume-normalized with respect to $\zeta$.  Since $I_A$ is holomorphic, $\zeta$ is closed, and hence $\theta$ is pseudo-Einstein~\cite{Hirachi1990,Lee1988}.
\end{proof}

\begin{prop}
 \label{prop:einstein_parallel}
 Let $(M^{2n+1},H)$ be a CR manifold.  Suppose that $\theta$ is an Einstein contact form.  Then locally there exists a $\sigma\in\mE(1,0)$, unique up to multiplication by a constant $\lambda\in\bC$ with $\lv\lambda\rv^2=1$, such that $\theta=(\sigma\bar\sigma)^{-1}\btheta$ and $\D_A\sigma$ is parallel; i.e.\ $\nabla_\beta \D_A\sigma=0$ and $\nabla_{\bar\beta}\D_A\sigma=0$.  Conversely, if $I_A\in\mE_A(0,0)$ is parallel, then $\theta=(\sigma\bar\sigma)^{-1}\btheta$ is Einstein wherever $\sigma:=Z^AI_A$ is nonzero.
\end{prop}

\begin{proof}
 Suppose that $\theta$ is Einstein.  Let $\sigma$ be as in Proposition~\ref{prop:einstein_holomorphic}, so that $I_A=\frac{1}{n+1}\D_A\sigma$ is holomorphic.  In the scale $\theta$, we have that $\lv\sigma\rv^2$ is parallel and hence, since $\dbar_b\sigma=0$, it holds that $d_b\sigma=0$.  It is then clear from~\eqref{tracconn} that $\nabla_\beta I_A=0$.

 Conversely, suppose that $I_A$ is parallel.  Set $\sigma=Z^AI_A$.  By Proposition~\ref{prop:einstein_holomorphic}, $\theta=(\sigma\bar\sigma)^{-1}\btheta$ is pseudo-Einstein.  Evaluating $\nabla_\beta I_A=0$ in the scale $\theta$ yields $A_{\alpha\beta}=0$, and hence $\theta$ is Einstein.
\end{proof}

%% file: algorithm.tex
\section{Tractor formulae for CR GJMS operators}
\label{sec:algorithm}

It is straightforward to compute from the definitions of $r$, $Z_A$ and $E$ that
\begin{equation}
 \label{eqn:Deltar_commutator}
 [\Delta,r]=n+E+\bar E+2 .
\end{equation}
It follows that, as an operator on ambient tensors of weight $\left(-\frac{n+3-k}{2},-\frac{n+3-k}{2}\right)$,
\begin{equation}
 \label{eqn:Deltar_commutator_tangential}
 [\Delta^k, r] = r\cdot\mathrm{Op} .
\end{equation}
Thus if $f\in\mE\left(-\frac{n+1-k}{2},-\frac{n+1-k}{2}\right)$, then
\begin{equation}
 \label{eqn:crgjms_defn}
 P_{2k}f = \left.(-2\Delta)^k\tilde f\right|_{\mQ}
\end{equation}
is independent of the choice of homogeneous extension $\tilde f$ of $f$ to $\MA$.  Let $k\leq n+1$.  Since the ambient metric is uniquely determined up to $O^+(\rho^{n+1})$, the operator~\eqref{eqn:crgjms_defn} is well-defined, and hence defines the CR GJMS operator of order $2k$.  Our goal in this section is to develop an algorithm which converts~\eqref{eqn:crgjms_defn} into a tractor formula.

To derive tractor formulae for the CR GJMS operators, we first need to derive some useful identities on the ambient manifold $\MA$.
In the following, $\Delta$ is the K\"{a}hler Laplacian
$\Delta=g^{A\bar{B}}\nabla_A\nabla_{\bar{B}}$
and $R_{A\bar{B}}{}^{\bar{C}}{}_{\bar{D}}$ is the $(1,1)$-part of the
ambient curvature. We write $R_{A\bar{B}}\sharp$ to denote the usual
action of this 2-form-valued endomorphism field on ambient
tensors. Since $g$ is K\"ahler, we have the operator equations on ambient tensors
\begin{align}
 \label{eqn:nablanabla} [\nabla_A,\nabla_B] & = O(r^{n-1}), \\
 \label{eqn:nablanablabar} [\nabla_A,\nabla_{\bar B}] & = R_{A\bar B}\hash;
\end{align}
e.g.\ on a vector field $T^A$ we have
$[\nabla_A,\nabla_{\bar{B}}]T^C=-R_{A\bar{B}}{}^C{}_DT^D$.  With these conventions we obtain the following:

\begin{lem}
 \label{lem:general_commutators}
 As operators on arbitrary ambient tensors,
 \begin{align}
  \label{eqn:DeltaZ} \left[\Delta, Z_{\bar{B}}\right] & = \nabla_{\bar{B}} , \quad \left[\Delta, Z_{A}\right]  = \nabla_{A} \\
  \label{eqn:Deltanabla} \left[\Delta,\nabla_A\right] & = -R_A{}^C\hash\nabla_C +O(r^{n-1}), \quad  \left[\Delta,\nabla_{\bar B}\right]  = R_{C\bar B }\hash\nabla^C + O(r^{n-1}) , \\
  \label{eqn:DeltaZZbar} \left[\Delta,Z_AZ_{\bar B}\right] & = Z_A\nabla_{\bar B} + Z_{\bar B}\nabla_A + H_{A\bar B}, \\
  \label{eqn:DeltaZnablabar}[\Delta,Z_{\bar{B}}\nabla_A] &=
  \nabla_{\bar{B}}\nabla_A- Z_{\bar{B}}R_A{}^C\hash\nabla_C +O(r^{n-1}),\\
  \label{eqn:Deltanablabarnabla}  \left[\Delta,\nabla_A\nabla_{\bar B}\right] & =- R_A{}^C\hash\nabla_C \nabla_{\bar{B}} + \nabla_{A}R_{C\bar{B}}\hash\nabla^C + O(r^{n-2}).
 \end{align}
 Moreover, as operators on ambient functions,
 \begin{align}
  \label{eqn:Deltanablafn} \left[\Delta,\nabla_A\right] & = O(r^{n}), \qquad  \left[\Delta,\nabla_{\bar B}\right]  = O(r^{n}) , \\
  \label{eqn:Deltanablabarnablafn} \left[\Delta,\nabla_A\nabla_{\bar B}\right] & = -R_A{}^C\hash\nabla_C\nabla_{\bar B} + O(r^{n-1}) .
 \end{align}
\end{lem}

\begin{proof}
Using~\eqref{ids} 
we compute that
\[ \left[\Delta,Z_A\right] = \left[\nabla_C,Z_A\right]\nabla^C + \nabla_C\left[\nabla^C,Z_A\right] = \nabla_A , \]
and the other result follows by conjugation.
Using~\eqref{eqn:nablanabla} and~\eqref{eqn:nablanablabar}, we compute that
\[ \left[\Delta,\nabla_A\right] = \left[\nabla_C,\nabla_A\right]\nabla^C + \nabla_C\left[\nabla^C,\nabla_A\right] = -\nabla_C R_A{}^C\hash = -R_A{}^C\hash\nabla_C + O(r^{n-1}) . \]
where the last equality uses $\Ric=O(r^n)$.  Moreover, since $R_A{}^C\hash$ annihilates functions, we see that $[\Delta,\nabla_A]$ also annihilates functions.  Similarly,
\[ \left[\Delta,\nabla_{\bar B}\right] = R_{C\bar B}\hash\nabla^C, \]
and hence on functions,
\[ \left[\Delta,\nabla_{\bar B}\right] = -R_{C\bar B}\nabla^C = O(r^n) . \]

Using~\eqref{ids} and~\eqref{eqn:DeltaZ}
 we compute that
\[ \left[\Delta,Z_AZ_{\bar B}\right] = \nabla_AZ_{\bar B} + Z_A\nabla_{\bar B} = Z_{\bar B}\nabla_A + Z_A\nabla_{\bar B} + H_{A\bar B} . \]
Using~\eqref{eqn:DeltaZ} and~\eqref{eqn:Deltanabla} immediately yields \eqref{eqn:DeltaZnablabar}.
Finally, using both parts of~\eqref{eqn:Deltanabla} (resp.\ of~\eqref{eqn:Deltanablafn} on functions) gives~\eqref{eqn:Deltanablabarnabla} (resp.\ \eqref{eqn:Deltanablabarnablafn} on functions).
\end{proof}

The main step in deriving a tractor formula for the CR GJMS operators is the following formula for the difference between $\Delta^kD_AD_{\bar B}$ and $Z_AZ_{\bar B}\Delta^{k+2}$ as operators on ambient tensors.  As in Lemma~\ref{lem:general_commutators}, we only identify this difference along the hypersurface $\cF$ and identify the order of the error.  To simplify the exposition, we do not keep track of the order of the error throughout the proof, but only record the final result; one counts the order of the error terms by using Lemma~\ref{lem:general_commutators}, counting additional derivatives, and using the identity~\eqref{eqn:Deltar_commutator}.

\begin{prop}
 \label{prop:general_formula}
 Modulo the addition of terms $O(r^{n-k-1})$ the following holds:
\begin{equation}
\label{eqn:general_formula}
 \begin{split}
  \Delta^kD_A D_{\bar B} & = Z_A Z_{\bar B}\Delta^{k+2} \\
& \quad - (n+E+\bar E+k+1)(Z_{\bar B}\nabla_A\Delta^{k+1} + Z_A\nabla_{\bar B}\Delta^{k+1} + H_{A\bar B}\Delta^{k+1}) \\
& \quad + (n+E+\bar E+k+1)(n+E+\bar E+k+2)\nabla_A \nabla_{\bar B}\Delta^k \\
& \quad - Z_{A}R_{C\bar{B}}\hash (\Delta \delta_E{}^C-R_E{}^C\hash)^kD^E \\
& \quad + \sum_{j=0}^{k-1}(j+1)\nabla_{A}\Delta^{k-1-j}R_{C\bar{B}}\hash (\Delta \delta_E{}^C - R_E{}^C\hash)^jD^E \\
& \quad - \sum_{j=0}^{k-1}(\Delta \delta_A{}^{C}+R_A{}^C\hash)^j R_C{}^E \hash \Delta^{k-1-j}D_E D_{\bar{B}} \\
& \quad + \sum_{j=0}^{k-1}(\Delta \delta_A{}^C+R_A{}^C\hash)^jD_CR_{E\bar{B}}\hash(\Delta \delta_F{}^E - R_F{}^E\hash)^{k-1-j}D^F .
\end{split}
\end{equation}
Here the left and right hand side are operators acting on arbitrary ambient tensors.  Moreover, when acting on ambient functions, the error term in~\eqref{eqn:general_formula} is $O(r^{n-k})$.
\end{prop}

\begin{proof}

The proof is by induction.  To begin, we compute that
\begin{align*}
D_A D_{\bar B} & = \left((n+E+\bar E+1)\nabla_A - Z_A\Delta\right)\left((n+E+\bar E+1)\nabla_{\bar B} - Z_{\bar B}\Delta\right)\\
& = Z_{\bar B}Z_A\Delta^2 - (n+E+\bar E+1)(Z_{\bar B}\nabla_A\Delta + Z_A\nabla_{\bar B}\Delta + H_{A\bar B}\Delta) \\
& \quad + (n+E+\bar E+1)(n+E+\bar E+2)\nabla_A\nabla_{\bar B}\\
& \quad - (n+E+\bar E+2)Z_{A}R_{C\bar{B}}\hash\nabla^C
\end{align*}
where the second equality uses Lemma~\ref{lem:general_commutators}; for the purposes of determining the order of vanishing of the error, note that the only commutator involving errors which was evaluated in the above derivation is $[\Delta,\nabla_{\bar B}]$.
Using the
definition of (the ambient) $D^A$ and the fact $R_{C\bar{B}}\hash
Z^C=0$, we see that
\[ (n+E+\bar E+2)Z_{A}R_{C\bar{B}}\hash\nabla^C = Z_{\bar B}R_{C\bar{B}}\hash D^C . \]
Together, these two displays yield the case $k=0$.

Suppose now that~\eqref{eqn:general_formula} holds.  We want to compute $\Delta^{k+1}D_AD_{\bar B}$.  To that end, we make a number of observations.  Note that, by Lemma~\ref{lem:general_commutators}, the formulae derived below all hold up to terms of order $O(r^{n-2})$, and up to terms of order $O(r^{n-1})$ on functions.  Thus the lowest order error comes from commuting the Laplacian through the error term in $\Delta^kD_AD_{\bar B}$, accounting for the loss of a power of $r$ in the order of the error.

First, as an immediate consequence of~\eqref{eqn:DeltaZZbar}, for any $k$ it holds that
\begin{equation}
\label{eqn:gf_obs1}
\Delta Z_AZ_{\bar B}\Delta^{k+2} = Z_AZ_{\bar B}\Delta^{k+3} + \left(Z_{\bar B}\nabla_A + Z_A\nabla_{\bar B} + H_{A\bar B}\right)\Delta^{k+2} .
\end{equation}

Second, as an immediate consequence of~\eqref{eqn:nablanablabar}
and~\eqref{eqn:DeltaZnablabar}, for any $k$ it holds that
\begin{equation}
\label{eqn:gf_obs2}
\begin{split}
\Delta&(n+E+\bar E+k+1)(Z_{\bar B}\nabla_A + Z_A\nabla_{\bar B} + H_{A\bar B})\Delta^{k+1} \\
& = (n+E+\bar E+k+3)\bigl((Z_{\bar B}\nabla_A + Z_A\nabla_{\bar B} + H_{A\bar B})\Delta^{k+2} \\
& \qquad + (2\nabla_A\nabla_{\bar B} - R_{A\bar B}\hash + Z_A\nabla^CR_{C\bar B}\hash - Z_{\bar B}\nabla_CR_A{}^C\hash)\Delta^{k+1}\bigr) .
\end{split}
\end{equation}
Third, it holds for any $k$ that
\begin{equation}
\label{eqn:gf_obs3}
(n+E+\bar E+k+1)\nabla^B\Delta^k = Z^B\Delta^{k+1} + \left(\Delta \delta_C{}^B - R_C{}^B\hash\right)^kD^C .
\end{equation}
Indeed, \eqref{eqn:gf_obs3} trivially holds when $k=0$. (As done implicitly in  \eqref{eqn:general_formula}, we interpret $\left(\Delta \delta_C{}^B - R_C{}^B\hash\right)^k$ at $k=0$ to mean the identity endomorphism field, while for $k\geq 2$ there is an obvious abuse of the abstact index notation.)
Using first \eqref{eqn:Deltanabla} and  then  proceeding by induction, one  finds that
\begin{align*}
(n+ & E+\bar E+k+1)\nabla^B\Delta^k \\
& = \left(\Delta \delta_C{}^B - R_C{}^B\hash\right)\left(n+E+\bar E+k-1\right)\nabla^C\Delta^{k-1} \\
& = \left(\Delta \delta_C{}^B - R_C{}^B\hash\right)\left(Z^C\Delta^k - \nabla^C\Delta^{k-1} + \left(\Delta \delta_E{}^C - R_E{}^C\hash\right)^{k-1}D^E\right) \\
& = \Delta Z^B\Delta^k - \Delta\nabla^B\Delta^{k-1} + R_{C}{}^B\hash\nabla^C\Delta^{k-1} + \left(\Delta \delta_C{}^B - R_C{}^B\hash\right)^kD^C \\
& = Z^B\Delta^{k+1} + \left(\Delta \delta_C{}^B - R_C{}^B\hash\right)^k D^C .
\end{align*}

Fourth, as an immediate consequence of~\eqref{eqn:DeltaE},
\eqref{eqn:Deltanablabarnabla} and~\eqref{eqn:gf_obs3}, for any $k$ it
holds that
\begin{equation}
\label{eqn:gf_obs4}
\begin{split}
\Delta&(n+E+\bar E+k+1)(n+E+\bar E+k+2)\nabla_A\nabla_{\bar B}\Delta^k \\
& = (n+E+\bar E+k+3)(n+E+\bar E+k+4)\nabla_A\nabla_{\bar B}\Delta^{k+1} \\
& \quad + (n+E+\bar E+k+3)\nabla_{A}R_{C\bar{B}}\hash (\Delta \delta_E{}^C - R_E{}^C\hash)^k D^E \\
& \quad - R_A{}^C\hash(n+E+\bar E+k+1)(n+E+\bar E+k+2)\nabla_C\nabla_{\bar{B}}\Delta^k .
\end{split}
\end{equation}
Note that the last summand is easily rewritten in terms of $\Delta^kD_CD_{\bar{B}}$ using the inductive hypothesis~\eqref{eqn:general_formula}.

Fifth, as an immediate consequence of~\eqref{eqn:DeltaZ} and the definition of the tractor $D$ operator, it holds that
\begin{equation}
\label{eqn:gf_obs5}
\Delta Z_{A} R_{C\bar{B}}\hash = -D_{A}R_{C\bar{B}}\hash
+ (n+E+\bar E+2)\nabla_AR_{C\bar{B}}\hash .
\end{equation}

Sixth, as an immediate consequence of~\eqref{eqn:DeltaE}, for any $k$ it holds that
\begin{equation}
\label{eqn:gf_obs6}
\begin{split}
R_A{}^C\hash&(n+E+\bar E+k+1)\left(Z_C\nabla_{\bar{B}} + Z_{\bar{B}}\nabla_C + H_{C\bar{B}}\right)\Delta^{k+1} \\
& = \left(n+E+\bar E+k+3\right)\left(Z_{\bar{B}}R_A{}^C\hash\nabla_C + R_{A\bar B}\hash\right)\Delta^{k+1} .
\end{split}
\end{equation}

Seventh, as an immediate consequence of~\eqref{eqn:gf_obs3}, for any $k$ it holds that
\begin{equation}
\label{eqn:gf_obs7}
\begin{split}
\left(n+E+ \bar E+k+3\right) & Z_A\nabla^C R_{C\bar{B}}\hash\Delta^{k+1}\\
& = Z_AR_{C\bar{B}}\hash\left(\Delta \delta_E{}^C - R_E{}^C\hash\right)^{k+1}D^E
\end{split}
\end{equation}

Eighth, as an immediate consequence of~\eqref{eqn:Deltanabla}, it holds that
\begin{equation}
\label{eqn:gf_obs8}
\Delta\nabla_AR_{C\bar{B}}\hash = \nabla_A\Delta R_{C\bar{B}}\hash - R_A{}^E\hash\nabla_ER_{C\bar{B}}\hash .
\end{equation}

It is then straightforward to use the above eight observations and the
inductive hypothesis~\eqref{eqn:general_formula} to show that
$\Delta^{k+1}D_AD_{\bar B}$ is given as
in~\eqref{eqn:general_formula}.
\end{proof}

The reason for the improved order of vanishing of the error in~\eqref{eqn:general_formula} when considered as an operator on functions is that $R_{AC}\hash$ and $R_{A\bar B}\hash$ annihilate functions.  This observation also yields a simplification of~\eqref{eqn:general_formula} on functions.

\begin{cor}
 \label{cor:general_formula}
 Modulo the addition of terms $O(r^{n-k})$, the following holds on functions:
 \begin{align*}
  \Delta^kD_A D_{\bar B} & = Z_AZ_{\bar B}\Delta^{k+2}\\
& - (n+E+\bar E+k+1)(Z_{\bar B}\nabla_A\Delta^{k+1} + Z_A\nabla_{\bar B}\Delta^{k+1} + H_{A\bar B}\Delta^{k+1}) \\
& \quad + (n+E+\bar E+k+1)(n+E+\bar E+k+2)\nabla_A \nabla_{\bar B}\Delta^k \\
& \quad - \sum_{j=0}^{k-1}(\Delta \delta_A{}^{C}+R_A{}^C\hash)^j R_C{}^E \hash \Delta^{k-1-j}D_E D_{\bar{B}} .
 \end{align*}
 In particular, when acting on $\mE\left(-\frac{n-1-k}{2},-\frac{n-1-k}{2}\right)$,
 \begin{equation}
  \label{eqn:general_formula_homogeneous_functions}
  \Delta^kD_A D_{\bar B} = Z_AZ_{\bar B}\Delta^{k+2} -\sum_{j=0}^{k-1}(\Delta \delta_A{}^{C}+R_A{}^C\hash)^j R_C{}^E \hash \Delta^{k-1-j}D_E D_{\bar{B}},
 \end{equation}
modulo the addition of terms $O(r^{n-k})$.
\end{cor}
\begin{proof}
 Let $\cf\in\cmE(w,w^\prime)$.  Then $R_E{}^C\hash D^E\cf =
 R_E{}^C{}_F{}^E D^F\cf = O(r^n)$.  Using this observation in
 Proposition~\ref{prop:general_formula} yields the final
 result.
\end{proof}

The next step in deriving a tractor formula for the CR GJMS operators is to derive a tractor formula for the summation on the right-hand side of~\eqref{eqn:general_formula_homogeneous_functions}.  To that end, we first observe that the operator acting on $D_AD_{\bar B}$ in the right-most summand acts tangentially on $\mE^\star\left(-\frac{n+1-k}{2},-\frac{n+1-k}{2}\right)$, where we recall from Subsection~\ref{Dop} that $\mE^\star(w,w^\prime)$ denotes a weighted tractor bundle.

\begin{prop}
 \label{prop:tangential}
 As an operator on $\mE^\star\left(-\frac{n+1-k}{2},-\frac{n+1-k}{2}\right)$,
 \begin{equation}
  \label{eqn:algorithm_lot}
\sum_{j=0}^{k-1}(\Delta \delta_A{}^{C}+R_A{}^C\hash)^j R_C{}^E \hash \Delta^{k-1-j}
 \end{equation}
 acts tangentially.
\end{prop}

\begin{proof}
 We must show that the commutator of~\eqref{eqn:algorithm_lot} with
 $r$ is of the form $r\cdot\mathrm{Op}$, for some operator
 $\mathrm{Op}$, when applied to an element of
 $\mE^\star\left(-\frac{n+3-k}{2},-\frac{n+3-k}{2}\right)$.  This
 follows directly from the formula~\eqref{eqn:Deltar_commutator}.
\end{proof}

The key point of Proposition~\ref{prop:tangential}
is that it shows that~\eqref{eqn:algorithm_lot} is a tangential differential operator of order $k-1$.
Since we assume that $k\leq n+1$, the operators~\eqref{eqn:algorithm_lot}
are all subcritical.  In particular, we can adapt the arguments
from~\cite{Gover2006,GoverPeterson2003,GoverPeterson2005} to easily
produce tractor formulae for the
operators~\eqref{eqn:algorithm_lot}. (Here and below we identify
tractor operators with the corresponding tangential ambient
operators.)

\begin{prop}
 \label{prop:tangential_tractor}
 Suppose that $k\leq n-1$.  There is a tractor operator
 \[\textstyle \Phi_{A\bar E}\colon\mE^\star\left(-\frac{n+1-k}{2},-\frac{n+1-k}{2}\right)\to\mE_{A\bar E}\otimes\mE^\star\left(-\frac{n+1+k}{2},-\frac{n+1+k}{2}\right) \]
 such that
 \[ \Phi_A{}^E =\sum_{j=0}^{k-1}(\Delta \delta_A{}^{C}+R_A{}^C\hash)^j R_C{}^E \hash \Delta^{k-1-j}. \]
 Moreover, if $I^A$ is parallel, it holds that $I^{A}\Phi_A{}^E=0$.
\end{prop}

\begin{proof}
 Let $T\in\mE^\star\left(-\frac{n+1-k}{2},-\frac{n+1-k}{2}\right)$.
 Since $k\leq n-1$, we may extend $T$ harmonically so that
 $\Delta^jT=0$ for $j\in\{0,\dotsc,k-1\}$.  In particular,
 \begin{equation}
  \label{eqn:tt1}
\sum_{j=0}^{k-1}(\Delta \delta_A{}^{C}+R_A{}^C\hash)^j R_C{}^E \hash \Delta^{k-1-j} T
=(\Delta \delta_A{}^{C}+R_A{}^C\hash)^{k-1} R_C{}^E \hash T
 \end{equation}
 Moreover, since $k\leq n-1$, the ambient tensors $\Delta^j R_{A\bar
   BC\bar E}$ are well-defined for all $j\in\{0,\dotsc,k-1\}$.  Using the
 commutator identity~\eqref{eqn:Deltanabla}, we may thus distribute
 the Laplacians in~\eqref{eqn:tt1} to obtain an equivalent expression
 which is polynomial in $\nabla^jT$ and $\nabla^r\Delta^sR_{A\bar BC\bar E}$ for
 $j,r,s\leq k-1$.  Using the identity $\Delta R_{A\bar BE\bar
   F}=-R_{C\bar B}\hash R_A{}^C{}_{E\bar F}$, we may then rewrite this as a
 polynomial in $D^jT$ and $D^rR_{A\bar BC\bar E}$, which is manifestly a tractor
 formula.  Finally, from direct inspection of~\eqref{eqn:tt1}, we see
 that in every term, after expanding and eliminating the Kronecker deltas,  the free indices
are always on  curvature terms $R_{A\bar BC\bar E}$.  Thus
 $I^{A}\Phi_{A\bar E}$ involves contracting $I$ into the
 curvature tensor.  Since $I^A$ is parallel, such contractions must
 vanish.
\end{proof}

By repeating this argument and arguing inductively from
Proposition~\ref{prop:general_formula}, we obtain the following
general tractor formula for the CR GJMS operators.

\begin{thm}
 \label{thm:general_formula}
 Let $k\leq n-1$.  There is a tractor operator
 \[\textstyle  \Psi_{C_{k+2}\dotso  C_2\bar C_1}{}^{ E_2\bar E_1} \colon
\mE_{E_2 \bar E_1}\left(\frac{k+1-n}{2},\frac{k+1-n}{2}\right) \to \mE_{ C_{k+2}\dotso  C_2\bar C_1}\left(-\frac{k+1+n}{2},-\frac{k+1+n}{2}\right) \]
 such that
 \[
\begin{split} (-1)^kZ_{ C_{k+2}}\dotso & Z_{C_4}Z_{\bar C_3}Z_{ C_2}Z_{\bar C_1}\Delta^{k+2}\\
& = D_{ C_{k+2}}\dotso D_{ C_4}D_{\bar C_3}D_{ C_2}D_{\bar C_1} + \Psi_{ C_{k+2}\dotso  C_4\bar C_3 C_2 \bar C_1}{}^{E_2\bar E_1}D_{E_2}D_{\bar E_1}
\end{split}\]
 and the full contraction $I^{ C_{k+2}}\dotsm I^{ C_2}I^{\bar C_1}\Psi_{C_{k+2}\dotso C_2\bar C_1}{}^{E_2\bar E_1}$ vanishes for any parallel tractor $I_A$.
\end{thm}

\begin{proof}
 By Corollary~\ref{cor:general_formula} and
 Proposition~\ref{prop:tangential_tractor}, it holds that
 \begin{multline*}
  Z_{C_{k+2}}\dotso Z_{C_4}Z_{\bar C_3}Z_{C_2}Z_{\bar C_1}\Delta^{k+2} = Z_{C_{k+2}}\dotso Z_{C_4}Z_{\bar C_3}\Delta^kD_{C_2}D_{\bar C_1} \\ + Z_{C_{k+2}}\dotso Z_{C_4}Z_{\bar C_3}\Phi_{C_2}{}^E\left(D_ED_{\bar C_1}\right) .
 \end{multline*}
 Moreover, the second summand on the right hand side is already a
 tractor operator satisfying the annihilation by contraction with
 $I_A$ condition, so we need only consider the first summand.  If
 $k=0$, we are done.  If $k=1$, we may immediately write
$-Z_{\bar C_3}\Delta D_{C_2}D_{\bar C_1}=D_{\bar C_3}D_{C_2}D_{\bar C_1}$, so
 we are again done.  Suppose now that $k\geq 2$.  By applying
 Proposition~\ref{prop:general_formula}, we find that
 \begin{align*}
  Z_{C_4}Z_{\bar C_3}\Delta^kD_{C_2}D_{\bar C_1} & = \Delta^{k-2}D_{C_4}D_{\bar C_3}D_{C_2}D_{\bar C_1} \\
  & \quad + \sum_{j=0}^{k-3}(\Delta \delta_{C_4}{}^{E}+R_{C_4}{}^E\hash)^j R_{E}{}^F \hash \Delta^{k-1-j} D_FD_{\bar C_3}D_{ C_2}D_{\bar C_1} \\
  & \quad - \Psi_{C_4\bar C_3}^{(1)}D_{C_2}D_{\bar C_1} ,
 \end{align*}
 where $\Psi_{C_4\bar C_3}^{(1)}$ is the action of the operators which
 factor through a single $D^{E}$ in~\eqref{eqn:general_formula}
 (i.e.\ the fourth, fifth and seventh summands).  From
 Proposition~\ref{prop:tangential} combined with Proposition \ref{prop:general_formula}
 we see that $\Psi_{C_4\bar
   C_3}^{(1)}$ is tangential. Here  we used that in Proposition \ref{prop:general_formula} the Laplacian powers appearing on the left hand side and in the first term on right hand side are each tangential.
Hence, by arguing as in
 Proposition~\ref{prop:tangential_tractor},  $\Psi_{C_4\bar
   C_3}^{(1)}$ admits a tractor
 formula which is annihilated upon complete contraction with a
 parallel tractor and its conjugate.  Arguing inductively in this
 manner yields the final conclusion.
\end{proof}

Note that the proofs of Proposition~\ref{prop:tangential_tractor} and Theorem~\ref{thm:general_formula} provide an algorithm for producting tractor formulae for the CR GJMS operators.  It is easy to carry out this algorithm at low orders to obtain tractor formulae for $P_4$ and $P_6$ as well as the $P^\prime$-operators and $Q^\prime$-curvatures of the corresponding order.  The formulae for the fourth-order invariants are self-evident; see also~\cite{CaseYang2012}.  The formulae for the sixth-order invariants are discussed in Section~\ref{sec:dim5}.

%% file: product.tex
\section{The product formulae}
\label{sec:product}

We are now prepared to prove Theorem~\ref{thm:factorisation}.  For convenience, we separate the proof into two parts.  First, we prove the factorisation of the CR GJMS operators.

\begin{prop}
 \label{prop:factorisation_p}
 Let $(M^{2n+1},H,\theta)$ be an embeddable Einstein pseudohermitian manifold.  For any integer $1\leq k\leq n+1$, the CR GJMS operator $P_{2k}$ is equal to
 \begin{equation}
 \label{eqn:factorisation_p}
  P_{2k} = \begin{cases}
            \displaystyle\prod_{\ell=1}^{\frac{k}{2}} \left(-\Delta_b+ic_\ell\nabla_0+d_\ell P\right)\left(-\Delta_b-ic_\ell\nabla_0+d_\ell P\right), & \text{if $k$ is even} \\
            Y\displaystyle\prod_{\ell=1}^{\frac{k-1}{2}} \left(-\Delta_b+ic_\ell\nabla_0+d_\ell P\right)\left(-\Delta_b-ic_\ell\nabla_0+d_\ell P\right), & \text{if $k$ is odd},
           \end{cases}
 \end{equation}
 where $c_\ell=k-2\ell+1$ and $d_\ell=\frac{n^2-(k-2\ell+1)^2}{n}$ and $Y=-\Delta_b+nP$ is the CR Yamabe operator.
\end{prop}

\begin{remark}
 To pass from~\eqref{eqn:factorisation_p} to the factorisation given in~\cite[Proposition~1.1]{BransonFontanaMorpurgo2007}, reindex the product in terms of $\ell^\prime=\frac{k}{2}-\ell$ when $k$ is even, or $\ell^\prime=\frac{k+1}{2}-\ell$ when $k$ is odd.  Note that our formula~\eqref{eqn:factorisation_p} in the case when $k$ is odd also corrects a minor typo in~\cite[Proposition~1.1]{BransonFontanaMorpurgo2007}, where the index of the product incorrectly starts at zero.
\end{remark}

\begin{proof}
 Fix a point $p\in M$.  By Proposition~\ref{prop:einstein_parallel}, there is a neighborhood $U$ of $p$ in which there exists a $\sigma\in\mE(1,0)$ such that $\theta=(\sigma\bar\sigma)^{-1}\btheta$ and $\bD_A\sigma$ is parallel.  Set $I_A=\frac{1}{n+1}\bD_A\sigma$ and observe that in the scale $\theta$ it holds that $I^{\bar A}=\left(-P\sigma/n,0,\sigma\right)$.  It follows from Theorem~\ref{thm:general_formula} that
 \[ (\sigma\bar\sigma)^{k/2}P_{2k} = 2^kI^{C_1}I^{\bar C_2}\dotsm I^{C_{k/2-1}}I^{\bar C_{k/2}}\bD_{C_1}\bD_{\bar C_2}\dotsm \bD_{C_{k/2-1}}\bD_{\bar C_{k/2}} \]
 if $k$ is even and
 \[ \sigma(\sigma\bar\sigma)^{(k-1)/2}P_{2k} = 2^kI^{\bar C_1}\dotsm I^{C_{(k-3)/2}} I^{\bar C_{(k-1)/2}} \bD_{\bar C_1}\dotsm \bD_{C_{(k-3)/2}}\bD_{\bar C_{(k-1)/2}} \]
 if $k$ is odd.  Since $I_A\in\mE_A(0,0)$ is parallel, this implies that
 \begin{equation}
  \label{eqn:prefactorisation}
  P_{2k} = \begin{cases}
            (\sigma\bar\sigma)^{-k/2}\left(4I^AI^{\bar B}\bD_A\bD_{\bar B}\right)^{k/2}, & \text{if $k$ is even}, \\
            2(\sigma\bar\sigma)^{-(k-1)/2}\sigma^{-1}I^{\bar C}\bD_{\bar C}\left(4I^AI^{\bar B}\bD_A\bD_{\bar B}\right)^{(k-1)/2}, & \text{if $k$ is odd} .
  \end{cases}
 \end{equation}
 Denote by $D_1$ and $D_2$ the operators
 \begin{align*}
  D_1 & := 2\sigma^{-1}I^{\bar C}\bD_{\bar C} \colon \mE(w,w) \to \mE(w-1,w-1), \\
  D_2 & := 4I^AI^{\bar B}\bD_A\bD_{\bar B}\colon\mE(w,w)\to\mE(w-1,w-1) .
 \end{align*}
 Using the definition of the tractor $D$-operator and Proposition~\ref{prop:ddbar-weight} and recalling that $\nabla\sigma\bar\sigma=0$ in the scale $\theta$, we readily compute that, in the scale $\theta$,
 \begin{align}
  \label{eqn:D2} D_2 & = \sigma\bar\sigma\left(-\Delta_b + ia\nabla_0 - \frac{4w(n+w)}{n}P\right)\left(-\Delta_b - ia\nabla_0 - \frac{4w(n+w)}{n}P\right), \\
  \label{eqn:D1} D_1 & = Yf + ia\nabla_0 - \frac{a^2}{n}P .
 \end{align}
 for $a=n+2w$.  Inserting~\eqref{eqn:D2} and~\eqref{eqn:D1} into~\eqref{eqn:prefactorisation} and recalling that $P_{2k}$ acts on $\mE\left(-\frac{n-k+1}{2},-\frac{n-k+1}{2}\right)$ yields the desired factorisation in $U$.  Since the factorisation is independent of the choice of $\sigma$, the factorisation holds in all of $M$.
\end{proof}

Second, we prove the factorisation of the $P^\prime$-operator and compute the $Q^\prime$-curvature.

\begin{prop}
 \label{prop:factorisation_primes}
 Let $(M^{2n+1},H,\theta)$ be an embeddable Einstein pseudohermitian manifold.  Then the $P^\prime$-operator is given by
 \begin{equation}
  \label{eqn:factorisation_pprime}
  P_{2n+2}^\prime = n!\left(\frac{2}{n}\right)^{n+1}\prod_{\ell=0}^{n} \left(-\Delta_b + 2\ell P\right) ,
 \end{equation}
 and the $Q^\prime$-curvature is given by
 \begin{equation}
  \label{eqn:factorisation_qprime}
  Q_{2n+2}^\prime = (n!)^2\left(\frac{4P}{n}\right)^{n+1} .
 \end{equation}
\end{prop}

\begin{proof}
 Since $\theta$ is Einstein, $[\Delta_b,T]=0$ and the operator $C$ defined in Subsection~\ref{subsec:pluri/pluri} satisfies $C=\Delta_b^2+n^2T^2$.  A straightforward computation reveals that, for $c_\ell$ and $d_\ell$ as in Proposition~\ref{prop:factorisation_p},
 \begin{multline*}
  \left(-\Delta_b + c_\ell i\nabla_0 + d_\ell P\right)\left(-\Delta_b - c_\ell i\nabla_0 + d_\ell P\right) \\= \frac{c_\ell^2}{n^2}C + \frac{d_\ell}{n}\left(-\Delta_b + (n-c_\ell)P\right)\left(-\Delta_b + (n+c_\ell)P\right) .
 \end{multline*}
 Inserting this into~\eqref{eqn:factorisation_p} and recalling that $\mP\subset\ker C$ yields
 \[ P_{2k}\big\rvert_{\mP} = \begin{cases}
         \displaystyle\prod_{\ell=1}^{\frac{k}{2}} \frac{d_\ell}{n}\left(-\Delta_b + (n-c_\ell)P\right)\left(-\Delta_b + (n+c_\ell)P\right), & \text{if $k$ is even} \\
         Y\displaystyle\prod_{\ell=1}^{\frac{k-1}{2}} \frac{d_\ell}{n}\left(-\Delta_b + (n-c_\ell)P\right)\left(-\Delta_b + (n+c_\ell)P\right), & \text{if $k$ is odd} .
                            \end{cases} \]
 Equivalently, we have that
 \begin{equation}
  \label{eqn:GJMS_pluriharmonic}
  P_{2k}\big\rvert_{\mP} = \prod_{\ell=1}^k\frac{n-k-1+2\ell}{n}\bigl(-\Delta_b + (n-k-1+2\ell)P\bigr) .
 \end{equation}
 Formally we have that
 \begin{align}
  \label{eqn:formal_Pprime} P_{2n+2}^\prime & = \left( \frac{2}{n-k+1}P_{2k}\big\rvert_{\mP}\right)_{k=n+1} \\
  \label{eqn:formal_Qprime} Q_{2n+2}^\prime & = \left( \frac{4}{(n-k+1)^2}P_{2k}(1)\right)_{k=n+1} .
 \end{align}
 Inserting~\eqref{eqn:GJMS_pluriharmonic} into~\eqref{eqn:formal_Pprime} yields~\eqref{eqn:factorisation_pprime}.  Inserting~\eqref{eqn:GJMS_pluriharmonic} into~\eqref{eqn:formal_Qprime} yields~\eqref{eqn:factorisation_qprime}.  As discussed in Subsection~\ref{subsec:pluri/pluri}, this argument is made rigorous via log densities.
\end{proof}

%% file: dim5.tex
\section{$Q^\prime$-curvature in dimension five}
\label{sec:dim5}

For five-dimensional CR manifolds, Corollary~\ref{cor:general_formula} implies the ambient formula
\begin{equation}
 \label{eqn:dim5_p_formula}
 -Z_{\bar C}Z_AZ_{\bar B}\Delta^3 = D_{\bar C}D_AD_{\bar B} - Z_{\bar C}R_{A\bar BE}{}^FD_FD^E
\end{equation}
acting on $\mE(0,0)$.  Obtaining from this a tractor formula for the sixth-order CR GJMS operator and both the $P^\prime$-operator and the $Q^\prime$-curvature in dimension five requires establishing a tractor formula for the restriction of the ambient curvature $R_{A\bar BC\bar E}$ to $\mQ$.  This tractor is a multiple of the CR Weyl tractor $S_{A\bar BC\bar E}\in\mE_{A\bar BC\bar E}(-1,-1)$, a tractor defined in all dimensions which has Weyl-type symmetries and whose projecting part is the CR Weyl curvature $S_{\alpha\bar\beta\gamma\bar\sigma}$ when $n\geq2$.  The following result both constructs the CR Weyl tractor and computes it in terms of the splitting~\eqref{tsplit} and pseudohermitian invariants whenever a contact form $\theta$ is chosen.


\begin{prop}
 \label{prop:cr_weyl}
 Let $(M^{2n+1},H,\theta)$ be an embeddable pseudohermitian manifold.  With respect to $\theta$, the CR Weyl tractor is
 \begin{multline*}
  S_{A\bar BC\bar E} = (n-1)W_A{}^\alpha W_{\bar B}{}^{\bar\beta}\Omega_{\alpha\bar\beta C\bar E} + (n-1)W_A{}^\alpha Z_{\bar B}\Phi_{\alpha C\bar E} \\ + (n-1)Z_A W_{\bar B}{}^{\bar\beta}\Phi_{\bar\beta C\bar E} + Z_AZ_{\bar B}\Psi_{C\bar E},
 \end{multline*}
 where $\Omega_{\alpha\bar\beta C\bar E}$ is the tractor curvature
 \begin{equation}
  \label{eqn:tractor_curv}
  \Omega_{\alpha\bar\beta C\bar E} = W_C{}^\gamma W_{\bar E}{}^{\bar\sigma}S_{\alpha\bar\beta\gamma\bar\sigma} + iW_C{}^\gamma Z_{\bar E}V_{\alpha\bar\beta\gamma} - iZ_CW_{\bar E}{}^{\bar\sigma}V_{\bar\sigma\alpha\bar\beta} + Z_CZ_{\bar E}U_{\alpha\bar\beta} ,
 \end{equation}
 $\Phi_{\alpha C\bar E}$ and $\Psi_{C\bar E}$ are given by
 \begin{align*}
  \Phi_{\alpha C\bar E} & = iW_C{}^\gamma W_{\bar E}{}^{\bar\sigma}V_{\gamma\bar\sigma\alpha} + iW_C{}^\gamma Z_{\bar E}Q_{\alpha\gamma} + Z_CW_{\bar E}{}^{\bar\sigma}U_{\alpha\bar\sigma} + iZ_CZ_{\bar E}Y_\alpha, \\
  \Psi_{C\bar E} & = (n-1)W_C{}^\gamma W_{\bar E}{}^{\bar\sigma}U_{\gamma\bar\sigma} + (n-1)iW_C{}^\gamma Z_{\bar E}Y_\gamma \\
   & \quad - (n-1)iZ_CW_{\bar E}{}^{\bar\sigma}Y_{\bar\sigma} + Z_CZ_{\bar E}\mO ,
 \end{align*}
 and $\mO$ is given by
 \[ \mO = -i\nabla^{\gamma}Y_{\gamma} + 2P^{\alpha\bar\beta}U_{\alpha\bar\beta} + A^{\alpha\gamma}Q_{\alpha\gamma} . \]
\end{prop}

\begin{remark}
 By the symmetries of the CR Weyl tractor, $\mO$ is real-valued.  This also follows immediately from~\eqref{eqn:RedivY}; in dimension three, this observation and~\eqref{eqn:divQ} together recover the Bianchi-type identity for the Cartan tensor discovered by Cheng and Lee~\cite[Proposition~3.1]{ChengLee1990}.  Although $\mO$ is not in general a CR invariant, Proposition~\ref{prop:cr_weyl} implies that it is a CR invariant for three-dimensional CR manifolds.  Indeed, by invariant theory, it must be a nonzero constant multiple of the obstruction function $\eta$ in~\eqref{eqn:ambient_metric_ricci} (cf.\ \cite{Graham1987}).  In these ways, we can regard the pseudohermitian invariant $\mO$ as an extension of the obstruction function for three-dimensional CR manifolds to higher dimensions.
\end{remark}

\begin{proof}
 Recall that the tractor curvature $\Omega_{\alpha\bar\beta C\bar E}\in\mE_{\alpha\bar\beta C\bar E}(0,0)$ is CR invariant and $\Omega_\gamma{}^\gamma{}_{C\bar E}=0$.  The CR Weyl tractor is obtained from the embedding $\mE_{(\alpha\bar\beta)_0}(0,0)\hookrightarrow\mE_{A\bar B}(-1,-1)$, where $\mE_{(\alpha\bar\beta)_0}(0,0)$ is the space of CR invariant trace-free Hermitian $(1,1)$-forms of weight $(0,0)$.  One can directly check that
 \begin{align*}
  M^{\alpha\bar\beta}_{A\bar B}S_{\alpha\bar\beta} & := (n+w-1)(n+w^\prime-1)W_A{}^\alpha W_{\bar B}{}^{\bar\beta}S_{\alpha\bar\beta} - (n+w^\prime-1)W_A{}^\alpha Z_{\bar B}\nabla^{\bar\sigma} S_{\alpha\bar\sigma} \\
   & \quad - (n+w-1)Z_AW_{\bar B}{}^{\bar\beta}\nabla^\gamma S_{\gamma\bar\beta} + Z_AZ_{\bar B}\left(\nabla^\gamma\nabla^{\bar\sigma} S_{\gamma\bar\sigma} + (n+w-1)P^{\gamma\bar\sigma}S_{\gamma\bar\sigma}\right)
 \end{align*}
 is a linear map $M^{\alpha\bar\beta}_{A\bar B}\colon\mE_{(\alpha\bar\beta)_0}(w,w^\prime)\to\mE_{A\bar B}(w-1,w^\prime-1)$.  To do so, make the ansatz that
 \[ M_{A\bar B} = W_A{}^\alpha W_{\bar B}{}^{\bar\beta} S_{\alpha\bar\beta} + 2\Real Z_AW_{\bar B}{}^{\bar\beta}\omega_{\bar\beta} + Z_AZ_{\bar\beta}\rho \]
 is a tractor.  Thus $\bD^{\bar B}M_{A\bar B}$ is a tractor, and hence zero for generic values of $w$ and $w^\prime$.  Computing $\bD^{\bar B}M_{A\bar B}$ in components then yields the components for $M_{A\bar B}$ in terms of $S_{\alpha\bar\beta}$; multiplying by $(n+w-1)(n+w^\prime-1)$ to cancel the poles of these components yields our expression for the operator $M^{\alpha\bar\beta}_{A\bar B}$.

 Next, applying Lemma~\ref{lem:div_identities} to~\eqref{eqn:tractor_curv} yields
 \[ \nabla^{\bar\beta}\Omega_{\alpha\bar\beta C\bar E} = -(n-1)\left[ iW_C{}^\gamma W_{\bar E}{}^{\bar\sigma}V_{\gamma\bar\sigma\alpha} + iW_C{}^\gamma Z_{\bar E}Q_{\alpha\gamma} + Z_CW_{\bar E}{}^{\bar\sigma}U_{\alpha\bar\sigma} + iZ_CZ_{\bar E}Y_\alpha \right] \]
 and
 \begin{multline*}
  \left(\nabla^\alpha\nabla^{\bar\beta} + (n-1)P^{\alpha\bar\beta}\right)\Omega_{\alpha\bar\beta C\bar E} = (n-1)^2W_C{}^\gamma W_{\bar E}{}^{\bar\sigma}U_{\gamma\bar\sigma} \\ + (n-1)^2iW_C{}^\gamma Z_{\bar E}Y_\gamma - (n-1)^2iZ_CW_{\bar E}{}^{\bar\sigma}Y_{\bar\sigma} + (n-1)Z_C Z_{\bar E}\mO .
 \end{multline*}
 In particular, we see that $S_{A\bar BC\bar E}=\frac{1}{n-1}M^{\alpha\bar\beta}_{A\bar B}\Omega_{\alpha\bar\beta C\bar E}$ makes sense in all dimensions.  Finally, direct inspection shows that the CR Weyl tractor has Weyl-type symmetries, as desired.
\end{proof}

The relationship between the restriction of the ambient curvature and the CR Weyl tractor can be derived by considering the commutator $[D_A,D_{\bar B}]$ and its tractor analogue acting on homogeneous vectors of degree $(0,0)$ and sections of $\mE_A(0,0)$, respectively.  The following argument is an adaptation to the CR setting of the argument given from~\cite[p.\ 369]{GoverPeterson2003} for the analogous result in conformal geometry.

\begin{lem}
 \label{lem:curv_amb_to_tractor}
 Let $(M^{2n+1},H)$ be an embeddable CR manifold with $n>1$.  Then
 \[ S_{A\bar BC\bar E} = (n-1)R_{A\bar BC\bar E} . \]
\end{lem}

\begin{proof}
 Let $V^C\in\mE^C(0,0)$ and let $\cV^C\in\tilde\mE^C$ be an extension of $V^C$ which is homogeneous of degree zero.  It follows from Proposition~\ref{prop:general_formula} and properties of the ambient connection that
 \begin{multline}
  \label{eqn:weyl_ambient_commutator}
  [D_A,D_{\bar B}]\cV^C = n(n-1)R_{A\bar BE}{}^C\cV^E \\ - nZ_AR_{E\bar BF}{}^CW^{E\bar\beta}\nabla_{\bar\beta}\cV^F - nZ_{\bar B}R_{A\bar EF}{}^CW^{\bar E\alpha}\nabla_\alpha\cV^F .
 \end{multline}
 Since all the operators in~\eqref{eqn:weyl_ambient_commutator} are tangential, we can restrict to $M$ and regard this as a tractor formula for $V^C$.  On the other hand, a straightforward computation using the definition of the tractor $D$-operator, the tractor curvature, and Proposition~\ref{prop:cr_weyl} yields
 \begin{multline}
  \label{eqn:weyl_tractor_commutator}
  [\bD_A,\bD_{\bar B}]V^C = nS_{A\bar BE}{}^CV^E \\ - \frac{n}{n-1}Z_AS_{E\bar BF}{}^CW^{E\bar\beta}\nabla_{\bar\beta} V^F - \frac{n}{n-1}Z_{\bar B}S_{A\bar EF}{}^CW^{\bar E\alpha}\nabla_\alpha V^F .
 \end{multline}
 Multiplying both~\eqref{eqn:weyl_ambient_commutator} and~\eqref{eqn:weyl_tractor_commutator} by $Z_GZ_{\bar H}$ and skewing over the pairs $(A,G)$ and $(\bar B,\bar H)$ yields $(n-1)Z_{[G}Z_{\bar H}R_{A\bar B]E\bar C}=Z_{[G}Z_{\bar H}S_{A\bar B]E\bar C}$.  Contracting $W^{G\bar\beta}$ into this yields $(n-1)Z_AZ_{[\bar H}R_{G\bar B]E\bar C}W^{G\bar\beta}=Z_AZ_{[\bar H}S_{G\bar B]E\bar C}W^{G\bar\beta}$, where our notation means skew over the pair $(\bar B,\bar H)$.  Now multiplying both~\eqref{eqn:weyl_ambient_commutator} and~\eqref{eqn:weyl_tractor_commutator} by $Z_{\bar H}$ and skewing over the pair $(\bar B,\bar H)$ yields $(n-1)Z_{[\bar H}R_{A\bar B]E\bar C}=Z_{[\bar H}S_{A\bar B]E\bar C}$.  Contracting with $W^{\bar H\gamma}$ yields $(n-1)Z_{\bar B}R_{A\bar HE\bar C}W^{\bar H\gamma}=Z_{\bar B}S_{A\bar HE\bar C}W^{\bar H\gamma}$.  Using this and its conjugate to compare~\eqref{eqn:weyl_ambient_commutator} and~\eqref{eqn:weyl_tractor_commutator} yields the desired result.
\end{proof}

Considering Lemma~\ref{lem:curv_amb_to_tractor} in the case of dimension five and using~\eqref{eqn:dim5_p_formula} yields the following tractor formulae for the sixth-order CR GJMS operator, the $P^\prime$-operator, and the $Q^\prime$-curvature.

\begin{prop}
 \label{prop:dim5_tractor}
 Let $(M^5,H)$ be an embeddable five-dimensional CR manifold.  Then
 \begin{equation}
  \label{eqn:dim5_tractor_p}
  \frac{1}{8}Z_{\bar C}Z_AZ_{\bar B}P_6f = \bD_{\bar C}\bD_A\bD_{\bar B}f - Z_{\bar C}S_{A\bar BE}{}^F\bD_F\bD^E f
 \end{equation}
 for all $f\in\mE(0,0)$.  Moreover, given a choice of contact form $\theta=(\sigma\bar\sigma)^{-1}\btheta$, it holds that
 \begin{align}
  \label{eqn:dim5_tractor_pprime} \frac{1}{8}Z_{\bar C}Z_AZ_{\bar B}P^\prime u & = \bD_{\bar C}K_{A\bar B}(u) - Z_{\bar C}S_{A\bar BE}{}^FK_F{}^E(u) , \\
  \label{eqn:dim5_tractor_qprime} \frac{1}{8}Z_{\bar C}Z_AZ_{\bar B}Q^\prime & = \bD_{\bar C}I_{A\bar B} - Z_{\bar C}S_{A\bar BE}{}^FI_F{}^E
 \end{align}
 for all $u\in\mP$, where $K_{A\bar B}(u)$ and $I_{A\bar B}$ are as in Lemma~\ref{lem:dbard_logpluri} and Lemma~\ref{lem:dbard_log2}, respectively, and we require that $\theta$ is pseudo-Einstein in~\eqref{eqn:dim5_tractor_qprime}.
\end{prop}

Using Proposition~\ref{prop:cr_weyl} and the formula for the tractor $D$-operator, one could derive local formulae for the sixth-order CR GJMS operator $P_6$ in general dimensions as well as the $P^\prime$-operator and the $Q^\prime$-curvature for five-dimensional CR manifolds.  Here we derive local formulae for the $P^\prime$-operator and $Q^\prime$-curvature of a pseudo-Einstein five-manifold.  First, we consider the $Q^\prime$-curvature.

\begin{cor}
 \label{cor:dim5_qprime}
 Let $(M^5,H,\theta)$ be a pseudo-Einstein manifold.  Then
 \begin{multline}
  \label{eqn:dim5/dim5_qprime}
  \frac{1}{8}Q^\prime = \frac{1}{2}\Delta_b^2 P + \frac{1}{2}\Delta_b\lv A_{\alpha\beta}\rv^2 - 2\Imaginary\nabla^\gamma\left(A_{\beta\gamma}\nabla^\beta P\right) \\ - 2\Delta_b P^2 - 4P\lv A_{\alpha\beta}\rv^2 + 4P^3 - 2\mO .
 \end{multline}
 In particular, the total $Q^\prime$-curvature of a compact pseudo-Einstein five-manifold is
 \begin{equation}
  \label{eqn:dim5_total_qprime}
  \int_{M^5} Q^\prime = 16\int_{M^5} \left( 2P^3 - 2P\lv A_{\alpha\beta}\rv^2 - S_{\alpha\bar\beta\gamma\bar\delta}A^{\alpha\gamma}A^{\bar\beta\bar\delta} - \lv V_{\alpha\bar\beta\gamma}\rv^2 \right) .
 \end{equation}
\end{cor}

\begin{proof}
 The local formula~\eqref{eqn:dim5/dim5_qprime} follows from a straightforward computation using Proposition~\ref{prop:cr_weyl} and Proposition~\ref{prop:dim5_tractor}.  Lemma~\ref{lem:div_identities} and the definition of $\mO$ imply that
 \[ \int_{M^5} \mO = \int_{M^5} \left( S_{\alpha\bar\beta\gamma\bar\delta}A^{\alpha\gamma}A^{\bar\beta\bar\delta} + \lv V_{\alpha\bar\beta\gamma}\rv^2\right) , \]
 from which~\eqref{eqn:dim5_total_qprime} readily follows.
\end{proof}

Second, we consider the $P^\prime$-operator on pseudo-Einstein manifolds.  Note that while the formula below can be derived from Proposition~\ref{prop:cr_weyl} and Proposition~\ref{prop:dim5_tractor}, the derivation is simplified using~\eqref{eqn:dim5/dim5_qprime} and the transformation formula~\eqref{eqn:qprime_trans_intro} for the $Q^\prime$-curvature.

\begin{cor}
 \label{cor:dim5_pprime}
 Let $(M^5,H,\theta)$ be a pseudo-Einstein manifold.  Then
 \begin{align*}
  P^\prime\Upsilon & = -2\Delta_b^3\Upsilon + 24\Real\Delta_b\nabla^\gamma\left(P\nabla_\gamma\Upsilon\right) + 24\Real\nabla^\gamma\left(P\nabla_\gamma\Delta_b\Upsilon\right) \\
   & \quad + 8\Imaginary\Delta_b\nabla^\gamma\left(A_{\beta\gamma}\nabla^\beta\Upsilon\right) + 8\Imaginary\nabla^\gamma\left(A_{\beta\gamma}\nabla^\beta\Delta_b\Upsilon\right) - 16\Real\nabla^\beta\nabla^\gamma\left(P\nabla_\gamma\nabla_\beta\Upsilon\right) \\
   & \quad - 16\Real\nabla^\gamma\left[\left( 2U_\gamma{}^\beta + 2A_{\gamma\mu}A^{\mu\beta} + \left(\Delta_bP - 4P^2 - \lv A_{\alpha\delta}\rv^2\right)h_\gamma{}^\beta\right)\nabla_\beta\Upsilon\right] \\
   & \quad - 64\Imaginary\nabla^\gamma\left(PA_{\beta\gamma}\nabla^\beta\Upsilon\right) .
 \end{align*}
 for all $\Upsilon\in\mP$.
\end{cor}

\begin{remark}
 Note that this formula for the $P^\prime$-operator is manifestly formally self-adjoint.  In particular, Corollary~\ref{cor:dim5_pprime} and the transformation formula~\eqref{eqn:qprime_trans_intro} give an intrinsic proof of the fact that the total $Q^\prime$-curvature is a global \invariant/ in dimension five.
\end{remark}

\begin{proof}
 Since $\Upsilon\in\mP$ and $U_{\gamma\bar\beta}+A_{\gamma\mu}A^\mu{}_{\bar\beta}-\frac{1}{2}\lv A_{\alpha\bar\sigma}\rv^2h_{\gamma\bar\beta}$ is trace-free, Lemma~\ref{lem:div_identities} implies that
 \begin{multline}
  \label{eqn:mP_id1}
  \nabla^\gamma\left(\left(U_\gamma{}^\beta + A_{\gamma\mu}A^{\mu\beta} - \frac{1}{2}\lv A_{\alpha\delta}\rv^2h_\gamma{}^\beta\right)\nabla_\beta\Upsilon\right) \\ = \left(iY^\gamma - iA^{\beta\gamma}\nabla_\beta P + \frac{1}{2}\nabla^\gamma\lv A_{\alpha\delta}\rv^2\right)\nabla_\gamma\Upsilon .
 \end{multline}
 Using the commutator identities~\cite[Lemma~2.2]{Lee1988} and the assumption $\Upsilon\in\mP$, we compute that
 \begin{equation}
  \label{eqn:mP_id2}
  \nabla^\gamma\nabla_\gamma\nabla_\beta\Upsilon = \frac{3}{2}\nabla_\beta\Delta_b\Upsilon + 3iA_{\beta\gamma}\nabla^\gamma\Upsilon + 3P\nabla_\beta\Upsilon .
 \end{equation}

 Consider now the family $\htheta_t=e^{t\Upsilon}\theta$ of pseudo-Einstein contact forms.  In the following, we shall use hats to denote pseudohermitian invariants defined in terms of $\htheta_t$ and suppress the dependence on $t$ in our notation.  It follows from~\eqref{eqn:qprime_trans_intro} that
 \begin{equation}
  \label{eqn:pprime_formula}
  P^\prime(\Upsilon) = \left.\frac{\partial}{\partial t}\right|_{t=0} e^{3t\Upsilon}\hQ^\prime .
 \end{equation}
 The right-hand side of~\eqref{eqn:pprime_formula} is readily expanded using the identities
 \begin{align*}
  \left.\frac{\partial}{\partial t}\right|_{t=0}e^{t\Upsilon}\hP & = -\frac{1}{2}\Delta_b\Upsilon, \\
  \left.\frac{\partial}{\partial t}\right|_{t=0}\hA_{\alpha\beta} & = i\nabla_\alpha\nabla_\beta\Upsilon , \\
  \left.\frac{\partial}{\partial t}\right|_{t=0}e^{2t\Upsilon}\hmO & = -4\Imaginary Y^\gamma\nabla_\gamma\Upsilon , \\
  \left.\frac{\partial}{\partial t}\right|_{t=0}e^{(1-w)t\Upsilon}\nabla^\gamma\left(e^{wt\Upsilon}\omega_\gamma\right) & = (w+2)\omega_\gamma\nabla^\gamma\Upsilon
 \end{align*}
 for all $\omega_\gamma\in\mE_\gamma$ and all $w\in\bR$ (cf.\ \cite{GoverGraham2005,Lee1988}).  Using~\eqref{eqn:mP_id1} and~\eqref{eqn:mP_id2} to simplify the resulting expansion yields the desired formula.
\end{proof}

It is interesting to compare the total $Q^\prime$-curvature~\eqref{eqn:dim5_total_qprime} to the other known and interesting global \invariant/ in dimension five, namely the Burns--Epstein invariant~\cite{BurnsEpstein1990c}.  Marugame~\cite{Marugame2013} computed the Burns--Epstein invariant $\mu(M^5)$ of the boundary $M^5$ of a strictly pseudoconvex bounded domain $X\subset\bC^3$, showing that
\begin{equation}
 \label{eqn:marugame_burns_epstein}
 \mu(M^5) = -\frac{1}{16\pi^3}\int_{M^5} \left( 2P^3 - 2P\lv A_{\alpha\beta}\rv^2 - S_{\alpha\bar\beta\gamma\bar\delta}A^{\alpha\gamma}A^{\bar\beta\bar\delta} + \frac{1}{2}P\lv S_{\alpha\bar\beta\gamma\bar\sigma}\rv^2 \right) ,
\end{equation}
while the Burns--Epstein invariant is related to the Euler characteristic of $X$ via the formula
\begin{equation}
 \label{eqn:marugame_gauss_bonnet}
 \chi(X) = \int_X \left(c_3 - \frac{1}{2}c_1c_2 + \frac{1}{8}c_1^3\right) + \mu(M^5) .
\end{equation}
Indeed, one can regard the formula~\eqref{eqn:marugame_burns_epstein} as defining a global pseudohermitian invariant $\mu(M^5)$.  By realizing $M$ as the boundary of a complex manifold, Marugame gave an extrinsic proof that $\mu(M^5)$ is a global \invariant/~\cite{Marugame2013}.

Direct comparison of~\eqref{eqn:dim5_total_qprime} and~\eqref{eqn:marugame_burns_epstein} implies both Theorem~\ref{thm:marugame} and the fact that $\int\lv V_{\alpha\bar\beta\gamma}\rv^2+\frac{1}{2}P\lv S_{\alpha\bar\beta\gamma\bar\sigma}\rv^2$ is a \invariant/.  We here give an intrinsic proof of the latter fact under the additional assumption that $c_2(H^{1,0})$ vanishes in $H^4(M;\bR)$ by studying properties of the pseudohermitian invariant $\mI^\prime$.

\begin{prop}
 \label{prop:iprime}
 Let $(M^5,H,\theta)$ be a pseudohermitian manifold and define
 \begin{align}
  \label{eqn:iprime} \mI^\prime & = -\frac{1}{8}\Delta_b\left| S_{\alpha\bar\beta\gamma\bar\sigma}\right|^2 + \left| V_{\alpha\bar\beta\gamma}\right|^2 + \frac{1}{2}P\left| S_{\alpha\bar\beta\gamma\bar\sigma}\right|^2 , \\
  \label{eqn:Xalpha} X_\alpha & = -iS_{\alpha\bar\rho\gamma\bar\sigma}V^{\bar\rho\gamma\bar\sigma} + \frac{1}{4}\nabla_\alpha\left| S_{\gamma\bar\sigma\delta\bar\rho}\right|^2 .
 \end{align}
 Suppose $\htheta=e^\Upsilon\theta$.  Then
 \begin{equation}
  \label{eqn:Iprime_transformation}
  e^{3\Upsilon}\hmI^\prime = \mI^\prime + 2\Real X^\gamma\nabla_\gamma\Upsilon .
 \end{equation}
 Moreover, if $M$ is compact, $c_2(H^{1,0})$ vanishes in $H^4(M;\bR)$, and both $\theta$ and $\htheta$ are pseudo-Einstein, then
 \[ \int_M \hmI^\prime\,\htheta\wedge d\htheta \wedge d\htheta = \int_M \mI^\prime\,\theta\wedge d\theta \wedge d\theta . \]
\end{prop}

The proof of Proposition~\ref{prop:iprime} depends on an explicit realisation of the real Chern class $c_2(H^{1,0})\in H^4(M;\bR)$ and the observation that, in dimension five,
\begin{equation}
 \label{eqn:mS}
 \mS_{\alpha\bar\beta} = S_{\alpha\bar\rho\gamma\bar\sigma}S^{\bar\rho}{}_{\bar\beta}{}^{\bar\sigma\gamma} - \frac{1}{2}\left| S_{\gamma\bar\sigma\delta\bar\rho}\right|^2\,h_{\alpha\bar\beta} = 0 .
\end{equation}
Indeed, since $\dim_{\bC} H^{1,0}=2$, we have that
\[ 0 = h_{[\alpha\bar\beta}S_\gamma{}^\gamma{}_{|\mu|}{}^\nu S_{\rho]}{}^\rho{}_\nu{}^\mu = 2\mS_{\alpha\bar\beta} , \]
where our notation means that we skew over the lower indices $\alpha,\gamma,\rho$.  Define
\begin{equation}
 \label{eqn:Xalpha_dual}
 \xi = X_\alpha\,\theta\wedge d\theta\wedge\theta^\alpha + X_{\bar\beta}\,\theta\wedge d\theta\wedge\theta^{\bar\beta} .
\end{equation}
The above observations enable us to identify $\xi$ as an element of $4\pi^2c_2(H^{1,0})$ on any five-dimensional pseudo-Einstein manifold.


\begin{lem}
 \label{lem:c2_rep}
 Let $(M^5,H,\theta)$ be a pseudo-Einstein manifold and let $\xi$ be as in~\eqref{eqn:Xalpha_dual}.  Then $\xi$ is a representative of $c_2(H^{1,0})\in H^4(M;\bR)$.  In particular, if the real Chern class $c_2(H^{1,0})$ vanishes, then
 \[ \Real \int_M X^\gamma\nabla_\gamma v = 0 \]
 for all $v\in\mP$.
\end{lem}

\begin{proof}
 Observe that $\xi\wedge\theta=0=\xi\wedge d\theta$.  Suppose that $\xi$ is exact.  As observed by Rumin~\cite{Rumin1994}, we obtain a three-form $\alpha$ such that $d\alpha=\xi$ and $\alpha\wedge\theta=0=\alpha\wedge d\theta$.  Denote $d_b^cv=-i\nabla_\alpha v\,\theta^\alpha + i\nabla_{\bar\beta}v\,\theta^{\bar\beta}$ and observe that $d_b^cv\wedge\xi=2\Real X^\gamma\nabla_\gamma v$.  Since $v\in\mP$ if and only if $dd_b^cv=0\mod d\theta$ (cf.\ \cite{Lee1988}), we conclude that if $v\in\mP$, then
 \[ 2\Real X^\gamma\nabla_\gamma v = d_b^cv\wedge\xi = -d\left(d_b^cv\wedge\alpha\right) . \]
 In particular, $\Real\int X^\gamma\nabla_\gamma v=0$.

 We now show that $\xi$ is exact.  It suffices to show that $\xi\in 4\pi^2c_2\left(H^{1,0}\right)$.

 Since $\theta$ is pseudo-Einstein, $c_1\left(H^{1,0}\right)$ vanishes in $H^1(M;\bR)$ (cf.\ \cite{Lee1988}).  It follows that $8\pi^2c_2\left(H^{1,0}\right)=\left[ \Pi_\mu{}^\nu\wedge\Pi_\nu{}^\mu\right]$ for $\Pi_\alpha{}^\beta$ the curvature forms~\eqref{eqn:structure_equation}.  Since $\dim_{\bC} H^{1,0}=2$, we compute that
 \begin{align*}
  \Pi_\mu{}^\nu\wedge\Pi_\nu{}^\mu & = R_{\alpha\bar\beta\mu}{}^\nu R_{\gamma\bar\sigma\nu}{}^\mu \theta^\alpha\wedge\theta^{\bar\beta}\wedge\theta^{\gamma}\wedge\theta^{\bar\sigma} - 2R_{\alpha\bar\beta\mu}{}^\nu\nabla^\mu A_{\gamma\nu}\,\theta\wedge\theta^\alpha\wedge\theta^{\bar\beta}\wedge\theta^\gamma \\
   & \quad + 2R_{\alpha\bar\beta\mu}{}^\nu\nabla_\nu A_{\bar\sigma}{}^\mu\,\theta\wedge\theta^\alpha\wedge\theta^{\bar\beta}\wedge\theta^{\bar\sigma} - 2iA_{\bar\sigma}{}^\nu\nabla_{\bar\beta}A_{\gamma\nu}\theta\wedge\theta^\gamma\wedge\theta^{\bar\beta}\wedge\theta^{\bar\sigma} \\
   & \quad  - 2iA_{\alpha\mu}\nabla_\gamma A_{\bar\sigma}{}^\mu\,\theta\wedge\theta^\alpha\wedge\theta^\gamma\wedge\theta^{\bar\sigma} - 2A_{\alpha\mu}A_{\bar\rho}{}^\mu h_{\nu\bar\sigma}\,\theta^\alpha\wedge\theta^{\bar\sigma}\wedge\theta^\nu\wedge\theta^{\bar\rho} .
 \end{align*}
 To simplify this, observe that, since $\dim_{\bC} H^{1,0}=2$ and $\nabla_\rho A_{\alpha\gamma}=\nabla_\gamma A_{\alpha\rho}$ (cf.\ \cite{Lee1988}),
 \begin{multline}
  \label{eqn:dA} d\left(A_{\alpha\mu}A^\mu{}_{\bar\beta}\,\theta\wedge\theta^\alpha\wedge\theta^{\bar\beta}\right) = A_{\alpha\mu}\nabla_\gamma A^\mu{}_{\bar\beta}\,\theta\wedge\theta^\alpha\wedge\theta^\gamma\wedge\theta^{\bar\beta} \\ + A_{\bar\sigma}{}^\mu\nabla_{\bar\beta}A_{\alpha\mu}\theta\wedge\theta^\alpha\wedge\theta^{\bar\beta}\wedge\theta^{\bar\sigma} + A_{\alpha\mu}A^\mu{}_{\bar\beta}\,d\theta\wedge\theta^\alpha\wedge\theta^{\bar\beta} .
 \end{multline}
 Furthermore, since $\dim_{\bC} H^{1,0}=2$ and $\theta$ is pseudo-Einstein,
 \begin{align}
  \label{eqn:R2} R_{\alpha\bar\beta\mu}{}^\nu R_{\gamma\bar\sigma\nu}{}^\mu \theta^\alpha\wedge\theta^{\bar\beta}\wedge\theta^{\gamma}\wedge\theta^{\bar\sigma} & = \left(\frac{1}{2}\lv S_{\alpha\bar\beta\gamma\bar\sigma}\rv^2 - 3P^2\right) d\theta\wedge d\theta , \\
  \label{eqn:RV} R_{\alpha\bar\beta\mu}{}^\nu\nabla^\mu A_{\gamma\nu}\,\theta\wedge\theta^\alpha\wedge\theta^{\bar\beta}\wedge\theta^\gamma & = \left(iS_{\gamma\bar\sigma\mu\bar\nu}V^{\bar\sigma\mu\bar\nu} - \frac{3}{2}\nabla_\gamma(P^2)\right)\theta\wedge d\theta\wedge\theta^\gamma .
 \end{align}
 Using~\eqref{eqn:dA}, \eqref{eqn:R2} and~\eqref{eqn:RV} to simplify the expression for $\Pi_\mu{}^\nu\wedge\Pi_\nu{}^\mu$ yields
 \[ \Pi_\mu{}^\nu\wedge\Pi_\nu{}^\mu = 2\xi + d\left[ \left(\frac{1}{2}\lv S_{\alpha\bar\beta\gamma\bar\sigma}\rv^2 - 3P^2\right)\theta\wedge d\theta - 2iA_{\alpha\mu}A^\mu{}_{\bar\beta}\,\theta\wedge\theta^\alpha\wedge\theta^{\bar\beta} \right] . \]
 In particular, we see that $4\pi^2c_2(H^{1,0})=[\xi]$, as desired.
\end{proof}

\begin{proof}[Proof of Proposition~\ref{prop:iprime}]
 Let $\htheta=e^\Upsilon\theta$.  A straightforward computation using the transformation formulae in~\cite{GoverGraham2005,Lee1988} yields
 \[ e^{2\Upsilon}\hX_\alpha = X_\alpha + \mS_{\alpha\bar\beta}\nabla^\beta\Upsilon . \]
 As $\mS_{\alpha\bar\beta}=0$, we see that $X_\alpha$ is a CR invariant.  From this observation and the transformation formulae in~\cite{GoverGraham2005,Lee1988}, it is straightforward to verify that~\eqref{eqn:Iprime_transformation} holds.  The final conclusion now follows from Lemma~\ref{lem:c2_rep}.
\end{proof}

One important class of CR manifolds are those which embed as boundaries of Stein domains.  As pointed out to us by Taiji Marugame, five-dimensional CR manifolds in this class always have vanishing Chern classes.

\begin{prop}
 \label{prop:stein_vanish}
 Let $(M^5,H,\theta)$ be a pseudo-Einstein manifold which is the boundary of a Stein manifold $V^3$.  Then $c_2(H^{1,0})=0$.
\end{prop}

\begin{proof}
 Note that the holomorphic normal bundle $N^{1,0}$ of $M\subset V$ is trivial and $H^{1,0}\oplus N^{1,0}=i^\ast T^{1,0}V$, the pullback of $T^{1,0}V$ with respect to the inclusion $i\colon M\to V$.  Thus
 \[ c_2(H^{1,0}) = c_2(H^{1,0}\oplus N^{1,0}) = c_2(i^\ast T^{1,0}V) . \]
 By naturality of the Chern classes,
 \[ c_2(i^\ast T^{1,0}V) = i^\ast c_2(T^{1,0}V) . \]
 Since $\dim_{\bC}V=3$ and $V$ is Stein, it holds that $H^4(V;\bC)=0$ (see~\cite{Serre1953}).  In particular, $c_2(T^{1,0}V)=0$.
\end{proof}

We conclude this article with two remarks about possible interpretations of the pseudohermitian invariant $\mI^\prime$.

\begin{remark}
 \label{rk:fefferman_graham}
 $\mI^\prime$ can be informally regarded as the ``prime analogue'' of the conformal invariant $\left|\nabla_A W_{BCEF}\right|^2$ discovered by Fefferman and Graham~\cite{FeffermanGraham2012}, in the same spirit as the $P^\prime$-operator is the ``primed analogue'' of the Paneitz operator.  Indeed, on a pseudohermitian manifold $(M^{2n+1},H,\theta)$, define
 \begin{multline}
  \label{eqn:mI}
  \mI = \nabla^\gamma\left(iS_{\gamma\bar\sigma\alpha\bar\beta}V^{\bar\sigma\alpha\bar\beta} - \frac{1}{2n}\nabla_\gamma\lv S_{\alpha\bar\beta\gamma\bar\sigma}\rv^2\right) - \mS^{\alpha\bar\beta}P_{(\alpha\bar\beta)_0} \\ + (n-2)\left(\frac{1}{2n(n-4)}\Delta_b\lv S_{\alpha\bar\beta\gamma\bar\sigma}\rv^2 + \lv V_{\alpha\bar\beta\gamma}\rv^2 - \frac{2}{n(n-4)}P\lv S_{\alpha\bar\beta\gamma\bar\sigma}\rv^2\right) .
 \end{multline}
 where $\mS_{\alpha\bar\beta}=S_{\alpha\bar\rho\gamma\bar\sigma}S^{\bar\rho}{}_{\bar\beta}{}^{\bar\sigma\gamma}-\frac{1}{n}\lv S_{\gamma\bar\sigma\delta\bar\rho}\rv^2h_{\alpha\bar\beta}$.  Set
 \[ K_{C\bar E} = \Omega_{\alpha\bar\beta C\bar F}\Omega^{\bar\beta\alpha\bar F}{}_{\bar E} \in \mE_{C\bar E}(-2) . \]
 It is straightforward to compute that
 \[ \frac{1}{n-3}D^C K_{C\bar E} - \frac{1}{2(n-4)}D_{\bar E}\left| S_{\alpha\bar\beta\gamma\bar\sigma}\right|^2 = \mI Z_{\bar E} \in \mE_{\bar E}(-2,-3) . \]
 In particular, it follows that~\eqref{eqn:mI} defines a CR invariant of weight $(-3,-3)$.  This invariant is defined via a tractor expression equivalent to the tractor expression giving the Fefferman--Graham invariant (cf.\ \cite{CapGover2003}).  Restricting to pseudo-Einstein metrics, one observes that $\mI=(n-2)\mI^\prime$ modulo divergences, which motivates the definition of $\mI^\prime$.
\end{remark}

\begin{remark}
 \label{rk:deser-schwimmer}
 Alexakis~\cite{Alexakis2012} proved that any local Riemannian invariant $I(g)$ for which $\int I(g)\dvol_g$ is a conformal invariant admits a decomposition
 \[ I(g) = cQ_g + \text{(local conformal invariant)} + \text{(divergence)} . \]
 Hirachi conjectured~\cite[p.\ 242]{Hirachi2013} that any local pseudohermitian invariant $I(\theta)$ for which $\int I(\theta)\,\theta\wedge d\theta^n$ is a \invariant/ should admit a similar decomposition in terms of a constant multiple of the $Q^\prime$-curvature, a local CR invariant, and a divergence.  It seems to us that the $\mI^\prime$-invariant, through the following two questions, provides a new insight into this conjecture.

 First, is there a five-dimensional pseudo-Einstein manifold for which the CR invariant one-form $X_\gamma$ is nonzero?  If so, then $\mI^\prime$ is not a local \invariant/, and thus provides a counterexample to Hirachi's conjecture.  If not, then Hirachi's conjecture seems correct, at least in dimension five and after modifying it to allow local \invariant/s.

 Second, how can one understand the transformation formula~\eqref{eqn:Iprime_transformation}?  Specifically, observe that the proof of Lemma~\ref{lem:c2_rep} shows that if $\xi=d\alpha$ for $\alpha=i\Omega_{\alpha\bar\beta}\theta\wedge\theta^\alpha\wedge\theta^{\bar\beta}$ and $\Omega_{\alpha\bar\beta}\in\mE_{(\alpha\beta)_0}$ --- that is, if $\alpha$ can be chosen to be an element of $F^{2,1}$ in the graded Rumin complex~\cite{GarfieldLee1998} --- then the map $\mP\ni v\mapsto\Real X^\gamma\nabla_\gamma v$ is formally self-adjoint on the space of CR pluriharmonic functions.  Thus, one might suspect that the transformation formula for $\mI^\prime$ is governed by a formally self-adjoint operator on CR pluriharmonic functions, a property shared by the $Q^\prime$-curvature.
\end{remark}
%